\documentclass[11pt]{amsart}
\usepackage[foot]{amsaddr}
\usepackage{amsfonts}
\usepackage[latin9]{inputenc}
\usepackage{amsmath}
\usepackage{amssymb}
\usepackage{breqn}
\usepackage{url}
\usepackage{graphicx}
\usepackage[pdfstartview=FitH]{hyperref}
\usepackage{amsthm}
\usepackage{hyperref}
\usepackage{url}
\usepackage{cleveref}
\usepackage{verbatim}
\usepackage{bbm}
\usepackage{enumitem}

\calclayout

\crefformat{section}{\S#2#1#3} 
\crefformat{subsection}{\S#2#1#3}
\crefformat{subsubsection}{\S#2#1#3}

\newcommand{\prefix}{\operatorname{prefix}}
\newcommand{\suffix}{\operatorname{suffix}}

\newcommand{\B}{\operatorname{\mathbb{E}}}

\newcommand{\Span}{\operatorname{span}}

\newcommand{\supp}{\operatorname{supp}}
\newcommand{\St}{\operatorname{St}}
\newcommand{\SL}{\operatorname{SL}}

\newcommand{\EL}{\operatorname{EL}}

\newcommand{\Prob}{\operatorname{Prob}}
\newcommand{\Isom}{\operatorname{Isom}}

\newcommand{\Cay}{\operatorname{Cay}}
\newcommand{\disp}{\operatorname{disp}}
\newcommand{\dist}{\operatorname{dist}}
\newcommand{\Core}{\operatorname{Core}}
\newcommand{\Borel}{\operatorname{Borel}}
\newcommand{\Line}{\operatorname{Line}}
\newcommand{\Poly}{\operatorname{Poly}}

\makeatletter
\renewcommand{\email}[2][]{%
  \ifx\emails\@empty\relax\else{\g@addto@macro\emails{,\space}}\fi%
  \@ifnotempty{#1}{\g@addto@macro\emails{\textrm{(#1)}\space}}%
  \g@addto@macro\emails{#2}%
}
\makeatother

\newcommand{\RN}[1]{%
  \textup{\uppercase\expandafter{\romannumeral#1}}%
}

\makeatletter
\begin{document}
\newtheorem{theorem}{Theorem}[section]
\newtheorem{lemma}[theorem]{Lemma}
\newtheorem{claim}[theorem]{Claim}
\newtheorem{proposition}[theorem]{Proposition}
\newtheorem{corollary}[theorem]{Corollary}
\theoremstyle{definition}
\newtheorem{definition}[theorem]{Definition}
\newtheorem{observation}[theorem]{Observation}
\newtheorem{example}[theorem]{Example}
\newtheorem{remark}[theorem]{Remark}

\title[]{Banach fixed point property for Steinberg groups over commutative rings}
\author{Izhar Oppenheim}
\address{Department of Mathematics, Ben-Gurion University of the Negev, Be'er Sheva 84105, Israel}
\email{izharo@bgu.ac.il}


\maketitle
\begin{abstract}
The main result of this paper is that all affine isometric actions of higher rank Steinberg groups over commutative rings on uniformly convex Banach spaces have a fixed point.  We consider Steinberg groups over classical root systems and our analysis covers almost all such Steinberg groups excluding a single rank 2 case.  

The proof of our main result stems from two independent results - a result regarding relative fixed point properties of root subgroups of Steinberg groups and a result regarding passing from relative fixed point properties to a (global) fixed point property.  The latter result is proven in the general setting of groups graded by root systems and provides a far reaching generalization of the work of Ershov,  Jaikin-Zapirain and Kassabov who proved a similar result regarding property (T) for such groups.

As an application of our main result, we give new constructions of super-expanders.  
\end{abstract} 

\section{Introduction}

Given a topological group $\Gamma$ and a Banach space $\B$,  we say that $\Gamma$ has property $(F_{\B})$ if every continuous affine isometric action of $\Gamma$ on $\B$ admits a fixed point.  For a class of Banach spaces $\mathcal{E}$,  we will say that $\Gamma$ has property $(F_{\mathcal{E}})$,  if $\Gamma$ has property $(F_{\B})$ for every $\B \in \mathcal{E}$.  

The most classical instantiation of property $(F_{\mathcal{E}})$ is property $(FH)$ which is property $(F_{\mathcal{E}})$ where $\mathcal{E}$ is the class of all real Hilbert spaces.  For a $\sigma$-compact,  locally compact group $\Gamma$,   Delorme-Guichardet Theorem states that property (FH) is equivalent to Kazhdan's property (T).  

In this work,  we will consider property $(F_{\mathcal{E}_{uc}})$,  where $\mathcal{E}_{uc}$ is the class of all the uniformly convex Banach spaces.  We recall that a Banach space $\B$ is called \textit{uniformly convex} if there is a function $\delta : (0,2] \rightarrow (0,1]$ that is called the \textit{modulus of convexity} such that for every two unit vectors $\xi,  \eta \in \B$ it holds that if $\Vert \xi - \eta \Vert \geq \varepsilon$,  then $\Vert \frac{\xi + \eta}{2} \Vert \leq 1 - \delta (\varepsilon)$.  For example,  Hilbert spaces are uniformly convex (due to the parallelogram equality) and for every $1 < p < \infty$,  every $L^p$ space is uniformly convex (due to Clarkson's inequalities). 

The study of property $(F_{\mathcal{E}_{uc}})$ was initiated in the seminal paper of Bader, Furman, Gelander and Monod \cite{BFGM},  in which they proved (among several other results) that higher rank algebraic groups have property $(F_{L^p})$ for every $L^p$ space with $1 < p < \infty$.  They also conjectured that higher rank algebraic group should have property $(F_{\mathcal{E}_{uc}})$.  This conjecture can be split into two cases depending on the local field.  For non-Archimedian local fields,  the conjecture of \cite{BFGM} was settled by the work of V. Lafforgue \cite{Laff2} and the subsequent work of Liao \cite{Liao} in which stronger results were proven,  i.e.,  in \cite{Laff2, Liao} it was proven that higher rank algebraic groups over non-Archimedian local fields have strong Banach property (T) which implies property $(F_{\mathcal{E}_{uc}})$.  In the Archimedian case,   the conjecture of \cite{BFGM}  was settled only recently: in \cite{Opp-SLZ},  the author made a breakthrough and proved that $\SL_n (\mathbb{R}),  n \geq 4$ and its lattices have property $(F_{\mathcal{E}_{uc}})$.  The technique of  \cite{Opp-SLZ} was then generalized by de Laat and de la Salle \cite{LaatSalle4} to yield a proof of property $(F_{\mathcal{E}_{uc}})$ for every real higher rank algebraic group.  

In this paper,  we turn our attention to the study of property $(F_{\mathcal{E}_{uc}})$ for Steinberg groups and elementary Chevalley groups over commutative rings.  For a classical (crystallographic) reduced,  irreducible root system $\Phi$ and a commutative ring $R$,  we denote $\mathbb{G}_{\Phi} (R)$ to be the simply-connected Chevalley group corresponding to $\Phi$ over $R$.  We further denote $\EL_{\Phi} (R)$ to be elementary Chevalley group corresponding to $\Phi$ over $R$, i.e.,  the subgroup of $\mathbb{G}_{\Phi} (R)$ generated by the root subgroups with respect to the standard torus (we refer the reader to \cite{EJZK} and references therein for more detailed definitions).  The Steinberg group $\St_{\Phi} (R)$ is a group extension of $\EL_{\Phi} (R)$ that ``forgets'' all the relations of $\EL_{\Phi} (R)$ apart from the relations of $R$ (as an Abelian group) in each root subgroup and the Chevalley commutator formula  (a more explicit definition is given in \cref{Steinberg groups over rings subsec} below).  For example,  in the case where $\Phi = A_n$,  the groups discussed above are $\mathbb{G}_{A_{n}} (R) = \SL_{n+1} (R)$, $\EL_{A_n} (R) = \EL_{n+1} (R)$ and $\St_{A_n} (R) = \St_{n+1} (R)$.  With these notations,  our main result is:
\begin{theorem}
\label{main thm intro}
Let $\Phi$ be a classical reduced,  irreducible root system of rank $\geq 2$ such that $\Phi \neq C_2$.  For every finitely generated, commutative (unital) ring $R$,  the groups  $\St_{\Phi} (R)$ and $\EL_{\Phi} (R)$ have property $(F_{\mathcal{E}_{uc}})$.   
\end{theorem}

Historically,  establishing property (T) (or,  equivalently,  property (FH)) for $\St_{\Phi} (R)$ and $\EL_{\Phi} (R)$ was more challenging than establishing property (T) for higher rank algebraic groups.  Indeed,  property (T) for higher rank algebraic groups was proved a few years after Kazhdan defined property (T).  In contrast,  property (T) for  $\EL_n (R),  n \geq 3$ was only established much more recently in the works of Shalom \cite{Shalom2} and Vaserstein \cite{Vaser} using a bounded generation machinery of Shalom (partial results were proven in \cite{Shalom1, KasNik}).  For  $\St_n  (R), n \geq 3$,  Property (T) was proven by Ershov and Jaikin-Zapirain \cite{ErshovJZ} using the machinery of angles between subspaces (which is very different from the bounded generation machinery of Shalom).  This angle machinery was later generalized by Ershov,  Jaikin-Zapirain and Kassabov \cite{EJZK} to prove property (T) for $\St_{\Phi} (R)$ and $\EL_{\Phi} (R)$ for every classical reduced,  irreducible root system $\Phi$.  

The proof of our main result has two components: relative property $(F_{\mathcal{E}_{uc}})$ and synthesis of property $(F_{\mathcal{E}_{uc}})$ (the term ``synthesis'' is due to Mimura \cite{Mimura4}).

\subsection{Relative property $(F_{\mathcal{E}_{uc}})$}

For a topological group $\Gamma$, a subgroup $K < \Gamma$ and a Banach space $\B$,  we say that the pair $(\Gamma,K)$ has \textit{relative property $(F_\B)$} if every continuous, affine, isometric action of $\Gamma$ on $\B$ has a $K$-fixed point.  For a class of Banach spaces,  we will say that the pair $(\Gamma,K)$ has \textit{relative property $(F_{\mathcal{E}})$} if it has relative property $(F_\B)$ for every $\B \in \mathcal{E}$.  

Our main result regarding relative property $(F_{\mathcal{E}_{uc}})$ is the following:
\begin{theorem}
\label{rel. f.p. thm - intro}
Let $\Phi \neq C_2$ be a classical,  reduced,  irreducible root system of rank $\geq 2$ and $R$ a finitely generated commutative (unital) ring.  For $\alpha \in \Phi$,  we denote $K_\alpha$ to be the root subgroup of $\St_{\Phi} (R)$.  Then for every $\alpha \in \Phi$,  the pair $(\St_{\Phi} (R),  K_{\alpha})$ has relative property $(F_{\mathcal{E}_{uc}})$.
\end{theorem}

Most of our work toward proving Theorem \ref{rel. f.p. thm - intro} is proving the relative fix point property in the case where $\Phi = A_2$:
\begin{theorem}
\label{A2 relative f.p. thm intro}
For every $m \in \mathbb{N}$ and every $\alpha \in A_2$,  the pair $(\St_{A_2} (\mathbb{Z} [t_1,...,t_m]),  K_{\alpha})$ has relative property $(F_{\mathcal{E}_{uc}})$.
\end{theorem}

Our approach for proving Theorem \ref{A2 relative f.p. thm intro} is to show that for every $\B \in \mathcal{E}_{uc}$ and every $\xi \in \B$,  there is $\widetilde{\xi}$ in the closure of the convex hull of the  $\St_{A_2} (\mathbb{Z} [t_1,...,t_m])$-orbit of $\xi$ such that $\widetilde{\xi}$ is fixed by $K_{\alpha}$.  This is done by defining a sequence of finitely supported averages of vectors in the orbit of $\xi$ and showing that this sequence converges to $\widetilde{\xi}$ as above.  The method of proof is very similar to that in \cite{Opp-SLZ} and relies of bounding the differences between averages in the Heisenberg group.  However,  there is a difficulty that arises from the fact that here we work with the Heisenberg group with entries in $\mathbb{Z} [t_1,...,t_m]$,  while in \cite{Opp-SLZ},  the analysis was for the Heisenberg group with entries in $\mathbb{Z}$.  Resolving this difficulty requires a non-trivial improvement of the methods developed in \cite{Opp-SLZ}.

Passing from Theorem \ref{A2 relative f.p. thm intro} to Theorem \ref{rel. f.p. thm - intro} when $\Phi$ is simply laced or $\Phi = F_4$ is straight-forward: in those cases every root is an $A_2$-subsystem and thus for every $\alpha \in \Phi$ there is $\St_{A_2} (R) < \St_\Phi (R)$ such that $K_\alpha < \St_{A_2} (R)$.  The cases where $\Phi = B_n,  C_n,  n \geq 3$ or $\Phi = G_2$ require some additional analysis and are proven via a bounded generation argument.

\subsection{Synthesis of property $(F_{\mathcal{E}_{uc}})$}

After establishing relative fixed point properties for the pairs $(\St_{\Phi} (R),  K_\alpha),  \alpha \in \Phi$,  we need to ``synthesize'' the relative fixed point property to a global fixed point property of the entire group $\St_{\Phi} (R)$.  In doing so,  we cannot use bounded generation arguments \`{a} la Shalom \cite{Shalom2}, since the bounded generation assumptions do not hold for Steinberg groups.  In \cite{EJZK},  Ershov,  Jaikin-Zapirain and Kassabov gave a synthesis argument for property (T) for Steinberg groups that relied on the notion of angle between subspaces.  Although there are some results of the author regarding generalizing the angle machinery to the Banach setting (see \cite{OppRobust, OppVanBan, OppAngle}),  it seems that generalizing the approach of \cite{EJZK} to all uniformly convex Banach spaces is hard to implement.  In \cite{Mimura4},  Mimura gave a ``soft'' synthesis argument for the groups $\St_{A_n} (R),  n \geq 2$ that does not use bounded generation or angle arguments and that can be applied to the class of all uniformly convex Banach spaces.  However,  the argument in \cite{Mimura4} was limited to $\St_{A_n} (R)$.   

Our work on synthesis of fixed point properties can be seen as a vast generalization of the basic idea of Mimura \cite{Mimura4}.  First,  we develop a very general machinery for synthesis of fixed point properties that is interesting in its own right:
\begin{theorem}
\label{general reduction thm intro - non-directed graph}
Let $\Gamma$ be a finitely generated group with a finite Abelianization and $\mathcal{E}$ a class of uniformly convex Banach spaces such that either is closed under passing to ultraproducts or $\mathcal{E} = \mathcal{E}_{uc}$.  Also,  let $\mathcal{G}$ be a (non-directed) connected graph with a vertex set $V$.  Assume that for every $v \in V$, there are subgroups $N^{v},  H_{+}^{v},  H_{-}^{v} < \Gamma$ such that the following holds:
\begin{enumerate}
\item For every $v \in V$,  $N^{v}$ normalizes $H_{+}^{v}$ and $H_{-}^{v}$.
\item For every $v \in V$,  $\langle H_{+}^{v},  H_{-}^{v} \rangle = \Gamma$.
\item If $u,v \in V$ such that $u \sim v$,  then
$H_{\pm}^{u} < \langle H_{\pm}^{v}, N^{v} \rangle ,$
and
$H_{\pm}^{v} < \langle H_{\pm}^{u}, N^{u} \rangle .$
\item It holds that
$$\langle  H_{+}^{v}, N^{v}  : v \in V \rangle = \Gamma.$$
\item For every $v \in V$,  the pairs $(\Gamma,  H_{+}^{v}),  (\Gamma,  H_{-}^{v})$ have relative property $(F_{\mathcal{E}})$.
\end{enumerate} 
Then $\Gamma$ has property $(F_{\mathcal{E}})$. 
\end{theorem}

Using this machinery, we derive a far reaching generalization of the main result of \cite{EJZK} (\cite[Theorem 1.2]{EJZK}):
\begin{theorem}
\label{uc synthesis thm for Steinberg - intro}
Let $\Phi$ be a classical reduced irreducible root system of rank $\geq 2$,  $R$ a commutative ring,  $\St_{\Phi} (R)$ the Steinberg group and $\lbrace K_\alpha : \alpha \in \Phi \rbrace$ the root subgroups of $\St_{\Phi} (R)$.   

If for every $\alpha \in \Phi$,  the pair $(\St_{\Phi} (R) , K_\alpha)$ has relative property $(F_{\mathcal{E}_{uc}})$, then $\St_{\Phi} (R)$ has property $(F_{\mathcal{E}_{uc}})$.
\end{theorem}

We note that Theorem \ref{uc synthesis thm for Steinberg - intro} is actually a special case of a more general result: In \cite{EJZK},  Ershov,  Jaikin-Zapirain and Kassabov defined the general notion of a group that is strongly graded over a root system and showed that Steinberg groups $\St_{\Phi} (R)$ are a special case of this notion.  They then proved a synthesis result for property (T) for groups strongly graded over a root systems.  In Theorem \ref{synthesis thm for groups graded by root systems} below,  we generalize their result and prove a synthesis result in the context of uniformly convex Banach spaces for groups that are strongly graded over root systems.



\subsection{Super-expanders}

As an application of Theorem \ref{main thm intro},  we derive new constructions of super-expanders (see exact definition in \cref{super-expanders section} below):
\begin{theorem}
\label{super-exp thm intro}
Let $n \geq 3,  m \in \mathbb{N}$ and let $\lbrace R_i \rbrace_{i \in \mathbb{N}}$ be a sequence of finite commutative (unital) rings such that for each $i$,  $R_i$ is generated by $p_0^{(i)} =1,  p_1^{(i)},...,p_m^{(i)} \in R_i$ and $\vert R_i \vert \rightarrow \infty$.  Also, let $\Phi \neq C_2$ be a classical reduced irreducible root system of rank $\geq 2$.  Denote  $\phi_i : \St_{\Phi} (\mathbb{Z} [t_1,...,t_m]) \rightarrow \EL_\Phi (R_i)$ be the epimorphisms induced by the ring epimorphism $\mathbb{Z} [t_1,...,t_m] \rightarrow R_i,  1 \mapsto r_0^{(i)},  t_j \mapsto r_j^{(i)},  \forall 1 \leq j \leq m$. 

For a finite generating set $S$ of  $\St_\Phi (\mathbb{Z} [t_1,...,t_m])$ it holds that the family of Cayley graphs of $\lbrace (\EL_{\Phi} (R_i),  \phi_i (S)) \rbrace_{i \in \mathbb{N}}$ is a super-expander family.
\end{theorem}

\paragraph{\textbf{Structure of this paper}}
This paper is organized as follows: In \cref{prelim sec}, we cover some needed preliminaries.  In \cref{Averaging operations sec},  we prove some bounds on the difference of averaging operations for the Heisenberg group that are needed for our relative fixed point result.  In \cref{Word norm growth of unipotents sec},  we prove a bound on the word norm growth in the $A_2$ Steinberg group.  In \cref{Relative fixed point property sec},  we prove our result regarding relative fixed point properties of root subgroups in Steinberg group (Theorem \ref{rel. f.p. thm - intro} above).   In \cref{General synthesis section},  we prove Theorem \ref{general reduction thm intro - non-directed graph} (and also some more general versions).  In  \cref{Synthesis of the fixed point property for groups graded by root systems sec},  we prove a general synthesis result for groups graded by root systems and derive Theorem \ref{uc synthesis thm for Steinberg - intro} as a special case.  In \cref{Banach fixed point properties for Steinberg groups and elementary Chevalley groups sec},  we prove our main result (Theorem \ref{main thm intro} above).  Last,  in \cref{super-expanders section},  we show that our main result yields new construction of super-expanders (Theorem \ref{super-exp thm intro}).

\section{Preliminaries}
\label{prelim sec}

\subsection{Basic notation}

\paragraph*{\textbf{$\lesssim$ notation}}
We will use the $\lesssim$ notation as follows: For to expressions $X,Y$,  we will write $X \lesssim Y$ if there is a universal constant $C$ such that $X \leq CY$.  We will write $X \lesssim_B Y$ if there is a constant $C = C (B)$ such that $X \leq C Y$. 

\paragraph*{\textbf{Monomials notation}} For a monomial $\underline{t} = t_1^{n_1} ... t_{m}^{n_m}$ where $n_i \in \mathbb{N} \cup \lbrace 0 \rbrace$,  we will denote $\deg (\underline{t} ) = n_1 +...+n_m$.  We will also use the convention $t_1^0 ... t_m^0 = 1$.

\paragraph*{\textbf{Graph notation}} We will denote (non-directed) graphs by $\mathcal{G}$ and directed graphs by $\vec{\mathcal{G}}$.  For two vertices $u,v$ in $\mathcal{G}$, we will denote $u \sim v$  if there is an edge between $u$ and $v$ (when there is a chance of ambiguity regarding $\mathcal{G}$,  we will denote $u \sim^{\mathcal{G}} v$).  For two vertices $u,v$ in $\vec{\mathcal{G}}$, we will denote $u \rightarrow v$  if there is a directed edge from $u$ to $v$.

\subsection{Linear representations and affine isometric actions of a group}

Let $\Gamma$ be a discrete group and $\B$ a Banach space.  An \textit{isometric representation} of $\Gamma$ on $\B$ is a homomorphism $\pi : \Gamma \rightarrow O (\B)$,  where $O (\B)$ denotes the group of all isometric linear invertible maps from $\B$ to itself.  An affine isometric action of $\Gamma$ on $\B$ is a homomorphism $\rho : G \rightarrow \Isom_{aff} (\B)$, where $\Isom_{aff} (\B)$ is the group of bijective affine isometries from $\B$ to itself.  We recall that $\rho$ is of the form 
$$\rho (g) \xi = \pi  (g) \xi + c(g),  \forall \xi \in \B$$
where $\pi : \Gamma \rightarrow O(\B)$ is an isometric linear representation that is called the \textit{linear part of $\rho$} and $c: \Gamma \rightarrow \B$ is a $1$-cocycle into $\pi$, i.e., for every $g,h \in \Gamma$, 
$$c (gh) = c (g) + \pi (g) c (h).$$

Denote $\mathbb{R} [\Gamma]$ to be the group ring of $\Gamma$,  i.e.,  the ring of functions $f: \Gamma \rightarrow \mathbb{R}$ with finite support where multiplication is via convolution.  Equivalently,  an element in $\mathbb{R} [\Gamma]$ can be represented of as a formal sum $\sum_{g \in F} a_g g$,  where $F \subseteq \Gamma$ is a finite set and $\lbrace a_g \rbrace_{g \in F} \subseteq \mathbb{R}$.  In this convention,  the multiplication is defined as
$$\left( \sum_{g \in F_1} a_g g \right) \left( \sum_{g ' \in F_2} b_{g '} g ' \right) = \sum_{g \in F_1,  g' \in F_2} a_g b_{g '} g g' .$$
Below,  we will use both conventions of $\mathbb{R} [\Gamma]$ according to convenience.

For an affine isometric action $(\rho, \B)$ and $f \in \mathbb{R} [\Gamma]$ define 
$$\rho (f) \xi = \sum_{g \in \Gamma} f (g) \rho (g) \xi,  \forall \xi \in \B$$
(note that the sum above is finite and thus well defined).  We note that $\rho (f)$ need not be an isometry,  but it is always an affine map, i.e.,  for every $f \in \mathbb{R} [\Gamma]$ there is a bounded linear map $T: \B \rightarrow \B$ and a vector $\xi_0 \in \B$ such that for every $\xi \in \B$,
$$\rho (f) \xi = T \xi + \xi_0.$$
We also note that since every isometric linear representation $\pi : \Gamma \rightarrow O(\B)$ is an affine isometric action, this definition also extends to $\pi (f)$ and in that case $\pi (f)$ is always a bounded linear map.   

We define $\Prob_c (\Gamma) \subseteq \mathbb{R} [\Gamma]$ to be finitely supported probability measures on $\Gamma$,  i.e.,  $f \in \Prob_c (\Gamma)$ if it is finitely supported,  $f (g) \geq 0,  \forall g$ and $\sum_{g} f(g) =1$.  Equivalently,  $\Prob_c (\Gamma)$ is the set of all the formal sums $\sum_{g \in F} a_g g$,  where $F \subseteq \Gamma$ is a finite set,  $\lbrace a_g \rbrace_{g \in F} \subseteq [0,1]$ and $\sum_{g \in F} a_g =1$.  

\begin{observation}
Let $(\rho, \B)$ be an affine isometric action of $\Gamma$.  For every $g \in \Gamma$,  every $\xi_1,...,\xi_n \in \B$ and every $a_1,...,a_n \in [0,1]$ such that $\sum_i a_i =1$,  it holds that
$$\rho (g) \left( \sum_{i=1}^n a_i \xi_i \right) = \sum_{i=1}^n a_i \rho (g) \xi_i.$$
It follows that for every $f_1,  f_2 \in \Prob_c (\Gamma)$,  it holds that $\rho (f_1) \rho (f_2) = \rho (f_1 f_2)$. 
\end{observation}

\begin{claim}
\label{difference to linear claim}
Let $(\rho, \B)$ be an affine isometric action of $\Gamma$ with a linear part $\pi$.  Then for every $f  \in \mathbb{R}[\Gamma]$ and every $\xi_1, \xi_2 \in \B$ it holds that 
$$\rho (f) \xi_1 - \rho (f) \xi_2 = \pi (f) (\xi_1 - \xi_2).$$
In particular, for $f \in \Prob_c (\Gamma)$,  
$$\Vert \rho (f) \xi_1 - \rho (f) \xi_2 \Vert \leq \lVert \xi_1 - \xi_2 \Vert.$$
\end{claim}

\begin{proof}
We note that by the definition of affine map it holds for every $g \in \Gamma$ that 
$$\rho (g)  \xi_1 - \rho (g) \xi_2 = \pi (g) (\xi_1 - \xi_2).$$

Thus,  for every $f \in \mathbb{R}[\Gamma]$ it holds that
\begin{align*}
& \rho (f) \xi_1 - \rho (f) \xi_2 = \sum_{g \in \Gamma} f(g) \rho (g) \xi_1 - \sum_{g \in \Gamma} f(g) \rho (g) \xi_2 = \\
& \sum_{g \in \Gamma} f(g) ( \rho (g) \xi_1 - \rho (g) \xi_2) = \sum_{g \in \Gamma} f(g) \pi (g) (\xi_1 - \xi_2) = \pi (f) (\xi_1 - \xi_2),
\end{align*} 
as needed.
\end{proof}

\begin{proposition}
\label{ergodic diff prop}
Let $(\rho, \B)$ be an affine isometric action of $\Gamma$.  For every integers $0 \leq N \leq n,  n >0$,  every $g \in \Gamma$ and every $\xi \in \B$ it holds that 
$$\left\Vert \rho \left( \frac{1}{n} \sum_{i=0}^{n-1} g^i \right) \xi -  \rho (g^N) \rho \left( \frac{1}{n} \sum_{i=0}^{n-1} g^i \right) \xi \right\Vert \leq \frac{N}{n} \Vert \xi -  \rho (g^{n} ) \xi  \Vert.$$
\end{proposition}

\begin{proof}
We note that 
$$ \rho (g^N) \rho \left( \frac{1}{n} \sum_{i=0}^{n-1} g^i \right) = \rho \left( \frac{1}{n} \sum_{i=N}^{n+N-1} g^i \right).$$
Thus,
\begin{align*}
\left\Vert \rho \left( \frac{1}{n} \sum_{i=0}^{n-1} g^i \right) \xi -  \rho (g^N) \rho \left( \frac{1}{n} \sum_{i=0}^{n-1} g^i \right) \xi \right\Vert = 
\left\Vert   \frac{1}{n} \sum_{i=0}^{n-1} \rho (g^i ) \xi -  \frac{1}{n} \sum_{i=N}^{n+N-1} \rho (g^i) \xi \right\Vert = \\
\left\Vert   \frac{1}{n} \sum_{i=0}^{N-1} \left( \rho (g^i ) \xi -  \rho (g^{i+n} ) \right) \xi \right\Vert  \leq \\
\frac{1}{n} \sum_{i=0}^{N-1}  \left\Vert  \rho (g^i ) \xi -  \rho (g^{i}) \rho (g^n)  \xi \right\Vert  \leq^{\text{Claim } \ref{difference to linear claim}} 
 \frac{N}{n} \Vert \xi -  \rho (g^{n} ) \xi  \Vert,
\end{align*}
as needed.
\end{proof}

Let $S$ be a symmetric generating set of $\Gamma$ (we do not assume that $S$ is necessarily finite).  For every $g \in \Gamma$,  denote $\vert g \vert_S$ to be the word norm with respect to $S$,  i.e.,  $\vert g \vert_S$ is the distance between $g$ and $e$ in the Cayley graph $\Cay (\Gamma,  S)$.  For $f \in  \mathbb{R}[\Gamma]$,  we denote 
$$\vert f \vert_S = \max \lbrace \vert g \vert_S : f(g) \neq 0 \rbrace.$$

\begin{observation}
Observe that for  $f_1, f_2 \in  \mathbb{R}[\Gamma]$ it holds that $\vert f_1 f_2 \vert_S \leq \vert f_1 \vert_S + \vert f_2 \vert_S$.
\end{observation}

\begin{proposition}
\label{distance of rho (g) prop}
Let $\B$ a Banach space,  $\Gamma$ a group,  $S$ a symmetric generating set of $\Gamma$ and $\rho$ an affine isometric action of $\Gamma$ .  Assume that 
$$\sup_{s \in S} \Vert \rho (s) 0 \Vert   < \infty.$$
Then for every $\xi \in \B$ and every $g \in \Gamma$ it holds that 
$$\Vert \xi - \rho (g) \xi \Vert \leq (\sup_{s \in S} \Vert  \rho (s) 0 \Vert  + 2 \Vert \xi \Vert) \vert g \vert_S.$$

Also,  for every $f, f_1,  f_2 \in \Prob_c (\Gamma)$ and every $\xi \in \B$ it holds that 
$$\Vert \xi - \rho (f) \xi \Vert \lesssim_{\Vert \xi \Vert , S} \vert f \vert_S,$$
and
$$\Vert \rho (f_1) \xi - \rho (f_2) \xi \Vert \lesssim_{\Vert \xi \Vert , S} \vert f_1 \vert_S +  \vert f_2 \vert_S.$$
\end{proposition}

\begin{proof}
Let $\pi$ denote the linear part of $\rho$ and $c$ denote the $1$-cocycle of $\rho$.  We note that 
$$\sup_{s \in S} \Vert c (s) \Vert = \sup_{s \in S} \Vert  \rho (s) 0 \Vert < \infty.$$

Fix $\xi \in \B$ and denote $C = 2 \Vert \xi \Vert + \sup_{s \in S} \Vert c (s) \Vert$.  Note that 
$$\sup_{s \in S} \Vert \xi - \rho (s) \xi \Vert = \sup_{s \in S} \Vert \xi - \pi (s) \xi - c (s) \Vert \leq 2 \Vert \xi \Vert + \sup_{s \in S} \Vert c (s) \Vert =C.$$

We will show that for every $g \in \Gamma$,  
$$\Vert \xi - \rho (g) \xi \Vert \leq C \vert g \vert_S.$$
The proof is by induction on $\vert g \vert_S$.  If $\vert g \vert_S =0$, then $g =e$ and $\Vert \xi - \rho (e) \xi \Vert =0$ as needed.  Assume that the inequality holds for every $g ' \in \Gamma$ with $\vert g' \vert_S =n$ and let $g \in \Gamma$ with $\vert g \vert_S = n+1$.  There are $s_1 ... s_{n+1} \in S$ such that $g = s_1 ... s_{n+1}$.   Denote $g ' = s_2 ... s_{n+1}$.  Then 
\begin{align*}
& \Vert \xi - \rho (g) \xi \Vert \leq \Vert \xi - \rho (s_1) \xi \Vert + \Vert \rho (s_1) \xi - \rho (s_1 ... s_{n+1}) \xi \Vert \leq \\
& C + \Vert \rho (s_1) \xi - \rho (s_1) \rho (g') \xi \Vert = C + \Vert \xi - \rho (g') \xi \Vert \leq C + C n = C (n+1),
\end{align*}
as needed.  

For $f \in \Prob_c (\Gamma)$,  it holds that 
\begin{align*}
& \Vert \xi - \rho (f) \xi \Vert = \left\Vert \xi - \sum_{g} f(g) \rho (g) \xi \right\Vert = \left\Vert  \sum_{g} f(g)(\xi - \rho (g) \xi ) \right\Vert \leq  \\
& \sum_{g} f(g) \Vert \xi - \rho (g) \xi \Vert \leq \max_{g, f(g) \neq 0} \Vert \xi - \rho (g) \xi \Vert \lesssim_{\Vert \xi \Vert , S} \vert f \vert_S.
\end{align*}
It follows that for $f_1,  f_2 \in \Prob_c (\Gamma)$ it holds that
$$ \Vert \rho (f_1) \xi - \rho (f_2) \xi \Vert \leq  \Vert \rho (f_1) \xi  - \xi \Vert + \Vert \xi - \rho (f_2) \xi \Vert \lesssim_{\Vert \xi \Vert , S} \vert f_1 \vert_S + \vert f_2 \vert_S.$$
\end{proof}

\subsection{Bounded generation and Banach fixed point properties}

Let $\Gamma$ be a group with subgroups $H_1,...,H_k, \Gamma'$.  We say that $H_1,...,H_k$ boundedly generate  $\Gamma ' $ if $\Gamma ' \subseteq \langle H_1,...,H_k \rangle$ and $\sup_{g \in \Gamma '} \vert g \vert_{H_1 \cup ...  \cup H_k} < \infty$.  Shalom \cite{Shalom1},  observed that bounded generation and relative property (T) implies global property (T).  His argument is very simple and requires very little adaptation to the setting of uniformly convex Banach spaces (this result is well-known and the proof is given below for completeness):
\begin{lemma}[The bounded generation Lemma]
\label{bounded generation lemma}
Let $\Gamma$ be a group with subgroups $H_1,...,H_k,  \Gamma '$ and $\B$ be a uniformly convex Banach space.  Assume that $(\Gamma, H_1),..., (\Gamma, H_k)$ have relative property $(F_{\B})$ and that $\Gamma ' \subseteq \langle H_1,...,H_k \rangle$.   If $H_1,...,H_k$ boundedly generate $\Gamma '$,  then $(\Gamma, \Gamma ')$ has relative property $(F_{\B})$. 
\end{lemma}

\begin{proof}
Let $\rho$ be an affine isometric action of $\Gamma$ on $\B$.   Denote $S = H_1 \cup ... \cup H_k$.  By the assumption that  $(\Gamma, H_1),..., (\Gamma, H_k)$ have relative property $(F_{\B})$ it follows that for every $\xi \in \B$,  
$$\sup_{s \in S} \Vert \xi - \rho (s) \xi \Vert < \infty.$$
Thus, by the bounded generation assumption and Proposition \ref{distance of rho (g) prop} it follows that for every $\xi \in \B$, 
$$\sup_{g \in \Gamma '} \Vert \xi - \rho (g) \xi \Vert < \infty,$$
i.e.,  the $\Gamma '$ orbit of every $\xi$ is bounded.  It follows from Ryll-Nardzewski Theorem or a circumcenter argument that $\Gamma '$ has property $(F_\B)$.  
\end{proof}

\subsection{Root systems}

Here,  we will give the basic definitions and results given in  \cite{EJZK}  regarding root systems.  
\begin{definition}[Root system]
Let $E$ be a finite dimensional real vector space.  A finite non-empty set $\Phi \subset E$ is called a \textit{root system} if 
\begin{itemize}
\item $\Phi$ spans $E$.
\item $\Phi$ does not contain $0$.
\item $\Phi$ is closed under inversion,  i.e., $\alpha \in \Phi \Rightarrow (- \alpha) \in \Phi$. 
\end{itemize}
The vectors in $\Phi$ are called the \textit{roots of $\Phi$} and the dimension of $E$ is called the \textit{rank of $\Phi$}.
\end{definition}

\begin{definition}
Let $\Phi$ be a root system in $E$.  
\begin{itemize}
\item $\Phi$ is called \textit{reduced} if every $1$-dimensional subspace of $E$ contains at most two elements in $\Phi$.
\item $\Phi$ is called \textit{irreducible} if there are no disjoint non-empty sets $\Phi_1, \Phi_2 \subset \Phi$ such that $\Phi = \Phi_1 \cup \Phi_2$ and $\Span (\Phi_1) \cap \Span (\Phi_2) = \lbrace 0 \rbrace$.
\item A non-empty set $\Psi \subseteq \Phi$ is called a \textit{root subsystem of $\Phi$} if $\Phi \cap \Span (\Psi) = \Psi$. 
\item $\Phi$ is called \textit{regular} if every root of $\Phi$ is contained in an irreducible root subsystem of rank $2$.
\end{itemize}
\end{definition}

We note that these definitions generalize the definition of classical root systems:
\begin{definition}[Classical root system]
Let $\Phi$ be a root system in a space $E$.  $\Phi$ is called a \textit{classical root system} if there exists an inner-product $(.,.)$ on $E$ such that for every $\alpha,  \beta \in \Phi$ it holds that 
$$\frac{2(\alpha,  \beta)}{(\beta, \beta)} \in \mathbb{Z} \text{ and } \alpha - \frac{2(\alpha,  \beta)}{(\beta, \beta)} \beta \in \Phi.$$
\end{definition}

We will give only the basic definitions and facts regarding classical root systems - a much more complete account can be found in \cite{EJZK} and references therein.  It is well-known that there is a classification of classical reduced irreducible root systems.  Namely,  every reduce irreducible classical root system is isomorphic to one of the following: $A_n,  B_n (n \geq 3),  C_n (n \geq 2),  D_n (n \geq 4),  F_4,  E_6,  E_7,  E_8,  G_2$.  A classical reduced irreducible root system $\Phi$ is called \textit{simply laced} if all the roots of $\Phi$ have the same length/norm.  We also note that all the classical reduced irreducible root systems are regular.


\begin{definition}
Let $\Phi$ be a root system in $E$.  A linear functional $f: E \rightarrow \mathbb{R}$ is said to be \textit{in general position with respect to $\Phi$} if
\begin{itemize}
\item $f(\alpha) \neq 0$ for every $\alpha \in \Phi$.
\item $f(\alpha) \neq f (\beta)$ for every $\alpha, \beta \in \Phi, \alpha \neq \beta$.
\end{itemize}
We denote $\mathfrak{F} (\Phi)$ to be the set of all the linear functionals that are general position with respect to $\Phi$.
\end{definition}

\begin{definition}[Borel set,  boundary, core]
Let $\Phi$ be a root system in $E$ and $f \in \mathfrak{F} (\Phi)$.  
\begin{itemize}
\item The \textit{Borel set of $f$} is the set 
$$\Phi_f = \lbrace \alpha \in \Phi : f (\alpha) >0 \rbrace.$$
(note that it can be that $f \neq f'$ and $\Phi_f = \Phi_{f '}$).
\item We denote $\Borel (\Phi)$ to be the set of all the Borel sets of $\Phi$ and note that this set is closed under inversion,  i.e.,  $\Phi_1 \in \Borel (\Phi)  \Rightarrow - \Phi_1 \in \Borel (\Phi)$.
\item For $\Phi_1,  \Phi_2 \in \Borel (\Phi)$,  $\Phi_1$ and $\Phi_2$ are called \textit{co-minimal} if $\Phi_1 \cap \Phi_{2}$ spans a $1$-dimensional space. 
\item For $\Phi_1,  \Phi_2 \in \Borel (\Phi)$,  $\Phi_1$ and $\Phi_2$ are called \textit{co-maximal} if $\Phi_1$ and  $- \Phi_{2}$ are co-minimal.
\item The \textit{boundary of the Borel set $\Phi_1$} is the set 
$$\partial  \Phi_1 = \lbrace \alpha \in \Phi_1 : \exists \Phi_2 \in \Borel (\Phi) \text{ such that } \Phi_{1} \text{ and } \Phi_{2} \text{ are co-minimal and } \alpha \in  \Phi_{1} \cap \Phi_{2} \rbrace.$$
\item The core of the Borel set $\Phi_2$ is the set 
$$\Core (\Phi_1) = \Phi_1 \setminus \partial  \Phi_1.$$
\end{itemize}
\end{definition}

Below,  we will need the following facts from \cite{EJZK}:
\begin{lemma}
\label{facts about root sys lemma}
Let $\Phi$ be a root system in $E$.
\begin{enumerate}
\item \cite[Lemma 4.4]{EJZK}  Let $\Phi_1,  \Phi_2 \in \Borel (\Phi)$ such that $\Phi_1$ and $\Phi_2$ are not co-maximal. Then there exists $\Phi_3 \in \Borel (\Phi)$ such that $\Phi_1 \cap \Phi_2 \subsetneqq \Phi_1 \cap \Phi_3$ and $\Phi_1 \cap \Phi_2 \subsetneqq \Phi_2 \cap \Phi_3$.
\item \cite[Lemma 4.6]{EJZK} Let $\Phi_1 \in \Borel (\Phi)$ and $\alpha,  \beta \in \Phi_1$ such that $\alpha,  \beta$ are linearly independent.  For every $a, b\in (0, \infty)$,  if $a \alpha + b \beta \in \Phi$,  then $a \alpha + b \beta \in \Core (\Phi_1)$.
\item \cite[Lemma 4.8]{EJZK}  Assume that $\Phi$ is regular,  then for every $\alpha \in \Phi$,  there is  $\Phi_1 \in \Borel (\Phi)$ such that $\alpha \in \Core (\Phi_1)$. 
\end{enumerate}
\end{lemma}

Let $\Phi$ be a root system in $E$.  We define $\mathcal{G}_{\text{co-max}}$ to be the graph with the vertex set $V (\mathcal{G}_{\text{co-max}}) = \Borel  (\Phi) $ such that for
$\Phi_1,  \Phi_2 \in \Borel (\Phi),  \Phi_1 \neq \Phi_2$ it holds that $\Phi_1 \sim \Phi_2$ if and only if $\Phi_1$ and $\Phi_2$ are co-maximal.

\begin{lemma}
\label{co-max graph is connected lemma}
The graph $\mathcal{G}_{\text{co-max}}$ is connected. 
\end{lemma}

The proof of this lemma is very similar to the proof given for the connectedness of the small Weyl graph in \cite{EJZK} and we give the proof for completeness.

\begin{proof}
For $\alpha \in \Phi$,  we denote $\mathbb{R} \alpha$ to be the $1$-dimensional subspace spanned by $\alpha$.

Let $\Phi_1,  \Phi_2 \in \Borel (\Phi)$,  $\Phi_1 \neq \Phi_2$.  We will show that there is a path in $\mathcal{G}_{\text{co-max}}$ connecting $\Phi_1$ and $\Phi_2$ by a downward induction on $\vert \lbrace \mathbb{R}  \alpha : \alpha \in \Phi_1 \cap \Phi_2 \rbrace \vert$.  We note that $\vert \lbrace \mathbb{R} \alpha : \alpha \in \Phi_1 \cap \Phi_2 \rbrace \vert$ is maximal if and only if
$$\vert \lbrace \mathbb{R} \alpha : \alpha \in \Phi_1 \cap \Phi_2 \rbrace \vert = \vert \lbrace \mathbb{R} \alpha : \alpha \in \Phi_1 \rbrace \vert -1 = \vert \lbrace \mathbb{R} \alpha : \alpha \in \Phi_2 \rbrace \vert -1$$
and this happens if and only if $\Phi_1$ and $\Phi_2$ are co-maximal and thus connected by an edge in $\mathcal{G}_{\text{co-max}}$.  Assume that $\Phi_1$ and $\Phi_2$ are not co-maximal,  then by Lemma \ref{facts about root sys lemma} (1),  there is $\Phi_3 \in \Borel (\Phi)$ such that 
$\Phi_1 \cap \Phi_2 \subsetneqq \Phi_1 \cap \Phi_3$ and $\Phi_1 \cap \Phi_2 \subsetneqq \Phi_2 \cap \Phi_3$.  Thus
$$\vert \lbrace \mathbb{R} \alpha : \alpha \in \Phi_1 \cap \Phi_2 \rbrace \vert < \vert \lbrace \mathbb{R} \alpha : \alpha \in \Phi_1 \cap \Phi_3 \rbrace \vert$$
and 
$$\vert \lbrace \mathbb{R} \alpha : \alpha \in \Phi_1 \cap \Phi_2 \rbrace \vert < \vert \lbrace \mathbb{R} \alpha : \alpha \in \Phi_2 \cap \Phi_3 \rbrace \vert.$$
It follows from the induction assumption that there is a path in $\mathcal{G}_{\text{co-max}}$ connecting $\Phi_1$ and $\Phi_3$ and a path in $\mathcal{G}_{\text{co-max}}$ connecting $\Phi_2$ and $\Phi_3$ and thus there is a path connecting $\Phi_1$ and $\Phi_2$.
\end{proof}

\subsection{Steinberg groups over rings}
\label{Steinberg groups over rings subsec}

Steinberg groups can be thought of as abstract groups defined via the Chevalley commutator relation.  More explicitly,  given a classical reduced irreducible root system $\Phi$ and a ring $R$,  the group $\St_{\Phi} (R)$ is the group generated by the set $\lbrace x_\alpha (p) : \alpha \in \Phi, p \in R \rbrace$ with the following relations:
\begin{enumerate}
\item For every $\alpha \in \Phi$ and every $p_1, p_2 \in R$,  $x_\alpha (p_1) x_\alpha (p_2) = x_\alpha (p_1 + p_2)$.
\item For every $\alpha, \beta \in \Phi$ and every $p_1, p_2 \in R$, 
$$[x_\alpha (p_1),  x_\beta (p_2)] = \prod_{i, j \in \mathbb{N},  i \alpha + j \beta \in \Phi} x_{ \alpha + j \beta } (C_{i,j}^{\alpha , \beta} , \b p_1^i p_2^j ),$$
where the product is taken in order of increasing $i+j$ and $C_{i,j}^{\alpha , \beta}$ are the constants of the Chevalley commutator relation for Chevalley groups over fields (see for instance \cite[Section 5.2]{Carter}).  
\end{enumerate}
The subgroups $K_\alpha = \lbrace x_\alpha (p) : p \in R \rbrace$ will be called the \textit{root subgroups} of $\St_{\Phi} (R)$.

The most important Steinberg group for our analysis is $\St_{A_2} (R)$,  which can be defined explicitly as follows: For every $p \in R$,  we denote
$$ x_{1,2} (p) = x_\alpha (p),  x_{2,3} (p) = x_\beta (p),  x_{1,3} (p) = x_{\alpha + \beta} (p),$$
$$x_{2,1} (p) = x_{-\alpha} (p),  x_{3,2} (p) = x_{-\beta} (p),  x_{3,1} (p) = x_{- (\alpha + \beta)} (p).$$
With this notation,  by \cite[Proposition 7.6]{EJZK},  the Steinberg group $\St_{A_2} (R)$ is the group generated by the set $\lbrace x_{i,j} (p) : 1 \leq i,j \leq 3, i \neq j,  p \in R \rbrace$ under the relations:
\begin{enumerate}
\item For every $1 \leq i,j \leq 3, i \neq j$ and every $p_1,  p_2 \in R$, 
$$x_{i,j} (p_1) x_{i,j} (p_2) = x_{i,j} (p_1 + p_2).$$
\item For every $1 \leq i,j, k\leq 3$ that are pairwise distinct and every $p_1,  p_2 \in R$, 
$$[x_{i,j} (p_1), x_{j,k} (p_2) ] = x_{i,k} (p_1 p_2).$$
\item For every $1 \leq i,j, k\leq 3$ that are pairwise distinct and every $p_1,  p_2 \in R$, 
$$[x_{i,j} (p_1), x_{i,k} (p_2) ] =  [ x_{j,i} (p_1), x_{k,i} (p_2) ] = e.$$
\end{enumerate}

We also denote $K_{i,j},  i \neq j$ to be the root subgroup
$$K_{i,j} =  \lbrace x_{i,j} (p) : p \in R \rbrace.$$

We will also need some of the relations for the groups $\St_{C_2} (R)$ and $\St_{G_2} (R)$ (these are not all the relations, but only the relations that are needed for us below): 
\begin{proposition} \cite[Proposition 7.6]{EJZK}
\label{C2,  G2 relations prop}
For $\Phi = C_2 = \lbrace \pm \alpha,  \pm \beta,  \pm (\alpha + \beta),  \pm (\alpha + 2 \beta) \rbrace$ ($\alpha$ denotes the long root) it holds for every $p_1, p_2 \in R$ that
$$[x_{\alpha + \beta} (p_1),  x_{\beta} (p_2) ] = x_{\alpha + 2 \beta} (2 p_1 p_2),$$
$$[x_{\alpha} (p_1),  x_\beta (p_2) ] = x_{\alpha + \beta} (p_1 p_2) x_{\alpha + 2 \beta} (p_1 p_2^2).$$

For $\Phi = G_2 = \lbrace \pm \alpha,  \pm \beta,  \pm (\alpha + \beta),  \pm (\alpha + 2 \beta) , \pm (\alpha + 3 \beta), \pm (2 \alpha + 3 \beta)   \rbrace$ ($\alpha$ denotes the long root) it holds for every $p_1, p_2 \in R$ that
$$[x_{\alpha} (p_1),  x_{\beta} (p_2) ] = x_{\alpha + \beta} (p_1 p_2) x_{\alpha + 2 \beta} (p_1 p_2^2) x_{\alpha + 3 \beta} (p_1 p_2^3) x_{2 \alpha + 3 \beta} (p_1^2 p_2^3),$$ 
$$[x_{\alpha} (p_1),  x_{\alpha + 3 \beta} (p_2) ] = x_{2\alpha + 3\beta} (p_1 p_2).$$
\end{proposition}

\section{Action of $\rm H_3 (\mathbb{Z} [t_1,...,t_m])$ on uniformly convex Banach spaces}
\label{Averaging operations sec}

Throughout this section,  $m \in \mathbb{N}$ is a fixed integer.  Below,  we will study the action of $\rm H_3 (\mathbb{Z} [t_1,...,t_m])$ on uniformly convex Banach spaces and deduce results similar to those of \cite{Opp-SLZ}  which dealt with isometric representations of $\rm H_3 (\mathbb{Z})$.  


For every $k \in \mathbb{N} \cup \lbrace 0 \rbrace$ and every $p \in \mathbb{Z} [t_1,...,t_m]$,  we define $X_k  (p) ,  Y_k (p),  Z_k (p)  \in \Prob_c (\rm H_3 (\mathbb{Z} [t_1,...,t_m]))$ by 
$$X_k (p) = \frac{e+x (2^k p)}{2},  Y_k (p) = \frac{e+y (2^k p)}{2}, Z_k (p) = \frac{e+z (2^k p)}{2}.$$

Also,  for $d \in \mathbb{N}$,  we define 
$$X^d (p) = \frac{1}{2^{d}} \sum_{a=0}^{2^{d}-1} x (a p), Y^d = \frac{1}{2^{d}} \sum_{b=0}^{2^{d}-1} y (b p),  Z^d = \frac{1}{2^{d}} \sum_{c=0}^{2^{d}-1} z (c p).$$

We note that 
$$X^d (p) = \prod_{a=0}^{d-1} X_a (p),   Y^d (p) = \prod_{b=0}^{d-1} Y_b (p), Z^d (p)= \prod_{c=0}^{d-1} Z_c (p).$$

Let $\underline{t}$ denote monomials of the form $\underline{t} = t_1^{n_1} ... t_m^{n_m}$ and recall that we denoted $\deg (t_1^{n_1} ... t_m^{n_m}) = n_1 +....+n_m$.  For $d \in \mathbb{N} \cup \lbrace 0 \rbrace$,  denote $\mathcal{S}^d$ to be 
$$\mathcal{S}^d = \lbrace \underline{t}  : \deg (\underline{t})  = d \rbrace,$$
e.g.,  $\mathcal{S}^0 = \lbrace 1 \rbrace$ and  $\mathcal{S}^1 = \lbrace t_1,...,t_m \rbrace$.  Also,  denote $\mathcal{B}^d = \bigcup_{k=0}^d \mathcal{S}^d$.  We observe that $\vert \mathcal{S}^{d} \vert = {d +m -1  \choose d}$ and $\vert \mathcal{B}^{d} \vert = {d +m  \choose d}$.

Define 
$$\underline{X}^d = \prod_{\underline{t} \in \mathcal{B}^d} X^d (\underline{t}),  \underline{Y}^d = \prod_{\underline{t} \in \mathcal{B}^d} Y^d (\underline{t}), \underline{Z}^d = \prod_{\underline{t} \in \mathcal{B}^d} Z^d (\underline{t}).$$

For $d,  d' \in \mathbb{N}$,  we further define
$$\Poly (d, d') = \left\lbrace  \sum_{\underline{t} \in \mathcal{B}^{d}} c_{\underline{t}} \underline{t} \in \mathbb{Z} [t_1,...,t_m] :0 \leq c_{\underline{t}} \leq 2^{d'} -1 \right\rbrace.$$ 

We note that with this notation,  
$$\underline{X}^d = \frac{1}{\vert \Poly (d, d) \vert} \sum_{p \in  \Poly (d, d)} x (p),$$
and similar equalities hold for $\underline{Y}^d$ and $\underline{Z}^d$.

The aim of this section is to bound the (norms of) the following differences: Bounding 
\begin{equation}
\label{needed bound1}
\left\Vert  \rho \left( \underline{X}^{d_1}  \right) \rho \left( \underline{Z}^{d_2} \right)  \rho \left( \underline{Y}^{d_3} \right) \xi -  \rho \left(\underline{Y}^{d_3}    \right) \rho \left( \underline{Z}^{d_2} \right) \rho \left(  \underline{X}^{d_1}  \right) \xi   \right\Vert ,
\end{equation}

\begin{align}
\label{needed bound2}
\left\Vert  \rho \left( \underline{X}^{d_1}  \right) \rho \left( \underline{Z}^{d_2} \right) \rho \left( \underline{Y}^{d_3} \right)  \xi -  \rho \left( \underline{X}^{d_1}  \right) \rho \left( \underline{Z}^{d_2+1} \right) \rho \left( \underline{Y}^{d_3} \right)\xi   \right\Vert, 
\end{align}

\begin{align}
\label{needed bound3}
\left\Vert  \rho \left( \underline{Y}^{d_1}  \right) \rho \left( \underline{Z}^{d_2} \right) \rho \left( \underline{X}^{d_3} \right)  \xi -  \rho \left( \underline{Y}^{d_1}  \right) \rho \left( \underline{Z}^{d_2+1} \right) \rho \left( \underline{X}^{d_3} \right)\xi   \right\Vert, 
\end{align}
where $\rho$ is an affine isometric action of $\rm H_3 (\mathbb{Z} [t_1,...,t_m])$ on a uniformly convex Banach space $\B$.


\begin{lemma}
\label{f_1, f_2 lemma}
Let $\B$ be a Banach space,  $\rho$ an affine isometric action of  $\rm H_3 (\mathbb{Z} [t_1,...,t_m])$ on $\B$ and $f, f', h_1,...,h_k \in \Prob_c (\rm H_3 (\mathbb{Z} [t_1,...,t_m]))$ such that $h_1,...,h_k$ are supported on $\lbrace z (p) : p \in \mathbb{Z} [t_1,...,t_m] \rbrace$.  Assume that there is $\epsilon \in \mathbb{R}$ such that for every $\xi \in \B$ it hold that
$$\Vert \rho (f) \xi - \rho (f ')  \xi \Vert \leq \epsilon \left( \max_{1 \leq i \leq k} \Vert \xi - \rho (h_i) \xi \Vert \right).$$
Then for every $f_1, f_2 \in  \Prob_c (\rm H_3 (\mathbb{Z} [t_1,...,t_m]))$ it holds that 
$$\Vert \rho (f_1) \rho (f) \rho (f_2) \xi - \rho (f_1) \rho (f ') \rho (f_2) \xi \Vert \leq \epsilon \left( \max_{1 \leq i \leq k}  \Vert \xi - \rho (h_i) \xi \Vert \right).$$
\end{lemma}

\begin{proof}
Let $f_1, f_2 \in  \Prob_c (\rm H_3 (\mathbb{Z} [t_1,...,t_m]))$.  We note that by Claim \ref{difference to linear claim} it holds that
\begin{align*}
& \Vert \rho (f_1) \rho (f) \rho (f_2) \xi - \rho (f_1) \rho (f ') \rho (f_2)  \xi \Vert \leq 
 \Vert  \rho (f) \rho (f_2) \xi - \rho (f' ) \rho (f_2)  \xi \Vert \leq  \\
& \epsilon \left( \max_{1 \leq i \leq k}  \Vert \rho (f_2) \xi - \rho (h_i) \rho (f_2)  \xi \Vert \right) = 
\epsilon \left( \max_{1 \leq i \leq k}  \Vert \rho (f_2) \xi -  \rho (f_2) \rho (h_i) \xi \Vert \right) \leq  \epsilon \left( \max_{1 \leq i \leq k}  \Vert \xi - \rho (h_i) \xi \Vert \right) .
\end{align*}
\end{proof}

\begin{lemma}
\label{z(p) lemma}
Let $d,  d',  d_2 \in \mathbb{N}$ be constants such that $d \leq d_2$ and $d' \leq d_2$.  Then for every Banach space $\B$,  every isometric affine action $\rho$ of $ \rm H_3 (\mathbb{Z} [t_1,...,t_m]))$,  every $\xi \in \B$ and every $p \in \Poly (d, d')$  it holds that
$$ \left\Vert \rho (z (p ))    \rho (\underline{Z}^{d_2}) \xi -     \rho (\underline{Z}^{d_2}) \xi \right\Vert \lesssim (d+m)^m  \left( \frac{1}{2} \right)^{d_2-d'} \max_{\underline{t} \in \mathcal{B}^{d}} \Vert \xi - \rho (z (2^{d_2} \underline{t})) \xi \Vert.$$
\end{lemma}

\begin{proof}
Fix $p = \sum_{\underline{t} \in \mathcal{B}^{d}} c_{\underline{t}} \underline{t}$ as above.

We fix an order on $\mathcal{B}^{d}$:
$$\mathcal{B}^{d} = \lbrace \underline{t} (1),  \underline{t} (2), ... \rbrace $$
and according to this order,  we denote the coefficients of $p$ as $c_{\underline{t} (i)} = c_i$.  

We denote $\sum_{i=1}^{0} c_i \underline{t} (i) =0$ and observe that for every $\xi$ it holds that
\begin{dmath*}
 \left\Vert \rho (z (p ))    \rho (\underline{Z}^{d_2}) \xi -     \rho (\underline{Z}^{d_2}) \xi \right\Vert \leq 
\sum_{k=1}^{\vert \mathcal{B}^{d} \vert}  \left\Vert \rho \left( z \left(\sum_{i=1}^{k} c_i \underline{t} (i) \right) \right)    \rho (\underline{Z}^{d_2}) \xi -   \rho \left( z \left(\sum_{i=1}^{k-1} c_i \underline{t} (i) \right) \right)  \rho (\underline{Z}^{d_2}) \xi \right\Vert \leq^{\text{Claim } \ref{difference to linear claim}} 
\sum_{k=1}^{\vert \mathcal{B}^{d} \vert}  \left\Vert \rho (z ( c_k \underline{t} (k) ))    \rho (\underline{Z}^{d_2}) \xi -   \rho (\underline{Z}^{d_2}) \xi \right\Vert.
 \end{dmath*}
 
 We recall that $\vert \mathcal{B}^{d} \vert = {d +m \choose d} \leq (d+m)^m$ and thus it is enough to show that for every $1 \leq c \leq 2^{d'} -1$ and every $\underline{t} \in \mathcal{B}^{d}$ it holds that 
 $$ \left\Vert \rho (z (c \underline{t}))    \rho (\underline{Z}^{d_2}) \xi -      \rho (\underline{Z}^{d_2}) \xi \right\Vert \lesssim c \left( \frac{1}{2} \right)^{d_2}  \Vert \xi - \rho (z (2^{d_2} \underline{t})) \xi \Vert .$$
 
Fix $1 \leq c \leq 2^{d'} -1$ and $\underline{t} \in \mathcal{B}^{d}$.  By Lemma \ref{f_1, f_2 lemma}, it is enough to show that for every $\xi \in \B$ it holds that 
$$ \left\Vert \rho (z (c \underline{t}))    \rho (Z^{d_2} (\underline{t})) \xi -     \rho (Z^{d_2} (\underline{t})) \xi \right\Vert \lesssim c \left( \frac{1}{2} \right)^{d_2}  \Vert \xi - \rho (z (2^{d_2} \underline{t})) \xi \Vert,$$
and this readily follows from Proposition \ref{ergodic diff prop}.
\end{proof}

The following result provides a bound on \eqref{needed bound1}: 
\begin{theorem}
\label{change of order thm}
Let $d_1, d_1', d_3, d_3', d_2 \in \mathbb{N}$ be constants such that $d_1 + d_3 \leq d_2$ and $d_1' +d_3' \leq d_2$.  Also,  let $f, h \in \Prob_c ( \rm H_3 (\mathbb{Z} [t_1,...,t_m]))$ such that 
$$\supp (f) \subseteq \lbrace x (p) : p \in \Poly (d_1,  d_1 ')  \rbrace,$$
and 
$$\supp (h) \subseteq \lbrace y (p) : p \in Poly (d_3,  d_3 ')  \rbrace.$$
Then for every Banach space $\B$  and every isometric affine action $\rho$ of $ \rm H_3 (\mathbb{Z} [t_1,...,t_m]))$ on $\B$ it holds for every $\xi \in \B$ that 
\begin{align*}
\left\Vert \rho (f h) \rho (\underline{Z}^{d_2}) \xi - \rho (h f ) \rho (\underline{Z}^{d_2}) \xi   \right\Vert \lesssim 
(d_2+m)^m  \left( \frac{1}{2} \right)^{d_2-d_1'  - d_3'} \max_{\underline{t} \in \mathcal{B}^{d_1+d_3}} \Vert \xi - \rho (z (2^{d_2} \underline{t})) \xi \Vert.
\end{align*}

In particular,  for $d_1, d_3, d_2 \in \mathbb{N} \cup \lbrace 0 \rbrace$,  if $d_1+d_3 \leq d_2$,  it holds for every Banach space $\B$,  every isometric affine action $\rho$ of $ \rm H_3 (\mathbb{Z} [t_1,...,t_m]))$ on $\B$ and every $\xi \in \B$ that 
\begin{dmath*}
\left\Vert  \rho \left( \underline{X}^{d_1} \right) \rho \left( \underline{Z}^{d_2} \right)  \rho \left( \underline{Y}^{d_3} \right)  \xi -  \rho \left(\underline{Y}^{d_3}    \right) \rho \left( \underline{Z}^{d_2} \right)  \rho \left( \underline{X}^{d_1}  \right)  \xi   \right\Vert \lesssim \\
(d_2+m)^m  \left( \frac{1}{2} \right)^{d_2-d_1  - d_3} \max_{\underline{t} \in \mathcal{B}^{d_2}} \Vert \xi - \rho (z (2^{d_2} \underline{t})) \xi \Vert .
\end{dmath*}
\end{theorem}

\begin{proof}
Fix $\rho,  \B$ and $\xi$ as above.  Let us denote 
$$A_1 = \lbrace p \in \Poly (d_1, d_1 ') : x (p) \in \supp (f) \rbrace,$$
$$A_2 = \lbrace p \in \Poly (d_3, d_3 ')  : y (p) \in \supp (h) \rbrace.$$
We note that 
$$A_1 A_2 \subseteq  \Poly (d_1, d_1 ') \Poly (d_3, d_3 ')  \subseteq \Poly (d_1+ d_3, d_1 '  + d_3 ').$$ 
Thus,
\begin{align*}
& \left\Vert \rho (f h) \rho (\underline{Z}^{d_2}) \xi - \rho (h f ) \rho (\underline{Z}^{d_2}) \xi   \right\Vert \leq \\
& \sum_{p_1 \in A_1,  p_2 \in A_2} f (x (p_1)) h (y (p_2)) \left\Vert \rho (x (p_1) y (p_2))    \rho (\underline{Z}^{d_2}) \xi - \rho ( y (p_2) x (p_1))    \rho (\underline{Z}^{d_2}) \xi \right\Vert =  \\
&  \sum_{p_1 \in A_1,  p_2 \in A_2} f (x (p_1)) h (y (p_2)) \left\Vert \rho (y (p_2) x (p_1) ) \rho (z (p_1 p_2))   \rho (\underline{Z}^{d_2}) \xi - \rho ( y (p_2) x (p_1))    \rho (\underline{Z}^{d_2}) \xi \right\Vert \leq^{\text{Claim } \ref{difference to linear claim}} \\
&  \sum_{p_1 \in A_1,  p_2 \in A_2} f (x (p_1)) h (y (p_2)) \left\Vert \rho (z (p_1 p_2))   \rho (\underline{Z}^{d_2}) \xi - \rho (\underline{Z}^{d_2}) \xi \right\Vert \lesssim^{\text{Lemma } \ref{z(p) lemma}, p_1 p_2 \in \Poly (d_1+ d_3, d_1 '  + d_3 ')} \\
& \sum_{p_1 \in A_1,  p_2 \in A_2} f (x (p_1)) h (y (p_2)) (d_1' + d_3 ' + m)^m  \left( \frac{1}{2} \right)^{d_2-d_1' - d_3 '} \max_{\underline{t} \in \mathcal{B}^{d_1+d_3}} \Vert \xi - \rho (z (2^{d_2} \underline{t})) \xi \Vert \leq \\
& (d_2 +m)^m \left( \frac{1}{2} \right)^{d_2-d_1' - d_3 '} \max_{\underline{t} \in \mathcal{B}^{d_2}} \Vert \xi - \rho (z (2^{d_2} \underline{t})) \xi \Vert ,
\end{align*}
as needed.
\end{proof}

Bounding \eqref{needed bound2},  \eqref{needed bound3} will take more work.  We start by proving a bound in the case of the Heisenberg group $\rm H_3 (\mathbb{Z})$.  To ease the reading,  in that case we denote 
$$X_k (1) = X_k,  X^k (1) = X^k,  Y_k (1) = Y_k,....,  Z^k (1) = Z^k.$$ 

In \cite{Opp-SLZ},  the following result was proven:
\begin{lemma}\cite[Lemma 5.4]{Opp-SLZ}
\label{X,Y prod ineq 2 lemma}
For every uniformly convex Banach space $\B$ with a modulus of convexity $\delta$,   there is a constant $r_1 = r_1 (\delta)$,  $0 \leq r_1  <1$ such that for every isometric linear representation $\pi : \rm H_3 (\mathbb{Z}) \rightarrow O (\B)$,  every  $n,  c_0 \in \mathbb{N}$ such that $1 \leq c_0 \leq 2^{n-1}$ and every  $\zeta \in \B$,  if $\Vert I - \pi (z (c_0) ) \zeta \Vert \geq \frac{1}{2} \Vert \zeta \Vert$,  then 
$$\left\Vert \pi \left( X_0 Y^n \right) \zeta  \right\Vert \leq r_1 \Vert \zeta \Vert.$$
\end{lemma}

Using this Lemma, we prove the following:
\begin{theorem}
\label{X_0 Y^n thm}
For every uniformly convex Banach space $\B$ with a modulus of convexity $\delta$,   there is a constant $r_1 = r_1 (\delta)$,  $0 \leq r_1  <1$ such that for every $n \in \mathbb{N}$,  every affine isometric action $\rho$ of $ \rm H_3 (\mathbb{Z})$ on $\B$ and every $\xi \in \B$ it holds that 
$$\left\Vert  \rho \left( X_0 Y^n \right)  \xi - \rho \left( X_0 Y^n \right) \rho \left( Z_0 \right) \xi   \right\Vert \leq \max \left\lbrace r_1 \Vert \xi - \rho (Z_0) \xi \Vert,  \frac{\Vert \xi - \rho (z (2^{n-1})) \xi \Vert}{2^{n-1}} \right\rbrace.$$
\end{theorem}

\begin{proof}
Let $\pi$ denote the linear part of $\rho$. 

If $\Vert \xi - \rho (Z_0) \xi \Vert \leq \frac{\Vert \xi - \rho (z (2^{n-1})) \xi \Vert}{2^{n-1}} $,  then by Claim \ref{difference to linear claim}, 
\begin{align*}
\left\Vert  \rho \left( X_0 Y^n \right)  \xi - \rho \left( X_0 Y^n \right) \rho \left( Z_0 \right) \xi   \right\Vert \leq
\left\Vert  \xi - \rho \left( Z_0 \right) \xi    \right\Vert \leq \frac{\Vert \xi - \rho (z (2^{n-1})) \xi \Vert}{2^{n-1}},
\end{align*}
and we are done.  

Assume that $\Vert \xi - \rho (Z_0) \xi \Vert > \frac{\Vert \xi - \rho (z (2^{n-1})) \xi \Vert}{2^{n-1}} $.  Then,  using Claim \ref{difference to linear claim}, 
\begin{align*}
&\left\Vert \pi (Z^{n-1} ) \left( \xi - \rho \left( Z_0 \right) \xi \right)   \right\Vert = 
\frac{1}{2} \left\Vert \pi (Z^{n-1} ) \left( \xi- \rho \left( z (1) \right) \xi \right)   \right\Vert = \\
& \frac{1}{2} \left\Vert  \rho (Z^{n-1} ) \xi - \rho (Z^{n-1} ) \rho \left( z (1) \right) \xi  \right\Vert = 
\frac{1}{2} \left\Vert  \frac{1}{2^{n-1}} \sum_{c= 0}^{2^{n-1}-1} \rho (z(c)) \xi - \frac{1}{2^{n-1}} \sum_{c= 1}^{2^{n-1}} \rho (z(c)) \xi \right\Vert = \\
& \frac{1}{2} \left\Vert  \frac{\xi - \rho (z (2^{n-1}) ) \xi}{2^{n-1}} \right\Vert < \frac{1}{2}  \Vert \xi - \rho (Z_0) \xi \Vert.
\end{align*}

It follows that 
\begin{align*}
& \frac{1}{2^{n-1}} \sum_{c=0}^{2^{n-1}-1} \left\Vert (I - \pi (z (c))) \left( \xi - \rho \left( Z_0 \right) \xi \right) \right\Vert \geq 
 \left\Vert (I - \pi (Z^{n-1})) \left( \xi - \rho \left( Z_0 \right) \xi \right) \right\Vert \geq \\
& \left\Vert \xi - \rho \left( Z_0 \right) \xi \right\Vert - \left\Vert \pi (Z^{n-1} ) \left( \xi - \rho \left( Z_0 \right) \xi \right)   \right\Vert  > \frac{1}{2}  \Vert \xi - \rho (Z_0) \xi \Vert.
\end{align*}

Thus, there is $1 \leq c_0 \leq 2^{n-1}$ such that 
$$\Vert (I - \pi (z (c_0) ))  \left( \xi - \rho \left( Z_0 \right) \xi \right)  \Vert \geq \frac{1}{2}  \Vert \xi - \rho (Z_0) \xi \Vert,$$
and by Lemma \ref{X,Y prod ineq 2 lemma} applied to $\zeta = \xi - \rho (Z_0) \xi$ it follows that 
\begin{align*}
\left\Vert  \rho \left( X_0 Y^n \right)  \xi - \rho \left( X_0 Y^n \right) \rho \left( Z_0 \right) \xi   \right\Vert =^{\text{Claim } \ref{difference to linear claim}} 
\left\Vert  \pi \left( X_0 Y^n \right) \left( \xi - \rho \left( Z_0 \right) \xi  \right)  \right\Vert \leq r_1 \Vert \xi - \rho (Z_0) \xi \Vert
\end{align*}
as needed.
\end{proof}

Going back to the group $\rm H_3 (\mathbb{Z} [t_1,...,t_m])$,  we deduce the following:
\begin{corollary}
\label{X,Y prod ineq final coro}
Let $\B$ be a uniformly convex Banach space $\B$ with a modulus of convexity $\delta$ and $r_1 = r_1 (\delta)$,  $0 \leq r_1  <1$ be the constant of Theorem \ref{X_0 Y^n thm}.  For every  $n \in \mathbb{N}$,   every $a_0,b_0 \in \mathbb{N} \cup \lbrace 0 \rbrace$, every $p, q \in  \mathbb{Z} [t_1,...,t_m] \setminus \lbrace 0 \rbrace$, every affine isometric action $\rho$ of $ \rm H_3 (\mathbb{Z}  [t_1,...,t_m])$ on $\B$ and every $\xi \in \B$ it holds that
\begin{align*}
 \left\Vert  \rho \left( X_{a_0} (p) \left( \prod_{b= 0}^{n-1} Y_{b_0+b} (q) \right)  \right)  \xi - \rho \left( X_{a_0} (p) \left( \prod_{b= 0}^{n-1} Y_{b_0+b} (q)\right)  \right) \rho \left( Z_{a_0+b_0} (pq) \right) \xi   \right\Vert \leq \\
 \max \left\lbrace r_1 \Vert \xi - \rho (Z_{a_0+b_0} (pq)) \xi \Vert,  \frac{\Vert \xi - \rho (z (2^{a_0 + b_0 + n-1} pq)) \xi \Vert}{2^{n-1}} \right\rbrace.
\end{align*}

\end{corollary}

\begin{proof}
Fix $\rho,n, a_0,b_0,  p, q$ as above.   

Let $H < H_3 (\mathbb{Z} [t_1,...,t_m])$ be the subgroup $H = \langle x (2^{a_0} p),  y (2^{b_0} q) \rangle$.  
We note that $H$ is isomorphic to $H_3 (\mathbb{Z})$ via the isomorphism $\phi :  H_3 (\mathbb{Z}) \rightarrow H$ induced by $\phi (x (1)) = x (2^{a_0} p),  \phi (y (1)) = y (2^{b_0} q)$.  Note that (by extending $\phi$ linearly)
$$\phi (X_0) = X_{a_0} (p),  \phi (Y^n) = \prod_{b= 0}^{n-1} Y_{b_0+b} (q).$$
Also,  note that
$$\phi (z (1)) = \phi ([x(1), y(1)]) = [x (2^{a_0} p),  y ( 2^{b_0} q)] = z ( 2^{a_0 + b_0} pq),$$
and thus $\phi (Z_0) = Z_{a_0 + b_0} (pq)$.   

Define a new isometric affine action $\rho_0$ of $H_3 (\mathbb{Z})$ on $\B$ by $\rho_0 = \rho \circ \phi$.  Let $\xi \in \B$,  then
\begin{align*}
\left\Vert \rho \left( X_{a_0} (p) \left( \prod_{b=0}^{n-1} Y_{b_0+b} (q) \right)  \right) \xi - \rho \left( X_{a_0} (p) \left( \prod_{b=0}^{n-1} Y_{b_0+b} (q) \right) \right)  \rho \left( Z_{a_0 + b_0} (pq) \right) \xi \right\Vert = \\
\left\Vert  \rho_0 \left( X_0 Y^n \right)  \xi - \rho_0 \left( X_0 Y^n \right) \rho_0 \left( Z_0 \right) \xi   \right\Vert \leq \\
 \max \left\lbrace r_1 \Vert \xi - \rho_0 (Z_0) \xi \Vert,  \frac{\Vert \xi - \rho_0 (z (2^{n-1} )) \xi \Vert}{2^{n-1}} \right\rbrace = \\
 \max \left\lbrace r_1 \Vert \xi - \rho (Z_{a_0+b_0} (pq)) \xi \Vert,  \frac{\Vert \xi - \rho (z (2^{n-1 + a_0 +b_0} pq)) \xi \Vert}{2^{n-1}} \right\rbrace
\end{align*}
as needed.
\end{proof}

As a result, we obtain the following:
\begin{lemma}
\label{a_1,..., lemma}
Let $\B$ be a uniformly convex Banach space $\B$ with a modulus of convexity $\delta$ and $r_1 = r_1 (\delta)$,  $0 \leq r_1  <1$ be the constant of Theorem \ref{X_0 Y^n thm}.  Also,  let $n,N \in \mathbb{N}$,  $a_1,...,a_N,  b_1,...,b_N \in \mathbb{N} \cup \lbrace 0 \rbrace$ and $p_1,...,p_N,  q_1,...,q_N \in \mathbb{Z} [t_1,...,t_m] \setminus \lbrace 0 \rbrace$. Assume there are $d \in \mathbb{N} \cup \lbrace 0 \rbrace$ and $p \in \mathbb{Z} [t_1,...,t_m] \setminus \lbrace 0 \rbrace$ such that for every $1 \leq k \leq N$ it holds that $a_k + b_k = d$ and $p_k q_k = p$.  For every $1 \leq k \leq N$,  we define $f_k \in \Prob_c (\rm H_3 (\mathbb{Z}  [t_1,...,t_m]))$ to be 
$$f_k = X_{a_k} (p_k) \left( \prod_{b= 0}^{n-1} Y_{b_k+b} (q_k) \right).$$

Then for every affine isometric action $\rho$ of $ \rm H_3 (\mathbb{Z}  [t_1,...,t_m])$ on $\B$ and every $\xi \in \B$ it holds that
\begin{align*}
 \left\Vert  \rho \left( \prod_{k=1}^{N} f_k  \right)  \xi - \rho  \left( \prod_{k=1}^{N} f_k \right)    \rho \left( Z_{d} (p) \right) \xi   \right\Vert \leq 
 \max \left\lbrace r_1^N \Vert \xi - \rho (Z_{d} (p)) \xi \Vert,  \frac{\Vert \xi - \rho (z (2^{d + n-1} p)) \xi \Vert}{2^{n-1}} \right\rbrace.
\end{align*}

In particular,  if $N= n$,  then 
\begin{align*}
 \left\Vert  \rho \left( \prod_{k=1}^{n} f_k  \right)  \xi - \rho  \left( \prod_{k=1}^{n} f_k \right)    \rho \left( Z_{d} (p) \right) \xi   \right\Vert \leq 
 \max \left\lbrace r_1^{n} \Vert \xi - \rho (Z_{d} (p)) \xi \Vert,  \frac{\Vert \xi - \rho (z (2^{d + n-1} p)) \xi \Vert}{2^{n-1}} \right\rbrace \lesssim \\
\left( \max \left\lbrace r_1,  \frac{1}{2} \right\rbrace \right)^{n} \max \left\lbrace  \Vert \xi - \rho (z (2^{d} p)) \xi \Vert,  \Vert \xi - \rho (z (2^{d + n-1} p)) \xi \Vert \right\rbrace.
\end{align*}
\end{lemma}

\begin{proof}
Fix $\rho$ and $\xi$ as above.  
The proof on the inequality is by induction on $N$.  The case $N=1$ was proved in Corollary \ref{X,Y prod ineq final coro}.  

Assume that the inequality holds for $N$ and let  $a_1,...,a_{N+1},  b_1,...,b_{N+1} \in \mathbb{N} \cup \lbrace 0 \rbrace$ and $p_1,...,p_{N+1},  q_1,...,q_{N+1} \in \mathbb{Z} [t_1,...,t_m] \setminus \lbrace 0 \rbrace$ such that there are $d \in \mathbb{N} \cup \lbrace 0 \rbrace$ and $p \in \mathbb{Z} [t_1,...,t_m] \setminus \lbrace 0 \rbrace$ such that for every $1 \leq k \leq N+1$ it holds that $a_k + b_k = d$ and $p_k q_k = p$.  Denote 
$$f =  \prod_{k=2}^{N+1} f_k .$$ 

By Corollary \ref{X,Y prod ineq final coro}
\begin{dmath*}
 \left\Vert  \rho \left( \prod_{k=1}^{N+1} f_k  \right)  \xi - \rho  \left( \prod_{k=1}^{N+1} f_k \right)    \rho \left( Z_{d} (p) \right) \xi   \right\Vert = 
  \left\Vert  \rho \left(f_ 1 f  \right)  \xi - \rho  \left( f_1 f \right)    \rho \left( Z_{d} (p) \right) \xi   \right\Vert = 
    \left\Vert  \rho \left(f_ 1 \right)  \rho (f) \xi - \rho  \left( f_1  \right)    \rho \left( Z_{d} (p) \right) \rho (f) \xi   \right\Vert \leq
 \max \left\lbrace r_1 \Vert \rho (f) \xi - \rho (Z_{d} (p)) \rho (f) \xi \Vert,  \frac{\Vert  \rho (f) \xi  - \rho (z (2^{d + n-1} p)) \rho (f) \xi \Vert}{2^{n-1}} \right\rbrace = 
 \max \left\lbrace r_1 \Vert \rho (f) \xi - \rho (f) \rho (Z_{d} (p))  \xi \Vert,  \frac{\Vert  \rho (f) \xi  - \rho (f)  \rho (z (2^{d + n-1} p)) \xi \Vert}{2^{n-1}} \right\rbrace .
\end{dmath*}
If $r_1 \Vert \rho (f) \xi - \rho (f) \rho (Z_{d} (p))  \xi \Vert \leq \frac{\Vert  \rho (f) \xi  - \rho (f)  \rho (z (2^{d + n-1} p)) \xi \Vert}{2^{n-1}}$, we are done by Claim \ref{difference to linear claim}:
$$ \frac{\Vert  \rho (f) \xi  - \rho (f)  \rho (z (2^{d + n-1} p)) \xi \Vert}{2^{n-1}} \leq  \frac{\Vert \xi  - \rho (z (2^{d + n-1} p)) \xi \Vert}{2^{n-1}}.$$

Assume that $r_1 \Vert \rho (f) \xi - \rho (f) \rho (Z_{d} (p))  \xi \Vert > \frac{\Vert  \rho (f) \xi  - \rho (f)  \rho (z (2^{d + n-1} p)) \xi \Vert}{2^{n-1}}$, then by the induction assumption
\begin{align*}
r_1 \left\Vert  \rho \left( f \right)  \xi - \rho  \left( f \right)    \rho \left( Z_{d} (p) \right) \xi   \right\Vert \leq 
r_1 \max \left\lbrace r_1^N \Vert \xi - \rho (Z_{d} (p)) \xi \Vert,  \frac{\Vert \xi - \rho (z (2^{d + n-1} p)) \xi \Vert}{2^{n-1}} \right\rbrace \leq \\
 \max \left\lbrace r_1^{N+1} \Vert \xi - \rho (Z_{d} (p)) \xi \Vert,  \frac{\Vert \xi - \rho (z (2^{d + n-1} p)) \xi \Vert}{2^{n-1}} \right\rbrace
\end{align*}
as needed.
\end{proof}

\begin{lemma}
\label{change h_k to f_k lemma}
Let $\B$ be a Banach space $\B$ and $d_2 \in \mathbb{N}$ be  a constant.  Also,  let $n,N \in \mathbb{N},   2n < d_2$,  $a_1,...,a_N,  b_1,...,b_N \in \mathbb{N} \cup \lbrace 0 \rbrace$ and $p_1,...,p_N,  q_1,...,q_N \in  \mathcal{B}^{d_2}$ such that for every $1 \leq k  \leq N-1$ the following holds:
\begin{itemize}
\item $a_k \geq a_{k+1}$ and $b_k \leq b_{k+1}$
\item $\deg (p_{k}) \geq \deg (p_{k+1})$ and $\deg (q_{k}) \leq \deg (q_{k+1})$
\item $a_{k+1} + b_k \leq d_2-2n$ and $\deg (p_{k+1}) + \deg (q_k) \leq d_2$
\item Either $a_k > a_{k+1} $ or $\deg (p_k) >  \deg (p_{k+1})$
\end{itemize}

For every $1 \leq k \leq N$,  we define $h_k,  f_k \in \Prob_c (\rm H_3 (\mathbb{Z}  [t_1,...,t_m]))$ to be 
$$h_k =  \prod_{b= 0}^{n-1} Y_{b_k+b} (q_k) ,$$  
$$f_k = X_{a_k} (p_k) h_k = X_{a_k} (p_k) \left( \prod_{b= 0}^{n-1} Y_{b_k+b} (q_k) \right).$$

Then for every affine isometric action $\rho$ of $\rm H_3 (\mathbb{Z} [t_1,...,t_m ])$ on $\B$ and every $\xi \in \B$ it holds that 
\begin{align*}
\left\Vert \rho \left( \prod_{k=1}^N X_{a_k} (p_k) \right) \rho \left(\prod_{k=1}^N  h_k \right) \rho \left(\underline{Z}^{d_2} \right) \xi - \rho \left( \prod_{k=1}^N f_k \right) \rho \left(\underline{Z}^{d_2} \right) \xi \right\Vert \lesssim \\
(N-1) (d_2 + m)^m \left( \frac{1}{2} \right)^{n} \max_{\underline{t} \in \mathcal{B}^{d_2}} \Vert \xi - \rho (z (2^{d_2} \underline{t})) \xi \Vert.
\end{align*}

In particular,  if $N=n$, then 
\begin{align*}
\left\Vert \rho \left( \prod_{k=1}^{n} X_{a_k} (p_k) \right) \rho \left(\prod_{k=1}^{n}  h_k \right) \rho \left(\underline{Z}^{d_2} \right) \xi - \rho \left( \prod_{k=1}^{n} f_k \right) \rho \left(\underline{Z}^{d_2} \right) \xi \right\Vert \lesssim \\
(n-1) (d_2 + m)^m \left( \frac{1}{2} \right)^{n} \max_{\underline{t} \in \mathcal{B}^{d_2}} \Vert \xi - \rho (z (2^{d_2} \underline{t})) \xi \Vert \lesssim (d_2 +m)^m \left( \frac{3}{4} \right)^{n} \max_{\underline{t} \in \mathcal{B}^{d_2}} \Vert \xi - \rho (z (2^{d_2} \underline{t})) \xi \Vert.
\end{align*}

\end{lemma}

\begin{proof}
By Lemma \ref{f_1, f_2 lemma} (and the triangle inequality),  it is enough to prove that for every $\rho$,  every $\xi$ and every $1 \leq k_0 \leq N-1$ it holds that 
\begin{align*}
\left\Vert \rho \left( \prod_{k=k_0+1}^N X_{a_k} (p_k) \right) \rho \left(  h_{k_0} \right) \rho \left(\underline{Z}^{d_2} \right) \xi -  \rho \left(  h_{k_0} \right) \rho \left( \prod_{k=k_0+1}^N X_{a_k} (p_k) \right) \rho \left(\underline{Z}^{d_2} \right) \xi \right\Vert \lesssim \\
(d_2 + m)^m \left( \frac{1}{2} \right)^{n} \max_{\underline{t} \in \mathcal{B}^{d_2}} \Vert \xi - \rho (z (2^{d_2} \underline{t})) \xi \Vert.
\end{align*}

Denote $f =  \prod_{k=k_0+1}^N X_{a_k} (p_k)$.  We note that by the assumption that for every $k$,  $a_k \geq a_{k+1} $ and $\deg (p_k) \geq  \deg (p_{k+1})$ and at least one of these inequalities is a sharp.  Thus,
$$\supp (f) \subseteq \lbrace x (p) : p \in  \Poly (\deg (p_{k_0 +1}), a_{k_0 +1} +1 ) \rbrace,$$
and
$$\supp (h_{k_0}) \subseteq \lbrace y (p) : p \in \Poly (\deg (q_{k_0} ),   b_{k_0} +n ) \rbrace .$$

By Theorem \ref{change of order thm} it holds that 
\begin{align*}
& \left\Vert \rho \left( f \right) \rho \left(  h_{k_0} \right) \rho \left(\underline{Z}^{d_2} \right) \xi -  \rho \left(  h_{k_0} \right) \rho \left( f \right) \rho \left(\underline{Z}^{d_2} \right) \xi \right\Vert \lesssim \\
& (d_2 + m)^m \left( \frac{1}{2} \right)^{d_2-(a_{k_0+1}+1)-(b_{k_0}+n)} \max_{\underline{t} \in \mathcal{B}^{\deg (p_{k_0+1} )+\deg (q_{k_0})}} \Vert \xi - \rho (z (2^{d_2} \underline{t})) \xi \Vert \lesssim^{a_{k_0+1} + b_{k_0} \leq d_2 - 2n} \\
& (d_2 + m)^m \left( \frac{1}{2} \right)^{n} \max_{\underline{t} \in \mathcal{B}^{d_2}} \Vert \xi - \rho (z (2^{d_2} \underline{t})) \xi \Vert.
\end{align*}
\end{proof}

Combining these two lemmas yields the following result:

\begin{theorem}
\label{r_2 thm}
Let $d_1, d_2, d_3,n \in \mathbb{N},  d ' \in \mathbb{N} \cup \lbrace 0 \rbrace$ be constants and $\underline{t} ' \in \mathbb{Z} [t_1,...,t_m] \setminus \lbrace 0 \rbrace$ be a monomial.  Also,  let $\B$ be a uniformly convex Banach space with a modulus of convexity $\delta$.  Denote $r_2 = r_2 (\delta) <1$ to be
$$r_2 = \max \left\lbrace r_1 (\delta ),  \frac{3}{4} \right\rbrace,$$
where $r_1 (\delta)$ is the constant of Theorem \ref{X_0 Y^n thm}.

If there are $n \in \mathbb{N}$,  $2n < d_2$,  $a_1,...,a_{n},  b_1,...,b_{n} \in \mathbb{N} \cup \lbrace 0 \rbrace$ and $p_1,...,p_{n},  q_1,...,q_{n} \in \mathcal{B}^{d_2}$ such that the following holds: 
\begin{itemize}
\item For every $1 \leq k \leq n$,  $a_k + b_k = d '$ and $p_k q_k = \underline{t} '$
\item For every $1 \leq k \leq n-1$,  $d_1 -1 \geq a_k \geq a_{k+1}$ and $b_k \leq b_{k+1} \leq d_3 -1$
\item For every $1 \leq k \leq n-1$,  $\deg (p_{k}) \geq \deg (p_{k+1})$ and $\deg (q_{k}) \leq \deg (q_{k+1})$
\item For every $1 \leq k \leq n-1$,  $a_{k+1} + b_k \leq d_2-2n$ and $\deg (p_{k+1}) + \deg (q_k) \leq d_2$
\item For every $1 \leq k \leq n-1$,  either $a_k > a_{k+1} $ or $\deg (p_k) >  \deg (p_{k+1})$
\end{itemize}
Then for every affine isometric action $\rho$ of  $ \rm H_3 (\mathbb{Z} [t_1,...,t_m])$ on $\B$ and every $\xi \in \B$ it holds that 
\begin{align*}
\left\Vert \rho \left( \underline{X}^{d_1} \right) \rho \left(  \underline{Z}^{d_2} \right) \rho \left(  \underline{Y}^{d_3} \right) \xi -  \rho \left( \underline{X}^{d_1}  \right) \rho \left(  \underline{Z}^{d_2} \right)  \rho \left(Z_{d '} (\underline{t} ' )\right) \rho \left(  \underline{Y}^{d_3} \right) \xi  \right\Vert \lesssim \\
(d_2 + m)^m r_2^{n}  \max \left\lbrace \max_{\underline{t} \in \mathcal{B}^{d_2}} \Vert \xi - \rho (z (2^{d_2} \underline{t})) \xi \Vert,   \Vert \xi - \rho (z (2^{d '} \underline{t} ' )) \xi \Vert,  \Vert \xi - \rho (z (2^{d ' + n-1} \underline{t} ')) \xi \Vert \right\rbrace.
\end{align*}
\end{theorem}

\begin{proof}
Fix $\rho$ as above.  Denote 
$$h_k =  \prod_{b= 0}^{n-1} Y_{b_k+b} (q_k) , $$
$$ f_k = X_{a_k} (p_k) h_k = X_{a_k} (p_k) \left( \prod_{b= 0}^{n-1} Y_{b_k+b} (q_k) \right).$$
By Lemma \ref{f_1, f_2 lemma}, it is enough to prove that for every $\xi \in \B$, 
\begin{dmath*}
\left\Vert \rho \left( \prod_{k=1}^{n} X_{a_k} (p_k) \right) \rho \left(\underline{Z}^{d_2} \right)  \rho \left(\prod_{k=1}^{n}  h_k \right) \xi - \rho \left( \prod_{k=1}^{n} X_{a_k} (p_k) \right) \rho \left(\underline{Z}^{d_2} \right) \rho \left(Z_{d '} (\underline{t} ') \right)  \rho \left(\prod_{k=1}^{n}  h_k \right) \xi  \right\Vert \lesssim \\
(d_2 + m)^m r_2^{n}  \max \left\lbrace \max_{\underline{t} \in \mathcal{B}^{d_2}} \Vert \xi - \rho (z (2^{d_2} \underline{t})) \xi \Vert,   \Vert \xi - \rho (z (2^{d '} \underline{t} ' )) \xi \Vert,  \Vert \xi - \rho (z (2^{d ' + n-1} \underline{t} ')) \xi \Vert \right\rbrace.
\end{dmath*}
Fix $\xi \in \B$. Then 
\begin{dmath*}
\left\Vert \rho \left( \prod_{k=1}^{n} X_{a_k} (p_k) \right) \rho \left(\underline{Z}^{d_2} \right)  \rho \left(\prod_{k=1}^{n}  h_k \right) \xi - \rho \left( \prod_{k=1}^{n} X_{a_k} (p_k) \right) \rho \left(\underline{Z}^{d_2} \right) \rho \left(Z_{d '} (\underline{t} ') \right)  \rho \left(\prod_{k=1}^{n}  h_k \right) \xi  \right\Vert \lesssim^{\text{Lemma } \ref{change h_k to f_k lemma}} \\
{(d_2 +m)^m r_2^{n}  \max_{\underline{t} \in \mathcal{B}^{d_2}} \Vert \xi - \rho (z (2^{d_2} \underline{t})) \xi \Vert } +  \left\Vert \rho \left(\underline{Z}^{d_2} \right) \rho \left( \prod_{k=1}^{n} f_k \right)    \xi - \rho \left(\underline{Z}^{d_2} \right) \rho \left( \prod_{k=1}^{n} f_k \right)  \rho \left(Z_{d '} (\underline{t} ') \right)  \xi  \right\Vert \leq^{\text{Claim } \ref{difference to linear claim}} \\
(d_2 +m)^m r_2^{n}  \max_{\underline{t} \in \mathcal{B}^{d_2}} \Vert \xi - \rho (z (2^{d_2} \underline{t})) \xi \Vert +  \left\Vert \rho \left( \prod_{k=1}^{n} f_k \right)    \xi -  \rho \left( \prod_{k=1}^{n} f_k \right)  \rho \left(Z_{d '} (\underline{t} ') \right)  \xi  \right\Vert \lesssim^{\text{Lemma } \ref{a_1,..., lemma}}  \\
(d_2 +m)^m r_2^{n}  \max \left\lbrace \max_{\underline{t} \in \mathcal{B}^{d_2}} \Vert \xi - \rho (z (2^{d_2} \underline{t})) \xi \Vert,   \Vert \xi - \rho (z (2^{d '} \underline{t} ' )) \xi \Vert,  \Vert \xi - \rho (z (2^{d ' + n-1} \underline{t} ')) \xi \Vert \right\rbrace.
\end{dmath*}
\end{proof}

This result has two useful instantiations:
\begin{corollary}
\label{d' geq d_2 -2n coro} 
Let $d_1, d_2, d_3 \in \mathbb{N}, d_2 \geq 3$ be constants such that $d_1, d_3 \leq d_2,  d_1+d_3 \geq d_2 +4$ and $\B$ be a uniformly convex Banach space $\B$ with a modulus of convexity $\delta$.  Also,  let $r_2 = r_2 (\delta)$ be the constant of Theorem \ref{r_2 thm}.  

Denote $n = \lfloor \frac{\sqrt{d_1+d_3 -d_2}}{2} \rfloor$.  Then for every $\underline{t} ' \in \mathcal{B}^{d_2+1}$,  every $d_2 - 2n \leq d ' \leq d_2$,  every affine isometric action $\rho$ of  $ \rm H_3 (\mathbb{Z} [t_1,...,t_m])$ on $\B$ and every $\xi \in \B$ it holds that 
\begin{align*}
\left\Vert \rho \left( \underline{X}^{d_1} \right) \rho \left(  \underline{Z}^{d_2} \right)  \rho \left(  \underline{Y}^{d_3} \right)  \xi -  \rho \left( \underline{X}^{d_1}  \right) \rho \left(  \underline{Z}^{d_2} \right)  \rho \left(Z_{d '} (\underline{t} ' )\right)   \rho \left( \underline{Y}^{d_3} \right)  \xi  \right\Vert \lesssim \\
(d_2 +m)^m r_2^{n}  \max_{0 \leq d \leq d_2 +n-1}  \max_{\underline{t} \in \mathcal{B}^{d_2+1}} \Vert \xi - \rho (z (2^{d} \underline{t})) \xi \Vert.
 \end{align*}
\end{corollary} 

\begin{proof}
Fix $\underline{t}',  d', \rho$ and $\xi$ as above.  

The proof is an application of Theorem  \ref{r_2 thm}.   Observe that $2n \leq \sqrt{2 d_2} < d_2$.  

We introduce the following notation:  recall that the monomial $\underline{t}' \in \mathcal{B}^{d_2+1}$ is of the form $\underline{t}' = t_1^{n_1} ... t_m^{n_m}$ where $n_1,...,n_m \in \mathbb{N} \cup \lbrace 0 \rbrace$ and $n_1 + ... + n_m \leq d_2 +1$.  For $k \in \mathbb{Z}$, we denote $\prefix (\underline{t} ',  k)$ to be the $i$-th prefix of $\underline{t} '$:
$$\prefix (\underline{t} ',  k) = \begin{cases}
1 & k \leq 0 \\
\underline{t} ' & k > n_1 + ...  + n_m \\
t_1^{n_1} ... t_i^{n_i} t_{i+1}^{k - (n_1 +...+n_i)} & n_1 +...+n_i < k \text{ and }  n_1 +...+n_{i+1} \geq k 
\end{cases}.$$
Similarly,  for $k \in \mathbb{Z}$, we denote $\suffix (\underline{t} ',  k)$ to be the $k$-th suffix of $\underline{t} '$,  i.e.,  
$$\prefix (\underline{t} ',  k) = \begin{cases}
1 & k \leq 0 \\
\underline{t} ' & k > n_1 + ...  + n_m \\
t_i^{k - (n_1 +...+n_{i+1})} t_{i+1}^{n_{i+1}} ... t_m^{n_m} & n_{i+1} +...+n_m < k \text{ and }  n_i +...+n_{m} \geq k 
\end{cases}.$$

We note that for every $k$,  
$$\prefix (\underline{t} ',   \deg (\underline{t}') - k)  \suffix (\underline{t} ',  k) = \underline{t} '.$$

For $1 \leq k \leq n$,  we define 
$$a_k = d_1 - 2n k,  b_k = d ' - a_k,  p_k = \prefix (\underline{t} ',   \deg (\underline{t}') - k) ,  q_k  = \suffix (\underline{t} ',  k) .$$
We will finish the proof by verifying that for this choice,  the conditions of Theorem \ref{r_2 thm} are fulfilled:
\begin{itemize}
\item It is obvious that for every $1 \leq k \leq n$,  $a_k + b_k = d '$ and $p_k q_k = \underline{t} '$.
\item For every $1 \leq k \leq n-1$,  $d_1 -1 \geq a_k > a_{k+1}$ and we note that
\begin{dmath*}
a_{n} = d_1  - 2n^2  \geq d_1  - 2\left(\frac{\sqrt{d_1+d_3 -d_2}}{2} \right)^2  = {\frac{d_1}{2}  + \frac{d_2}{2} -\frac{d_3}{2} \geq^{d_3 \leq d_2} 0},
\end{dmath*}
i.e.,  for every $k$,  $a_k \in \mathbb{N} \cup \lbrace 0 \rbrace$.  Also,  for every $1 \leq k \leq n-1$,  $b_k < b_{k+1}$ and we note that 
$$b_1 = d ' - a_1 = d ' - d_1   + 2n \geq^{d' > d_2 -2n} d_2 -d_1 \geq 0 .$$
Also,
\begin{align*}
b_{n} = d ' - a_{n} \leq d ' - (\frac{d_1}{2}  + \frac{d_2}{2} -\frac{d_3}{2}) \leq^{d ' \leq d_2}  \frac{d_3}{2} + \frac{d_2}{2} - \frac{d_1}{2} = \\
d_3 + \frac{d_2-(d_1+d_3)}{2}  \leq^{d_1+d_3 \geq d_2 +4} d_3 -2, 
\end{align*}
i.e.,  for every $1 \leq k \leq n$,  $b_k \in \mathbb{N} \cup \lbrace 0 \rbrace$ and $b_k \leq d_3 -1$.  
\item It is obvious that for every $1 \leq k \leq n-1$,  $\deg (p_{k}) \geq \deg (p_{k+1})$ and $\deg (q_{k}) \leq \deg (q_{k+1})$.
\item For every $1 \leq k \leq n-1$,  
$$a_{k+1} + b_k = d_1- 2n (k+1) + (d ' - (d_1- 2n k)) = d' + 2n \leq d_2 +2n.$$
Last,  we need to show that 
$$\deg (p_{k+1}) + \deg (q_k) = \deg (\underline{t} ' ) \leq d_2.$$
If $\deg (\underline{t} ') \leq n$,  then 
$$\deg (p_{k+1}) + \deg (q_k) \leq 2\deg (\underline{t} ' ) \leq 2n  < d_2.$$
If $d_2 +1 \geq \deg (\underline{t} ')  >n$,  then for every $1 \leq k \leq n$,
$$\deg  (p_{k+1}) + \deg (q_k) = \deg (\underline{t} ') - (k+1) + k = \deg (\underline{t} ') -1 \leq d_2.$$
\end{itemize}

\end{proof}

\begin{corollary}
\label{d' < d_2 -2n coro}
Let $d_1, d_2, d_3 \in \mathbb{N},  d_2 \geq 3$ be constants such that $d_1, d_3 \leq d_2,  d_1+d_3 \geq d_2 +4$ and $\B$ be a uniformly convex Banach space $\B$ with a modulus of convexity $\delta$.  Also,  let $r_2 = r_2 (\delta)$ be the constant of Theorem \ref{r_2 thm}.  

Denote $n = \lfloor \frac{\sqrt{d_1+d_3 -d_2}}{2} \rfloor$.  Then for every $\underline{t} ' \in \mathcal{S}^{d_2+1}$, every , $d ' \in \mathbb{N} \cup \lbrace 0 \rbrace,  d' \leq d_2 -2n$,  every affine isometric action $\rho$ of  $ \rm H_3 (\mathbb{Z} [t_1,...,t_m])$ on $\B$ and every $\xi \in \B$ it holds that 
\begin{align*}
\left\Vert \rho \left( \underline{X}^{d_1} \right) \rho \left(  \underline{Z}^{d_2} \right)  \rho \left(  \underline{Y}^{d_3} \right)  \xi -  \rho \left( \underline{X}^{d_1}  \right) \rho \left(  \underline{Z}^{d_2} \right)  \rho \left(Z_{d '} (\underline{t} ' )\right)   \rho \left( \underline{Y}^{d_3} \right)  \xi  \right\Vert \lesssim \\
(d_2 + m)^m r_2^{n}  \max_{0 \leq d \leq d_2 + n-1} \max_{\underline{t} \in \mathcal{B}^{d_2+1}} \Vert \xi - \rho (z (2^{d} \underline{t})) \xi \Vert.
 \end{align*}
\end{corollary} 

\begin{proof}
Fix $\rho,  \xi,  \underline{t}',  d'$ as above.  

The proof is an application of Theorem  \ref{r_2 thm}.   Observe that $2n \leq \sqrt{2 d_2} < d_2$.  

By the assumption that $d_1 + d_3 \geq d_2+4$ there are $0 \leq a_0 \leq d_1 -1$ and $0 \leq b_0 \leq d_3-1$ such that $a_0 + b_0 = d'$. Using the notation of $\prefix,  \suffix$ as in the proof of the previous corollary,  we define
$$a_k = a_0,  b_k = b_0,  p_k = \prefix (\underline{t} ',  d_2+1-k),  q_k = \suffix (\underline{t} ',  k).$$
Checking that the conditions of Theorem \ref{r_2 thm} are fulfilled is straight-forward and we leave it for the reader.
\end{proof}

Combining these two corollaries leads to the following:
\begin{theorem}
\label{d_2 to d_2+1 thm}
Let $d_1, d_2, d_3 \in \mathbb{N},  d_2 \geq 3$ be constants such that $d_1, d_3 \leq d_2,  d_1+d_3 \geq d_2 +4$ and $\B$ be a uniformly convex Banach space $\B$ with a modulus of convexity $\delta$.  Also,  let $r_2 = r_2 (\delta)$ be the constant of Theorem \ref{r_2 thm}.  

Denote $n = \lfloor \frac{\sqrt{d_1+d_3 -d_2}}{2} \rfloor$.  Then for every affine isometric action $\rho$ of  $ \rm H_3 (\mathbb{Z} [t_1,...,t_m])$ on $\B$ and every $\xi \in \B$ it holds that 
\begin{align*}
\left\Vert \rho \left( \underline{X}^{d_1} \right) \rho \left(  \underline{Z}^{d_2} \right)  \rho \left(  \underline{Y}^{d_3} \right)  \xi -  \rho \left( \underline{X}^{d_1}  \right) \rho \left(  \underline{Z}^{d_2+1} \right)     \rho \left( \underline{Y}^{d_3} \right)  \xi  \right\Vert \lesssim \\
(d_2 + m)^{2m} r_2^{n}  \max_{0 \leq d \leq d_2 + n-1} \max_{\underline{t} \in \mathcal{B}^{d_2+1}} \Vert \xi - \rho (z (2^{d} \underline{t})) \xi \Vert.
 \end{align*}
\end{theorem}

\begin{proof}
Fix $\rho,  \xi$ as above.
We note that 
$$\underline{Z}^{d_2+1} = \underline{Z}^{d_2} \left( \prod_{\underline{t'} \in \mathcal{B}^{d_2}} Z_{d_2} (\underline{t} ') \right) \left( \prod_{\underline{t} '' \in \mathcal{S}^{d_2+1}} \prod_{d' = 0}^{d_2} Z_{d '} (\underline{t} '') \right).$$

Denote 
$$Z_{d_2} (\mathcal{B}^{d_2} ) = \left( \prod_{\underline{t'} \in \mathcal{B}^{d_2}} Z_{d_2} (\underline{t} ') \right),$$
$$Z^{d_2+1} (\mathcal{S}^{d_2+1} ) =\left( \prod_{\underline{t} '' \in \mathcal{S}^{d_2+1}} \prod_{d' = 0}^{d_2} Z_{d '} (\underline{t} '') \right).$$

Thus, 
\begin{dmath*}
\left\Vert \rho \left( \underline{X}^{d_1} \right) \rho \left(  \underline{Z}^{d_2} \right)  \rho \left(  \underline{Y}^{d_3} \right)  \xi -  \rho \left( \underline{X}^{d_1}  \right) \rho \left(  \underline{Z}^{d_2+1} \right)     \rho \left( \underline{Y}^{d_3} \right)  \xi  \right\Vert \leq 
\left\Vert \rho \left( \underline{X}^{d_1} \right) \rho \left(  \underline{Z}^{d_2} \right)  \rho \left(  \underline{Y}^{d_3} \right)  \xi -  \rho \left( \underline{X}^{d_1}  \right) \rho \left(  \underline{Z}^{d_2} Z_{d_2} (\mathcal{B}^{d_2} ) \right)     \rho \left( \underline{Y}^{d_3} \right)  \xi  \right\Vert + \\
 \left\Vert \rho \left( \underline{X}^{d_1} \right) \rho \left(  \underline{Z}^{d_2} Z_{d_2} (\mathcal{B}^{d_2} )  \right)  \rho \left(  \underline{Y}^{d_3} \right)  \xi -  \rho \left( \underline{X}^{d_1}  \right) \rho \left(  \underline{Z}^{d_2} Z_{d_2} (\mathcal{B}^{d_2} )  Z^{d_2+1} (\mathcal{S}^{d_2+1} ) \right)     \rho \left( \underline{Y}^{d_3} \right)  \xi  \right\Vert = 
 \left\Vert \rho \left( \underline{X}^{d_1} \right) \rho \left(  \underline{Z}^{d_2} \right)  \rho \left(  \underline{Y}^{d_3} \right)  \xi -  \rho \left( \underline{X}^{d_1}  \right) \rho \left(  \underline{Z}^{d_2} Z_{d_2} (\mathcal{B}^{d_2} ) \right)     \rho \left( \underline{Y}^{d_3} \right)  \xi  \right\Vert + \\
 \left\Vert \rho \left( Z_{d_2} (\mathcal{B}^{d_2} ) \right) \rho \left( \underline{X}^{d_1} \right) \rho \left(  \underline{Z}^{d_2}  \right)  \rho \left(  \underline{Y}^{d_3} \right)  \xi -  \rho \left( Z_{d_2} (\mathcal{B}^{d_2} ) \right) \rho \left( \underline{X}^{d_1}  \right) \rho \left(  \underline{Z}^{d_2}  Z^{d_2+1} (\mathcal{S}^{d_2+1} ) \right)     \rho \left( \underline{Y}^{d_3} \right)  \xi  \right\Vert \leq^{\text{Claim } \ref{difference to linear claim}} 
 \left\Vert \rho \left( \underline{X}^{d_1} \right) \rho \left(  \underline{Z}^{d_2} \right)  \rho \left(  \underline{Y}^{d_3} \right)  \xi -  \rho \left( \underline{X}^{d_1}  \right) \rho \left(  \underline{Z}^{d_2} Z_{d_2} (\mathcal{B}^{d_2} ) \right)     \rho \left( \underline{Y}^{d_3} \right)  \xi  \right\Vert + \\
 \left\Vert \rho \left( \underline{X}^{d_1} \right) \rho \left(  \underline{Z}^{d_2} \right)  \rho \left(  \underline{Y}^{d_3} \right)  \xi -  \rho \left( \underline{X}^{d_1}  \right) \rho \left(  \underline{Z}^{d_2}  Z^{d_2+1} (\mathcal{S}^{d_2+1} ) \right)     \rho \left( \underline{Y}^{d_3} \right)  \xi  \right\Vert.
\end{dmath*}

It follows that it is enough to prove that 
\begin{align*}
 \left\Vert \rho \left( \underline{X}^{d_1} \right) \rho \left(  \underline{Z}^{d_2} \right)  \rho \left(  \underline{Y}^{d_3} \right)  \xi -  \rho \left( \underline{X}^{d_1}  \right) \rho \left(  \underline{Z}^{d_2} Z_{d_2} (\mathcal{B}^{d_2} ) \right)     \rho \left( \underline{Y}^{d_3} \right)  \xi  \right\Vert \lesssim \\
 (d_2 + m)^{2m+1} r_2^{n}  \max_{0 \leq d \leq d_2 + n-1} \max_{\underline{t} \in \mathcal{B}^{d_2+1}} \Vert \xi - \rho (z (2^{d} \underline{t})) \xi \Vert
\end{align*}
and
\begin{align*}
 \left\Vert \rho \left( \underline{X}^{d_1} \right) \rho \left(  \underline{Z}^{d_2} \right)  \rho \left(  \underline{Y}^{d_3} \right)  \xi -  \rho \left( \underline{X}^{d_1}  \right) \rho \left(  \underline{Z}^{d_2}  Z^{d_2+1} (\mathcal{S}^{d_2+1} ) \right)     \rho \left( \underline{Y}^{d_3} \right)  \xi  \right\Vert \lesssim \\
 (d_2 + m)^{2m+1} r_2^{n}  \max_{0 \leq d \leq d_2 + n-1} \max_{\underline{t} \in \mathcal{B}^{d_2+1}} \Vert \xi - \rho (z (2^{d} \underline{t})) \xi \Vert .
\end{align*}

We observe that $\vert \mathcal{B}^{d_2} \vert = {d_2 +m  \choose d_2}$.  Fix an order on $\mathcal{B}^{d_2}$:
$$\mathcal{B}^{d_2} = \lbrace \underline{t} ' (1),  \underline{t} ' (2), ... \rbrace.$$
Thus
\begin{dmath*}
 \left\Vert \rho \left( \underline{X}^{d_1} \right) \rho \left(  \underline{Z}^{d_2} \right)  \rho \left(  \underline{Y}^{d_3} \right)  \xi -  \rho \left( \underline{X}^{d_1}  \right) \rho \left(  \underline{Z}^{d_2} Z_{d_2} (\mathcal{B}^{d_2} ) \right)     \rho \left( \underline{Y}^{d_3} \right)  \xi  \right\Vert \leq 
 \sum_{k=0}^{\vert \mathcal{B}^{d_2} \vert -1} \left\Vert \rho \left( \underline{X}^{d_1} \right) \rho \left(  \underline{Z}^{d_2} \left( \prod_{i=1}^{k} Z_{d_2} (\underline{t} ' (i)) \right)  \right)  \rho \left(  \underline{Y}^{d_3} \right)  \xi -  \rho \left( \underline{X}^{d_1}  \right) \rho \left(  \underline{Z}^{d_2}  \left( \prod_{i=1}^{k+1} Z_{d_2} (\underline{t} ' (i)) \right)  \right)  \rho \left( \underline{Y}^{d_3} \right)  \xi  \right\Vert = 
  \sum_{k=0}^{\vert \mathcal{B}^{d_2} \vert-1} \left\Vert \rho \left( \prod_{i=1}^{k} Z_{d_2} (\underline{t} ' (i)) \right) \rho \left( \underline{X}^{d_1} \right) \rho \left(  \underline{Z}^{d_2}  \right)  \rho \left(  \underline{Y}^{d_3} \right)  \xi \\
  - \rho \left( \prod_{i=1}^{k} Z_{d_2} (\underline{t} ' (i)) \right) \rho \left( \underline{X}^{d_1}  \right) \rho \left(  \underline{Z}^{d_2}   Z_{d_2} (\underline{t} ' (k+1))  \right)  \rho \left( \underline{Y}^{d_3} \right)  \xi  \right\Vert \leq^{\text{Claim } \ref{difference to linear claim}} 
  \sum_{k=0}^{\vert \mathcal{B}^{d_2} \vert -1} \left\Vert \rho \left( \underline{X}^{d_1} \right) \rho \left(  \underline{Z}^{d_2}  \right)  \rho \left(  \underline{Y}^{d_3} \right)  \xi - \rho \left( \underline{X}^{d_1}  \right) \rho \left(  \underline{Z}^{d_2}   Z_{d_2} (\underline{t} ' (k+1))  \right)  \rho \left( \underline{Y}^{d_3} \right)  \xi  \right\Vert \lesssim^{\text{Corollary } \ref{d' geq d_2 -2n coro}} \\
   \sum_{k=0}^{\vert \mathcal{B}^{d_2} \vert-1}  (d_2 + m)^m r_2^{n}  \max_{0 \leq d \leq d_2 + n-1} \max_{\underline{t} \in \mathcal{B}^{d_2+1}} \Vert \xi - \rho (z (2^{d} \underline{t})) \xi \Vert \lesssim \\
 (d_2 + m)^{2m} r_2^{n}  \max_{0 \leq d \leq d_2 + n-1} \max_{\underline{t} \in \mathcal{B}^{d_2+1}} \Vert \xi - \rho (z (2^{d} \underline{t})) \xi \Vert,
\end{dmath*}
as needed.

The proof of the second inequality is similar and we allow ourselves to omit some details.  We recall that $\vert \mathcal{S}^{d_2+1} \vert = {d_2 +m  \choose d_2 + 1}$ and fix an order on $\mathcal{S}^{d_2+1}$:
$$\mathcal{S}^{d_2+1} = \lbrace \underline{t} '' (1),  \underline{t} '' (2), ... \rbrace.$$

Thus, 
\begin{dmath*}
  \left\Vert \rho \left( \underline{X}^{d_1} \right) \rho \left(  \underline{Z}^{d_2} \right)  \rho \left(  \underline{Y}^{d_3} \right)  \xi -  \rho \left( \underline{X}^{d_1}  \right) \rho \left(  \underline{Z}^{d_2}  Z^{d_2+1} (\mathcal{S}^{d_2+1} ) \right)     \rho \left( \underline{Y}^{d_3} \right)  \xi  \right\Vert \leq 
  \sum_{k=0}^{\vert \mathcal{S}^{d_2+1} \vert-1} \sum_{d'=0}^{d_2}   \left\Vert \rho \left( \underline{X}^{d_1} \right) \rho \left(  \underline{Z}^{d_2} \right)  \rho \left(  \underline{Y}^{d_3} \right)  \xi -  \rho \left( \underline{X}^{d_1}  \right) \rho \left(  \underline{Z}^{d_2}  Z_{d'} (\underline{t} '' (k+1)) (\mathcal{S}^{d_2+1} ) \right)     \rho \left( \underline{Y}^{d_3} \right)  \xi  \right\Vert \\
   \lesssim^{\text{Corollaries } \ref{d' geq d_2 -2n coro},  \ref{d' < d_2 -2n coro}} 
  \sum_{k=0}^{\vert \mathcal{S}^{d_2+1} \vert -1} \sum_{d'=0}^{d_2} (d_2 + m)^m  r_2^{n}  \max_{0 \leq d \leq d_2 + n-1} \max_{\underline{t} \in \mathcal{B}^{d_2+1}} \Vert \xi - \rho (z (2^{d} \underline{t})) \xi \Vert \\
  \lesssim 
(d_2 + m)^{2m} r_2^{n}  \max_{0 \leq d \leq d_2 + n-1} \max_{\underline{t} \in \mathcal{B}^{d_2+1}} \Vert \xi - \rho (z (2^{d} \underline{t})) \xi \Vert,
\end{dmath*}
as needed.
\end{proof}

As a consequence,  we can bound \eqref{needed bound2} and \eqref{needed bound3}:
\begin{theorem}
\label{Final d_1,d_2,d_3 thm}
Let $d_1, d_2, d_3 \in \mathbb{N},  d_2 \geq 16$ be constants such that $\min \lbrace d_1, d_3 \rbrace \geq 16$,   $\max \lbrace d_1, d_3 \rbrace  \geq d_2$ and $\B$ be a uniformly convex Banach space $\B$ with a modulus of convexity $\delta$.  Also,  let $r = r(\delta) = \sqrt[4]{r_2 (\delta)} <1$ where $r_2 (\delta)$ is the constant of Theorem \ref{r_2 thm}.  

For every affine isometric action $\rho$ of  $ \rm H_3 (\mathbb{Z} [t_1,...,t_m])$ on $\B$ and every $\xi \in \B$ it holds that 
\begin{align*}
\left\Vert \rho \left( \underline{X}^{d_1} \right) \rho \left(  \underline{Z}^{d_2} \right)  \rho \left(  \underline{Y}^{d_3} \right)  \xi -  \rho \left( \underline{X}^{d_1}  \right) \rho \left(  \underline{Z}^{d_2+1} \right)     \rho \left( \underline{Y}^{d_3} \right)  \xi  \right\Vert \lesssim \\
(d_2 + m)^{2m} r^{\sqrt{\min \lbrace d_1, d_3, d_2  \rbrace}}  \max_{0 \leq d \leq 2d_2} \max_{\underline{t} \in \mathcal{B}^{d_2+1}} \Vert \xi - \rho (z (2^{d} \underline{t})) \xi \Vert,
 \end{align*}
 and 
\begin{align*}
\left\Vert \rho \left( \underline{Y}^{d_1} \right) \rho \left(  \underline{Z}^{d_2} \right)  \rho \left(  \underline{X}^{d_3} \right)  \xi -  \rho \left( \underline{Y}^{d_1}  \right) \rho \left(  \underline{Z}^{d_2+1} \right)     \rho \left( \underline{X}^{d_3} \right)  \xi  \right\Vert \lesssim \\
(d_2 + m)^{2m}  r^{\sqrt{\min \lbrace d_1, d_3, d_2  \rbrace}}  \max_{0 \leq d \leq 2d_2} \max_{\underline{t} \in \mathcal{B}^{d_2+1}} \Vert \xi - \rho (z (2^{d} \underline{t})) \xi \Vert .
 \end{align*} 
\end{theorem}

\begin{proof}
We will start by proving the first inequality.  Fix $\rho$ as above.  

We denote $d_1 ' = \min \lbrace d_1,  d_2 \rbrace,  d_3 ' = \min \lbrace d_3,  d_2 \rbrace$.  By Lemma \ref{f_1, f_2 lemma},  it is enough to prove that for every $\xi \in \B$ it holds that 
\begin{align*}
\left\Vert \rho \left( \underline{X}^{d_1 '} \right) \rho \left(  \underline{Z}^{d_2} \right)  \rho \left(  \underline{Y}^{d_3 '} \right)  \xi -  \rho \left( \underline{X}^{d_1 '}  \right) \rho \left(  \underline{Z}^{d_2+1} \right)     \rho \left( \underline{Y}^{d_3 '} \right)  \xi  \right\Vert \lesssim \\
(d_2 + m)^{2m} r^{\sqrt{\min \lbrace d_1, d_3, d_2  \rbrace}}  \max_{0 \leq d \leq 2d_2} \max_{\underline{t} \in \mathcal{B}^{d_2+1}} \Vert \xi - \rho (z (2^{d} \underline{t})) \xi \Vert .
 \end{align*}

Note that $d_1 ' + d_3 ' \geq d_2 + \min \lbrace d_1,  d_3,  d_2 \rbrace \geq d_2 +4$ and thus we can apply Theorem \ref{d_2 to d_2+1 thm}.  Namely,  we denote $n = \lfloor \frac{\sqrt{d_1 '+d_3 ' -d_2}}{2} \rfloor$ and note that $n \leq d_2$ and that 
$$n \geq \lfloor \frac{\sqrt{\min \lbrace d_1,  d_3,  d_2 \rbrace}}{2} \rfloor \geq^{\min \lbrace d_1,  d_3,  d_2 \rbrace \geq 16} \frac{\sqrt{\min \lbrace d_1,  d_3,  d_2 \rbrace}}{4}.$$

It follows from Theorem \ref{d_2 to d_2+1 thm} that for every $\xi \in \B$ it holds that
\begin{dmath*}
\left\Vert \rho \left( \underline{X}^{d_1 '} \right) \rho \left(  \underline{Z}^{d_2} \right)  \rho \left(  \underline{Y}^{d_3 '} \right)  \xi -  \rho \left( \underline{X}^{d_1 '}  \right) \rho \left(  \underline{Z}^{d_2+1} \right)     \rho \left( \underline{Y}^{d_3 '} \right)  \xi  \right\Vert \lesssim  \\
(d_2 + m)^{2m} r_2^{n}  \max_{0 \leq d \leq d_2+n-1} \max_{\underline{t} \in \mathcal{B}^{d_2+1}} \Vert \xi - \rho (z (2^{d} \underline{t})) \xi \Vert \leq^{\frac{\sqrt{\min \lbrace d_1,  d_3,  d_2 \rbrace}}{4} \leq n \leq  d_2} \\
(d_2 + m)^{2m} r_2^{\frac{\sqrt{\min \lbrace d_1,  d_3,  d_2 \rbrace}}{4}}  \max_{0 \leq d \leq 2 d_2} \max_{\underline{t} \in \mathcal{B}^{d_2+1}} \Vert \xi - \rho (z (2^{d} \underline{t})) \xi \Vert = \\
(d_2 + m)^{2m} r^{\sqrt{\min \lbrace d_1,  d_3,  d_2 \rbrace}}  \max_{0 \leq d \leq 2 d_2} \max_{\underline{t} \in \mathcal{B}^{d_2+1}} \Vert \xi - \rho (z (2^{d} \underline{t})) \xi \Vert.
\end{dmath*}

Next, we will show that the second inequality stated in the theorem follows from the first one.  Define $\phi : \rm H_3 (\mathbb{Z} [t_1,...,t_m]) \rightarrow \rm H_3 (\mathbb{Z} [t_1,...,t_m])$ to be the isomorphism induced by 
$$\phi (y (p)) = x (p),  \phi (x (p)) = y ( - p),  \forall p \in  \mathbb{Z} [t_1,...,t_m].$$

We note that for every $p \in \mathbb{Z} [t_1,...,t_m]$ it follows that $\phi (z (p)) = z (p)$.  Fix $\rho$ and define $\rho_0$ to by the action $\rho \circ \phi$.  By the inequality we already proven,  it is enough to prove that for every $\xi \in \B$ it holds that
\begin{align*}
\left\Vert \rho \left( \underline{Y}^{d_1} \right) \rho \left(  \underline{Z}^{d_2} \right)  \rho \left(  \underline{X}^{d_3} \right)  \xi -  \rho \left( \underline{Y}^{d_1}  \right) \rho \left(  \underline{Z}^{d_2+1} \right)     \rho \left( \underline{X}^{d_3} \right)  \xi  \right\Vert \leq \\
\left\Vert \rho_0 \left( \underline{X}^{d_1} \right) \rho_0 \left(  \underline{Z}^{d_2} \right)  \rho_0 \left(  \underline{Y}^{d_3} \right)  \xi -  \rho_0 \left( \underline{X}^{d_1}  \right) \rho_0 \left(  \underline{Z}^{d_2+1} \right)     \rho_0 \left( \underline{Y}^{d_3} \right)  \xi  \right\Vert.
\end{align*}

We note that 
$$\rho_0 (\underline{Y}^{d_3} ) = \rho (\underline{X}^{d_3}),   \rho_0 (\underline{Z}^{d_2} ) = \rho (\underline{Z}^{d_2}), \rho_0 (\underline{Z}^{d_2+1} ) = \rho (\underline{Z}^{d_2+1}).$$
Also,  for every $\underline{t} \in \mathcal{B}^{d_1}$
\begin{align*}
 \rho_0 (x (-(2^{d_1} -1) \underline{t})) \rho_0 (X^{d_1} (\underline{t} ) ) =  \rho (y ((2^{d_1} -1) \underline{t})) \left( \frac{1}{2^{d_1}} \sum_{a=0}^{2^{d_1}-1} \rho (y (- a \underline{t})) \right) = \rho (Y^{d_1} (\underline{t} )).
\end{align*}

Thus,
$$ \left( \prod_{\underline{t} \in \mathcal{B}^{d_1}} \rho_0 (x (-(2^{d_1} -1) \underline{t})) \right) \rho_0 (\underline{X}^{d_1}) = \rho (\underline{Y}^{d_1}).$$
It follows that for every $\xi \in \B$ it holds that
\begin{dmath*}
\left\Vert \rho \left( \underline{Y}^{d_1} \right) \rho \left(  \underline{Z}^{d_2} \right)  \rho \left(  \underline{X}^{d_3} \right)  \xi -  \rho \left( \underline{Y}^{d_1}  \right) \rho \left(  \underline{Z}^{d_2+1} \right)     \rho \left( \underline{X}^{d_3} \right)  \xi  \right\Vert = 
\left\Vert  \left( \prod_{\underline{t} \in \mathcal{B}^{d_1}} \rho_0 (x (-(2^{d_1} -1) \underline{t}) ) \right)  \rho_0 \left( \underline{X}^{d_1} \right) \rho_0 \left(  \underline{Z}^{d_2} \right)  \rho_0 \left(  \underline{Y}^{d_3} \right)  \xi -  \left( \prod_{\underline{t} \in \mathcal{B}^{d_1}} \rho_0 (x (-(2^{d_1} -1) \underline{t})) \right)  \rho_0 \left( \underline{X}^{d_1}  \right) \rho_0 \left(  \underline{Z}^{d_2+1} \right)     \rho_0 \left( \underline{Y}^{d_3} \right)  \xi  \right\Vert \leq^{\text{Claim } \ref{difference to linear claim}} 
\left\Vert \rho_0 \left( \underline{X}^{d_1} \right) \rho_0 \left(  \underline{Z}^{d_2} \right)  \rho_0 \left(  \underline{Y}^{d_3} \right)  \xi -  \rho_0 \left( \underline{X}^{d_1}  \right) \rho_0 \left(  \underline{Z}^{d_2+1} \right)     \rho_0 \left( \underline{Y}^{d_3} \right)  \xi  \right\Vert,
\end{dmath*}
as needed.
\end{proof}

\section{Word norm growth in the $A_2$ Steinberg group}
\label{Word norm growth of unipotents sec}

Let $\St_{A_2} (\mathbb{Z} [t_1,...,t_m])$ be the Steinberg group defined above.   We note that $\St_{A_2} (\mathbb{Z} [t_1,...,t_m])$ is finitely generated and in particular the set 
$$S = \bigcup_{1 \leq i,j \leq 3,  i \neq j} \lbrace x_{i,j} (\pm 1),  x_{i,j} (\pm t_1),...,  x_{i,j} (\pm t_m) \rbrace,$$
is a finite generating set.  

Below,  we will need the following result regarding the growth rate of the word norm in $\St_{A_2} (\mathbb{Z} [t_1,...,t_m])$:

\begin{proposition}
\label{unipotent bound prop}
Let $a \in \mathbb{Z} \setminus \lbrace 0 \rbrace,  d \in \mathbb{N} \cup \lbrace 0 \rbrace$ be constants and $\underline{t} \in \mathbb{Z} [t_1,...,t_m]$ be a monomial.  Also,  let 
$$S = \bigcup_{1 \leq i,j \leq 3,  i \neq j} \lbrace x_{i,j} (\pm 1),   x_{i,j} (\pm t_1),...,  x_{i,j} (\pm t_m) \rbrace .$$
For every $1 \leq i,j \leq 3, i \neq j$,  it holds that 
$$\vert x_{i,j} (2^d \underline{t} ) \vert_S  \lesssim (1  + (d + \deg (\underline{t}))^2)$$
and
$$\vert x_{i,j} (a \underline{t} ) \vert_S  \lesssim ( 1+ \log^3 (\vert a \vert)  + (\deg (\underline{t}))^2)$$ where $\log = \log_2$.
\end{proposition}

\begin{proof}

We will start by showing that for every $n \in \mathbb{N}$ and every $p_1,....,p_n \in \mathbb{Z} [t_1,...,t_m]$ it holds that 
\begin{equation}
\label{Cay norm prop ineq 1}
\max_{1 \leq i, j \leq 3,  i\neq j} \left\vert x_{i,j} \left( \prod_{k=1}^n p_k \right) \right\vert_S \leq 4 n^2 \left(\max_{1 \leq i, j \leq 3,  i\neq j,  1 \leq k \leq n}  \vert x_{i,j} (p_k) \vert_S \right). 
\end{equation}

We note that for every $p_1,  p_2 \in \mathbb{Z} [t_1,...,t_m]$,  it holds that 
$$\vert x_{i,j} (p_1 p_2 ) \vert_S = \vert [x_{i,k} (p_1),  x_{k,j} (p_2)] \vert_S \leq 2 (\vert x_{i,k} (p_1) \vert_S + \vert x_{k,j} (p_2) \vert_S).$$
Thus, 
\begin{equation}
\label{Cay norm prop ineq 2}
\max_{1 \leq i, j \leq 3,  i\neq j} \vert x_{i,j} (p_1 p_2 ) \vert_S \leq 2 \left( \max_{1 \leq i, j \leq 3,  i\neq j} \vert x_{i,j} (p_1 ) \vert_S + \max_{1 \leq i, j \leq 3,  i\neq j} \vert x_{i,j} (p_2 ) \vert_S \right).
\end{equation}

Using \eqref{Cay norm prop ineq 2} it follows that for every $N \in \mathbb{N}$ and every $p_1,...,p_{2^N}$ it holds that 
\begin{align*}
& \max_{1 \leq i, j \leq 3,  i\neq j} \left\vert x_{i,j} \left( \prod_{k=1}^{2^N} p_{k} \right) \right\vert_S \leq \\
& 2  \max_{1 \leq i, j \leq 3,  i\neq j} \left\vert x_{i,j} \left(\prod_{k=1}^{2^{N-1}} p_{k} \right) \right\vert_S + 2  \max_{1 \leq i, j \leq 3,  i\neq j} \left\vert x_{i,j} \left(\prod_{k=2^{N-1}+1}^{2^{N}} p_{k} \right) \right\vert_S,
\end{align*}
and by induction on $N$ it follows that
\begin{align}
\label{4^N ineq}
 \max_{1 \leq i, j \leq 3,  i\neq j} \left\vert x_{i,j} \left( \prod_{k=1}^{2^N} p_{k} \right) \right\vert_S \leq 
 \sum_{k=1}^{2^N} 2^N  \max_{1 \leq i, j \leq 3,  i\neq j} \left\vert x_{i,j} \left( p_{k} \right) \right\vert_S \leq 
4^N \max_{1 \leq i, j \leq 3,  i\neq j, 1 \leq k \leq n} \left\vert x_{i,j} \left( p_{k} \right) \right\vert_S.
\end{align}

We will show by induction on $N$ that for every $1 \leq n < 2^{N}$ and every $p_1,...,p_n \in \mathbb{Z} [t_1,....,t_n]$ it holds that 
$$\max_{1 \leq i, j \leq 3,  i\neq j} \left\vert x_{i,j} \left( \prod_{k=1}^{n} p_{k} \right) \right\vert_S \leq (2^{2N} - 2^{N}) \left(\max_{1 \leq i, j \leq 3,  i\neq j, 1 \leq k \leq n} \vert x_{i,j} (p_k ) \vert_S \right).$$
If $N =1$,  then $n =1$ and $2^{2N} - 2^{N} = 2^{2} - 2^{1}  = 2$ and the inequality holds.  Next, we will assume that the inequality holds for $N$ and prove it for $N+1$.  Let $1 \leq n < 2^{N+1}$ and $p_1,...,p_n \in \mathbb{Z} [t_1,...,t_m]$.  If $n < 2^N$,  we are done by the induction assumption.  Also,  if $n=2^N$,  then we are done by \eqref{4^N ineq} and the fact that $4^N \leq 2^{2N +2} - 2^{N+1}$.

Thus,  we can assume that $2^N+1 \leq n < 2^{N+1}$.  Denote $n ' = n-2^N$.  Then $1 \leq n ' < 2^N$ and by \eqref{Cay norm prop ineq 2} it holds that
\begin{align*}
&\max_{1 \leq i, j \leq 3,  i\neq j} \left\vert x_{i,j} \left( \prod_{k=1}^{n} p_{k} \right) \right\vert_S \leq \\
& 2  \max_{1 \leq i, j \leq 3,  i\neq j} \left\vert x_{i,j} \left( \prod_{k=1}^{2^N} p_{k} \right) \right\vert_S+ 2 \max_{1 \leq i, j \leq 3,  i\neq j} \left\vert x_{i,j} \left( \prod_{k=2^N+1}^{2^N+n'} p_{k} \right) \right\vert_S\leq \\
& \left( 2 \cdot 4^{N} + 2 (2^{2N} - 2^{N}) \right) \left( \max_{1 \leq i, j \leq 3,  i\neq j,  1 \leq k \leq n} \vert x_{i,j} (p_k ) \vert_S \right)  =  \\
&(2^{2N+2} - 2^{N+1} )\left( \max_{1 \leq i, j \leq 3,  i\neq j,  1 \leq k \leq n} \vert x_{i,j} (p_k ) \vert_S \right),
\end{align*}
as needed.  

It follows that for every $N$ and every $2^{N-1} \leq n < 2^N$ it holds that 
\begin{align*}
\max_{1 \leq i, j \leq 3,  i\neq j} \left\vert x_{i,j} \left( \prod_{k=1}^{n} p_{k} \right) \right\vert_S \leq 2^{2N} \left( \max_{1 \leq i, j \leq 3,  i\neq j,  1 \leq k \leq n} \vert x_{i,j} (p_k ) \vert_S \right) \leq 4n^2 \left( \max_{1 \leq i, j \leq 3,  i\neq j,  1 \leq k \leq n} \vert x_{i,j} (p_k ) \vert_S \right)
\end{align*}
and thus \eqref{Cay norm prop ineq 1} is proven.

As a consequence  \eqref{Cay norm prop ineq 1},  we deduce that for $\underline{t} = 2^d t_1^{n_1} ...  t_m^{n_m}$,  if $d+n_1 + .... + n_m >0$, then 
\begin{align*}
 \max_{1 \leq i, j \leq 3,  i\neq j} \vert x_{i,j} (2^d t_1^{n_1} ...  t_m^{n_m}) \vert_S \leq 8 (d+n_1+...+n_m )^2 = 8 (d+ \deg (\underline{t}))^2,
\end{align*}
which proves the first inequality stated in the Proposition.

In order to prove the second inequality stated in the Proposition,  it is enough to show that
$$\max_{1 \leq i, j \leq 3,  i\neq j} \vert x_{i,j} (a) \vert_S \leq 1 + 8 \log^3 (\vert a \vert),$$
for every $a \in \mathbb{Z} \setminus \lbrace 0 \rbrace$.   Without loss of generality,  we can assume that $a \in \mathbb{N}$.  Fix $a \in \mathbb{N}$  and $1 \leq i,  j \leq 3,  i \neq j$.  We observe that by \eqref{Cay norm prop ineq 1} it holds for every $l \in \mathbb{N}$ that 
$$\vert x_{i,j} (2^l) \vert_S \leq 4 l^2 \left( \max_{1 \leq i', j' \leq 3,  i ' \neq j '}  \vert x_{i',j'} (2) \vert_S \right) \leq 8 l^2.$$
Denote $n = \lfloor \log (a) \rfloor$.  Then $2^n \leq a < 2^{n+1}$ and there are $a_0,...,a_n \in \lbrace 0,1 \rbrace$ such that $a = \sum_l a_l 2^l$.  Thus
\begin{align*}
\vert x_{i,j} (a) \vert_S \leq \sum_{l=0}^n \vert x_{i,j} (a_l 2^l) \vert_S \leq a_0 + \sum_{l=1}^n  8 a_l l^2\leq  1+8 \sum_{l=1}^n  l^2 \leq 1+ 8 n^3 \leq 1 + 8 \log^3 (a),
\end{align*}
as needed.
\end{proof}


\begin{corollary}
\label{vert X_i,j vert_S coro}
Using the notations of \cref{Averaging operations sec},  for $i,j \in \lbrace 1,2,3 \rbrace, i \neq j$ and $d \in \mathbb{N}$,  we define $\underline{X}_{i,j}^d \in \Prob_c (\St_{A_2} (\mathbb{Z} [t_1,...,t_m])$ as  
$$\underline{X}_{i,j}^d = \frac{1}{\vert \Poly (d,d) \vert} \sum_{p \in  \Poly (d,d)} x_{i,j} (p).$$
Then for 
$$S = \bigcup_{1 \leq i,j \leq 3,  i \neq j} \lbrace x_{i,j} (\pm 1),   x_{i,j} (\pm t_1),...,  x_{i,j} (\pm t_m) \rbrace$$
it holds that 
$$\vert  \underline{X}_{i,j}^d  \vert_S \lesssim d^{m + 3}.$$
\end{corollary}

\begin{proof}
Fix $i,j \in \lbrace 1,2,3 \rbrace, i \neq j$.  It is enough to show that for every $p \in \Poly (d,d),  p \neq 0$ it holds that 
$$\vert x_{i,j} (p) \vert_S \lesssim (d+m)^{m+3}.$$

Let $p = \sum_{\underline{t} \in \mathcal{B}^d} c_{\underline{t}} \underline{t} \in \Poly (d,d) \setminus \lbrace 0 \rbrace$.  Then,
\begin{dmath*}
\vert x_{i,j} (p) \vert_S \leq  \sum_{\underline{t} \in \mathcal{B}^d} \vert x_{i,j} (c_{\underline{t}} \underline{t}) \vert_S \lesssim^{\text{Proposition } \ref{unipotent bound prop}}
\sum_{\underline{t} \in \mathcal{B}^d} (\log^3 (c_{\underline{t}} ) + \deg (\underline{t})^2 ) \leq 
\sum_{\underline{t} \in \mathcal{B}^d} (d^3 + d^2 ) \lesssim
(d+m)^{m+3}. 
\end{dmath*}
\end{proof}

\section{Relative fixed point property}
\label{Relative fixed point property sec}
In this section,  we will prove Theorem \ref{rel. f.p. thm - intro} stated in the introduction.  The heart of the proof is proving relative fixed point property for $\St_{A_2} (\mathbb{Z} [t_1,...,t_m])$ and the idea of the proof in that case is similar to that of the proofs of \cite[Theorem 6.1]{Opp-SLZ} and \cite[Theorem 4.6]{LaatSalle4}.  

\subsection{General convergence argument}
\label{General convergence argument subsec}
Let $\B$ be a Banach space, $D \in \mathbb{N}$ a constant and $F_{1},  F_2 : \bigcup_{d \in \mathbb{N}} (\mathbb{N} \cap [d,3d] )^6  \rightarrow \B,  \epsilon : (\mathbb{N} \cap [D, \infty))  \rightarrow [0, \infty)$ functions such that the following holds:
\begin{enumerate}[label=($\mathcal{A}${{\arabic*}})]
\item For every $d \in \mathbb{N},  d \geq D$ and every $(d_1,...,d_6) \in (\mathbb{N} \cap [d,3d] )^6$ such that $d_2 - (d_1+d_3) \geq d-1$ and $d_5 - (d_4+d_6) \geq d-1$ it holds that 
$$\Vert F_1 (d_1,d_2,d_3,d_4,d_5,d_6) -  F_2 (d_3,d_2,d_1,d_6,d_5,d_4) \Vert  \lesssim \epsilon (d).$$
\item For every $d \in \mathbb{N},  d \geq D$,  every $(d_1,...,d_6) \in (\mathbb{N} \cap [d,3d] )^6$ and every $2 \leq i \leq 5$,  if 
$\max \lbrace d_{i-1}, d_{i+1}  \rbrace \geq d_i$ and $d_i < 3d$,  then
$$\Vert F_1 (d_1,...,d_i,...,d_6) -  F_1 (d_1,...,d_i +1,...,d_6) \Vert \lesssim \epsilon (d),$$
and
$$\Vert F_2 (d_1,...,d_i,...,d_6) -  F_2 (d_1,...,d_i +1,...,d_6) \Vert \lesssim \epsilon (d).$$
\end{enumerate}

\begin{lemma}
\label{2 in a row lemma}
Let $d \in \mathbb{N}, d \geq D$,  $(d_1,...,d_6) \in  (\mathbb{N} \cap [d,3d] )^6$ and $d',  d'' \in   \mathbb{N} \cap [d,3d]$.

For $j =1,2$ it holds that
$$\Vert F_j (d_1,d', d',d_4,d_5,d_6) -  F_j (d_1, d'',d'',d_4,d_5,d_6) \Vert \lesssim d \epsilon (d),$$
$$\Vert F_j (d_1,d_2, d',d',d_5,d_6) -  F_j (d_1, d_2,d'',d'',d_5,d_6) \Vert \lesssim d \epsilon (d),$$
$$\Vert F_j (d_1,d_2, d_3,d',d',d_6) -  F_j (d_1, d_2,d_3,d'',d'',d_6) \Vert \lesssim d \epsilon (d).$$
\end{lemma}

\begin{proof}
The proofs of all these inequalities are similar and we will prove only the first one.  

Assume without loss of generality that $d' < d''$,  then it is enough to prove that
$$\Vert F_j (d_1,d', d',d_4,d_5,d_6) -  F_j (d_1, d'+1,d'+1,d_4,d_5,d_6) \Vert \lesssim  \epsilon (d).$$
Indeed,  by $(\mathcal{A} 2)$ it holds that 
\begin{dmath*}
\Vert F_j (d_1,d', d',d_4,d_5,d_6) -  F_j (d_1, d'+1,d'+1,d_4,d_5,d_6) \Vert \leq \\
\Vert F_j (d_1,d', d',d_4,d_5,d_6) -  F_j (d_1, d'+1,d',d_4,d_5,d_6) \Vert + \\
\Vert F_j (d_1,d'+1, d',d_4,d_5,d_6) -  F_j (d_1, d'+1,d'+1,d_4,d_5,d_6) \Vert \lesssim \epsilon (d).
\end{dmath*}

\end{proof}

\begin{corollary}
\label{4 in a row coro}
Let $d \in \mathbb{N}, d \geq D$ and $d_1, d_6,  d',  d'' \in  (\mathbb{N} \cap [d,3d] )^6$.

For $j =1,2$ it holds that
$$\Vert F_j (d_1,d', d',d' ,d' ,d_6) -  F_j (d_1, d'',d'',d'' ,d '',d_6) \Vert \lesssim d \epsilon (d).$$
\end{corollary}

Thus,  we can prove the following general convergence result:
\begin{theorem}
\label{general convergence thm}
Let $\B,  D,  F_1,  F_2,  \epsilon$ be as above and assume that $\sum_{d \in \mathbb{N}, d \geq D} d \epsilon (d) < \infty$.  Define a sequence $\lbrace \xi_k \rbrace_{k \in \mathbb{N}} \subseteq \B$ as follows:
$$\xi_k = 
\begin{cases}
F_{1} (\frac{k+1}{2},..., \frac{k+1}{2}) & k \text{ is odd} \\
F_{2} (\frac{k}{2},..., \frac{k}{2}) & k \text{ is even} 
\end{cases}.$$
Then this sequence is a Cauchy sequence in $\B$ and thus converges. 
\end{theorem}

\begin{proof}
By the assumption that $\sum_{d \in \mathbb{N}d \geq D} d \epsilon (d) < \infty$, it is enough to prove that for $\lbrace j_1,  j_2 \rbrace = \lbrace 1,2 \rbrace$ and $d \geq D$ it holds that 
$$\Vert F_{j_1} (d,...,d) - F_{j_2} (d+1,...,d+1) \Vert \lesssim d \epsilon (d).$$

Fix $\lbrace j_1,  j_2 \rbrace = \lbrace 1,2 \rbrace$ and $d \geq D$.  Then
\begin{dmath*}
\Vert F_{j_1} (d,...,d) - F_{j_2} (d+1,...,d+1) \Vert \lesssim^{\text{Corollary } \ref{4 in a row coro}} \\
d \epsilon (d) + \Vert F_{j_1} (d,3d,3d,3d,3d,d) - F_{j_2} (d+1,...,d+1) \Vert \lesssim^{\text{Lemma } \ref{2 in a row lemma}} \\
d \epsilon (d) + \Vert F_{j_1} (d,3d,d+1,d+1,3d,d) - F_{j_2} (d+1,...,d+1) \Vert  \lesssim^{(\mathcal{A} 1)} \\
 d \epsilon (d) + \Vert F_{j_2} (d+1,3d,d,d,3d,d+1) - F_{j_2} (d+1,...,d+1) \Vert  \lesssim^{\text{Lemma } \ref{2 in a row lemma}} \\
  d \epsilon (d) + \Vert F_{j_2} (d+1,3d,3d,3d,3d,d+1) - F_{j_2} (d+1,...,d+1) \Vert  \lesssim^{\text{Corollary } \ref{4 in a row coro}} 
d \epsilon (d),   
\end{dmath*}
as needed.
\end{proof}

\subsection{Relative fixed point property for $\St_{A_2} (\mathbb{Z} [t_1,...,t_m])$}

We fix $m \in \mathbb{N}$ and for $1 \leq i, j \leq 3,  i \neq j$,  we denote $K_{i,j}$ to be the $(i,j)$-root subgroup of $\St_{A_2} (\mathbb{Z} [t_1,...,t_m])$ as defined above.   Also,  for $\lbrace i,j,k \rbrace = \lbrace 1,2,3 \rbrace$,  we define 
$$H_{i,k} = \langle K_{i,j},  K_{j, k}  \rangle < \St_{A_2} (\mathbb{Z} [t_1,...,t_m ])$$
and note that $H_{i,k}$ is isomorphic to the group $\rm H_3 (\mathbb{Z} [t_1,...,t_m])$.  A key ingredient of the proof of our relative fixed point property is that we can order the root subgroups such that each three consecutive root subgroups form a Heisenberg group and the middle group in the each three consecutive  root subgroups is the center of this Heisenberg group.   Explicitly,  we define two maps $\sigma_{1}, \sigma_2 : \lbrace 1,...,6 \rbrace \rightarrow \lbrace (i_1,  i_2) : 1 \leq i_1,  i_2 \leq 3,  i_1 \neq i_2 \rbrace$ as follows:
$$\sigma_1 (1)  = (1,2), \sigma_1 (2) = (1,3),  \sigma_1 (3) = (2,3),  $$
$$\sigma_1 (4) = (2,1), \sigma_1 (5) = (3,1),   \sigma_1 (6) = (3,2),$$
and
$$\sigma_2 (1)  =  (2,3) , \sigma_2 (2) = (1,3),  \sigma_2 (3) =  (1,2),  $$
$$\sigma_2 (4) =  (3,2),\sigma_2 (5) = (3,1),   \sigma_2 (6) = (2,1).$$

We note that for every $2 \leq i \leq 5$ and every $j=1,2$ it holds that 
$$\langle K_{\sigma_{j} (i-1)},  K_{\sigma_{j} (i)},  K_{\sigma_{j} (i+1)} \rangle = H_{\sigma_{j} (i)},$$
and that $K_{\sigma_{j} (i)}$ is the center of $H_{\sigma_{j} (i)}$ (we use a slight abuse of notation here, i.e., we identify $K_{i_1, i_2} = K_{(i_1, i_2)}$ and  $H_{i_1, i_2} = H_{(i_1, i_2)}$).

Using the notations of \cref{Averaging operations sec},  for $i \in \lbrace 1, ..., 6 \rbrace,  j=1,2$ and $d \in \mathbb{N}$,  we define $\underline{X}_{\sigma_j (i)}^d \in \Prob_c (\St_{A_2} (\mathbb{Z} [t_1,...,t_m])$ as  
$$\underline{X}_{\sigma_j (i)}^d = \frac{1}{\vert \Poly (d,d) \vert} \sum_{p \in \Poly (d,d)} x_{\sigma_j (i)} (p).$$

With these notations,  we can use the results of \cref{Averaging operations sec} and prove the following:
\begin{theorem}
\label{bounding norm moves thm}
Let $\B$ be a uniformly convex Banach space with a modulus of convexity $\delta : (0,2] \rightarrow (0,1]$ and $r = r(\delta) <1$ be the constant of Theorem \ref{Final d_1,d_2,d_3 thm}.  Also,  let 
$$S = \bigcup_{1 \leq i_1, i_2 \leq 3,  i_1 \neq i_2} \lbrace x_{i_1,i_2} (\pm 1),   x_{i_1,i_2} (\pm t_1),...,  x_{i_1,i_2} (\pm t_m) \rbrace .$$
For every $d \geq \max \lbrace 16, m\rbrace$,  every $f \in \Prob_c (\St_{A_2} (\mathbb{Z} [t_1,...,t_m]))$, every affine isometric action $\rho$ of $\St_{A_2} (\mathbb{Z} [t_1,...,t_m])$ on $\B$ and every unit vector $\xi \in \B$ the following holds:
\begin{enumerate}
\item For $i=2,5$ and $d_{i-1}, d_i,  d_{i+1} \in \mathbb{N} \cap [d,3d]$,  if $d_i - (d_{i-1} + d_{i+1}) \geq d-1$, then 
\begin{dmath*}
\left\Vert \rho (\underline{X}_{\sigma_1 (i-1)}^{d_{i-1}}) \rho (\underline{X}_{\sigma_1 (i)}^{d_i})  \rho (\underline{X}_{\sigma_1 (i+1)}^{d_{i+1}}) \rho (f) \xi - \rho (\underline{X}_{\sigma_2 (i-1)}^{d_{i+1}}) \rho (\underline{X}_{\sigma_2 (i)}^{d_i})  \rho (\underline{X}_{\sigma_2 (i+1)}^{d_{i-1}}) \rho (f) \xi \right\Vert \lesssim d^m \left( \frac{1}{2} \right)^d (d^2 + \vert f \vert_S ).
\end{dmath*}
\item For $2 \leq i \leq 5$, $j =1,2$ and $d_{i-1}, d_i,  d_{i+1} \in \mathbb{N} \cap [d,3d]$,  if $\max \lbrace d_{i-1}, d_{i+1} \rbrace \geq d_i$, then 
\begin{dmath*}
\left\Vert \rho (\underline{X}_{\sigma_j (i-1)}^{d_{i-1}}) \rho (\underline{X}_{\sigma_j (i)}^{d_i})  \rho (\underline{X}_{\sigma_j (i+1)}^{d_{i+1}}) \rho (f) \xi - \rho (\underline{X}_{\sigma_j (i-1)}^{d_{i-1}}) \rho (\underline{X}_{\sigma_j (i)}^{d_i+1})  \rho (\underline{X}_{\sigma_j (i+1)}^{d_{i+1}}) \rho (f) \xi \right\Vert \lesssim d^{2m} r^{\sqrt{d}} (d^2 + \vert f \vert_S ).
\end{dmath*}
\end{enumerate}
\end{theorem}

\begin{proof}

For $2 \leq i \leq 5$ and $j=1,2$,  we note that $H_{\sigma_{j} (i)} \cong \rm H_3 (\mathbb{Z} [t_1,...,t_m])$ and that the following holds:
\begin{itemize}
\item If $i+j$ is odd,  then the isomorphism between $H_{\sigma_{j} (i)}$ and $\rm H_3 (\mathbb{Z} [t_1,...,t_m])$  is induced by
$x_{\sigma_j (i-1)} (p) \mapsto x (p),  x_{\sigma_j (i+1)} (p) \mapsto y (p), x_{\sigma_j (i)} (p) \mapsto z(p),  \forall p \in \mathbb{Z} [t_1,...,t_m]$. Moreover,  extending this isomorphism to the group ring yields that 
$$\underline{X}_{\sigma_j (i-1)}^{d_{i-1}} \mapsto \underline{X}^{d_{i-1}},  \underline{X}_{\sigma_j (i)}^{d_i} \mapsto \underline{Z}^{d_i},   \underline{X}_{\sigma_j (i+1)}^{d_{i+1}} \mapsto \underline{Y}^{d_{i+1}}.$$
\item If $i+j$ is even,  then the isomorphism between $H_{\sigma_{j} (i)}$ and $\rm H_3 (\mathbb{Z} [t_1,...,t_m])$  is induced by
$x_{\sigma_j (i-1)} (p) \mapsto y (p),  x_{\sigma_j (i+1)} (p) \mapsto x (p), x_{\sigma_j (i)} (p) \mapsto z(p),  \forall p \in \mathbb{Z} [t_1,...,t_m]$. Moreover,  extending this isomorphism to the group ring yields that 
$$\underline{X}_{\sigma_j (i-1)}^{d_{i-1}} \mapsto \underline{Y}^{d_{i-1}},  \underline{X}_{\sigma_j (i)}^{d_i} \mapsto \underline{Z}^{d_i},   \underline{X}_{\sigma_j (i+1)}^{d_{i+1}} \mapsto \underline{X}^{d_{i+1}}.$$
\end{itemize}
Below,  we will use these isomorphisms implicitly in order to apply theorems from \cref{Averaging operations sec} in our setting.  

Fix $d, f, \xi$ as above.  

\begin{enumerate}
\item Assume that $i=2,5$,  $d_{i-1},d_i, d_{i+1} \in \mathbb{N} \cap [d, 3d]$ and  $d_i - (d_{i-1} + d_{i+1}) \geq d-1$.  Then 
\begin{dmath*}
\left\Vert \rho (\underline{X}_{\sigma_1 (i-1)}^{d_{i-1}}) \rho (\underline{X}_{\sigma_1 (i)}^{d_i})  \rho (\underline{X}_{\sigma_1 (i+1)}^{d_{i+1}}) \rho (f) \xi - \rho (\underline{X}_{\sigma_2 (i-1)}^{d_{i+1}}) \rho (\underline{X}_{\sigma_2 (i)}^{d_i})  \rho (\underline{X}_{\sigma_2 (i+1)}^{d_{i-1}}) \rho (f) \xi \right\Vert \lesssim^{\text{Theorem } \ref{change of order thm}, d \geq m} \\
d^m \left( \frac{1}{2} \right)^{d-1} \max_{\underline{t} \in \mathcal{B}^{d_i}} \Vert \rho (f) \xi - \rho (x_{\sigma_1 (i)} (2^{d_i} \underline{t})) \rho (f) \xi \Vert \lesssim_S^{\text{Proposition } \ref{distance of rho (g) prop},  \Vert \xi \Vert =1} \\
d^m \left( \frac{1}{2} \right)^{d}  \left(\max_{\underline{t} \in \mathcal{B}^{d_i}} \vert x_{x_{\sigma_1 (i)}} (2^{d_i} \underline{t}) \vert_S + \vert f \vert_S  \right) \lesssim^{\text{Proposition } \ref{unipotent bound prop},  d_i \leq 3d} \\
d^m \left( \frac{1}{2} \right)^{d}  \left(d^2 + \vert f \vert_S  \right),
\end{dmath*}
as needed.
\item Assume that $2 \leq i \leq 5$,  $d_{i-1},d_i, d_{i+1} \in \mathbb{N} \cap [d, 3d]$ and  $\max \lbrace d_1, d_3 \rbrace \geq d_2$.  Then
\begin{dmath*}
\left\Vert \rho (\underline{X}_{\sigma_j (i-1)}^{d_{i-1}}) \rho (\underline{X}_{\sigma_j (i)}^{d_i})  \rho (\underline{X}_{\sigma_j (i+1)}^{d_{i+1}}) \rho (f) \xi - \rho (\underline{X}_{\sigma_j (i-1)}^{d_{i-1}}) \rho (\underline{X}_{\sigma_j (i)}^{d_i+1})  \rho (\underline{X}_{\sigma_j (i+1)}^{d_{i+1}}) \rho (f) \xi \right\Vert \lesssim^{\text{Theorem } \ref{Final d_1,d_2,d_3 thm}, \min \lbrace d_{i-1}, d_i , d_{i+1}  \rbrace \geq d \geq m} \\
d_i^{2m} r^{\sqrt{d}}  \max_{0 \leq d' \leq 2d_i} \max_{\underline{t} \in \mathcal{B}^{d_i+1}} \Vert \rho (f) \xi - \rho (x_{\sigma_j (i)} (2^{d'} \underline{t})) \rho (f) \xi \Vert \lesssim^{d_i \leq 3d} \\
d^{2m} r^{\sqrt{d}}  \max_{0 \leq d' \leq 6d} \max_{\underline{t} \in \mathcal{B}^{3d+1}} \Vert \rho (f) \xi - \rho (x_{\sigma_j (i)} (2^{d'} \underline{t})) \rho (f) \xi \Vert  \lesssim_S^{\text{Proposition } \ref{distance of rho (g) prop},  \Vert \xi \Vert =1} \\
d^{2m} r^{\sqrt{d}}  \left( \max_{0 \leq d' \leq 6d} \max_{\underline{t} \in \mathcal{B}^{3d+1}}  \vert x_{\sigma_j (i)} (2^{d'} \underline{t}) \vert_S +  \vert f \vert_S \right)  \lesssim^{\text{Proposition } \ref{unipotent bound prop},} \\
d^{2m} r^{\sqrt{d}} \left(d^2 + \vert f \vert_S \right),
\end{dmath*} 
as needed.
\end{enumerate}
  
\end{proof}

After this,  we will prove the following relative fixed point result:
\begin{theorem}
\label{relative prop f.p thm}
For every $1 \leq i,j \leq 3$,  the pair $(\St_{A_2} (\mathbb{Z} [t_1,...,t_m ],  K_{i,j})$ has relative property $(F_{\mathcal{E}_{uc}})$. 
\end{theorem}

\begin{proof}
By symmetry,  it is enough to prove that the pair $(\St_{A_2} (\mathbb{Z} [t_1,...,t_m ],  H_{1,3})$ has relative property $(F_{\mathcal{E}_{uc}})$. 

Fix $\B$ to be a uniformly convex Banach space with a modulus of convexity $\delta$.   We denote $r = r(\delta) <1$ to be the constant of Theorem \ref{Final d_1,d_2,d_3 thm} and note that $r \geq \frac{1}{2}$ (by the choice of $r_2 (\delta)$ in Theorem \ref{r_2 thm}). 

Let  $\rho$ to be an affine isometric action of $\St_{A_2} (\mathbb{Z} [t_1,...,t_m])$ on $\B$ and fix a unit vector $\xi \in \B$.  Also,  let 
$$S = \bigcup_{1 \leq i,j \leq 3,  i \neq j} \lbrace x_{i,j} (\pm 1),   x_{i,j} (\pm t_1),...,  x_{i,j} (\pm t_m) \rbrace .$$

In order to use Theorem \ref{general convergence thm},  we will define $F_1, F_2 : \mathbb{N}^6 \rightarrow \B$ as follows: For the maps $\sigma_1, \sigma_2$ define above and $j=1,2$,  we define
$$F_j (d_1,...,d_6) = \rho (\underline{X}_{\sigma_j (1)}^{d_1} ) \rho (\underline{X}_{\sigma_j (2)}^{d_2} ) \rho (\underline{X}_{\sigma_j (3)}^{d_3} ) \rho (\underline{X}_{\sigma_j (4)}^{d_4} ) \rho (\underline{X}_{\sigma_j (5)}^{d_5} ) \rho (\underline{X}_{\sigma_j (6)}^{d_6} ) \xi.$$

We further define a function $\epsilon : \mathbb{N} \rightarrow [0,\infty)$ as 
$\epsilon (d) = d^{3m+ 3} r^{\sqrt{d}} $.  We will show that for $D = \max \lbrace 16,  m \rbrace$, the functions $F_1, F_2$ fulfil the conditions $(\mathcal{A} 1)$ and $(\mathcal{A} 2)$ defined in \cref{General convergence argument subsec}.  

First,  we will verify that $(\mathcal{A} 1)$ holds.  Fix $d \geq D$ and $(d_1,...,d_6) \in (\mathbb{N} \cap [d, 3d])^6$ such that $d_2 - (d_1+d_3) \geq d-1$ and $d_5 - (d_4+d_6) \geq d-1$.  Define $F_3 : \mathbb{N}^6 \rightarrow \B$ as 
$$F_3 (d_1,...,d_6) = \rho (\underline{X}_{\sigma_2 (1)}^{d_1} ) \rho (\underline{X}_{\sigma_2 (2)}^{d_2} ) \rho (\underline{X}_{\sigma_2 (3)}^{d_3} ) \rho (\underline{X}_{\sigma_1 (4)}^{d_4} ) \rho (\underline{X}_{\sigma_1 (5)}^{d_5} ) \rho (\underline{X}_{\sigma_1 (6)}^{d_6} ) \xi.$$

We note that 
\begin{dmath*}
\Vert F_1 (d_1,d_2,d_3,d_4,d_5,d_6) - F_2 (d_3,d_2,d_1,d_6,d_5,d_4) \Vert \leq \\
\Vert F_1 (d_1,d_2,d_3,d_4,d_5,d_6) - F_3 (d_3,d_2,d_1,d_4,d_5,d_6) \Vert + \\
\Vert  F_3 (d_3,d_2,d_1,d_4,d_5,d_6) - F_2 (d_3,d_2,d_1,d_6,d_5,d_4) \Vert.
\end{dmath*}
Thus,  it is enough to prove that 
$$\Vert F_1 (d_1,d_2,d_3,d_4,d_5,d_6) - F_3 (d_3,d_2,d_1,d_4,d_5,d_6) \Vert  \lesssim \epsilon (d)$$
and 
$$\Vert  F_3 (d_3,d_2,d_1,d_4,d_5,d_6) - F_2 (d_3,d_2,d_1,d_6,d_5,d_4) \Vert \lesssim \epsilon (d).$$

For the first inequality,  we denote $f =  \underline{X}_{\sigma_1 (4)}^{d_4}  \underline{X}_{\sigma_1 (5)}^{d_5} \underline{X}_{\sigma_1 (6)}^{d_6} \in \Prob_c (\St_{A_2} (\mathbb{Z} [t_1,...,t_m])$ and note that 
$$ \vert f \vert_S  \lesssim^{\text{Corollary } \ref{vert X_i,j vert_S coro},  d_4,  d_5,  d_6 \leq 3d} d^{m+3}.$$
Then 
\begin{dmath*}
\Vert F_1 (d_1,d_2,d_3,d_4,d_5,d_6) - F_3 (d_3,d_2,d_1,d_4,d_5,d_6) \Vert = \\
\left\Vert  \rho (\underline{X}_{\sigma_1 (1)}^{d_1} ) \rho (\underline{X}_{\sigma_1 (2)}^{d_2} ) \rho (\underline{X}_{\sigma_1 (3)}^{d_3} ) \rho (f) \xi -    \rho (\underline{X}_{\sigma_2 (1)}^{d_3} ) \rho (\underline{X}_{\sigma_2 (2)}^{d_2} )  \rho (\underline{X}_{\sigma_1 (3)}^{d_1} ) \rho (f) \xi \right\Vert \lesssim^{\text{Theorem } \ref{bounding norm moves thm}} \\ 
d^m \left( \frac{1}{2} \right)^{d} \left( d^2 +  \vert f \vert_S \right) \lesssim 
\left( \frac{1}{2} \right)^{d} ( d^{m+2} + d^{2m+3} ) \lesssim \epsilon (d).  
\end{dmath*}

The proof of the bound on $\Vert  F_3 (d_3,d_2,d_1,d_4,d_5,d_6) - F_2 (d_3,d_2,d_1,d_6,d_5,d_4) \Vert$ is somewhat simpler:
\begin{dmath*}
\Vert  F_3 (d_3,d_2,d_1,d_4,d_5,d_6) - F_2 (d_3,d_2,d_1,d_6,d_5,d_4) \Vert \leq^{\text{Claim } \ref{difference to linear claim}} \\
\left\Vert \rho (\underline{X}_{\sigma_1 (4)}^{d_4} ) \rho (\underline{X}_{\sigma_1 (5)}^{d_5} ) \rho (\underline{X}_{\sigma_1 (6)}^{d_6} ) \xi - \rho (\underline{X}_{\sigma_2 (4)}^{d_6} ) \rho (\underline{X}_{\sigma_1 (5)}^{d_5} ) \rho (\underline{X}_{\sigma_1 (6)}^{d_4} ) \xi \right\Vert \lesssim^{\text{Theorem } \ref{bounding norm moves thm}} \\
d^m \left( \frac{1}{2} \right)^{d}  d^2 \leq \epsilon (d).
\end{dmath*}
Thus,  we proved that $(\mathcal{A} 1)$ holds.  

Next,  we will prove that $(\mathcal{A} 2)$ holds.  Fix $j \in \lbrace 1,2 \rbrace$,  $d \in \mathbb{N},  d \geq D$,   $(d_1,...,d_6) \in (\mathbb{N} \cap [d,3d] )^6$ and $2 \leq i \leq 5$.  Assume that $\max \lbrace d_{i-1}, d_{i+1}  \rbrace \geq d_i$.  

We denote $f \in \Prob_c (\St_{A_2} (\mathbb{Z} [t_1,...,t_m ])$ to be
$$f = \underline{X}_{\sigma_{j} (i+2)}^{d_{i+2}} ... \underline{X}_{\sigma_{j} (6)}^{d_{6}}$$
(in the case where $i=5$,  $f = \delta_e$).  We note that 
$$\vert f \vert_S \lesssim^{\text{Corollary } \ref{vert X_i,j vert_S coro}}  d^{m+3}.$$
Then,
\begin{dmath*}
\Vert F_j (d_1,...,d_i,...,d_6) -  F_j (d_1,...,d_i +1,...,d_6) \Vert \leq^{\text{Claim } \ref{difference to linear claim}} \\
\left\Vert  \rho (\underline{X}_{\sigma_j (i-1)}^{d_{i-1}})  \rho (\underline{X}_{\sigma_j (i)}^{d_{i}} ) \rho (\underline{X}_{\sigma_j (i+1)}^{d_{i+1}}) \rho (f) \xi -   \rho (\underline{X}_{\sigma_j (i-1)}^{d_{i-1}})  \rho (\underline{X}_{\sigma_j (i)}^{d_{i} +1} ) \rho (\underline{X}_{\sigma_j (i+1)}^{d_{i+1}}) \rho (f) \xi  \right\Vert \lesssim^{\text{Theorem } \ref{bounding norm moves thm}} \\
d^{2m} r^{\sqrt{d}} \left(d^2 + \vert f \vert_S \right) \lesssim 
d^{2m} r^{\sqrt{d}} (d^2 + d^{m+3}) \lesssim  d^{3m+3} r^{\sqrt{d}} = \epsilon (d).
\end{dmath*}
Thus,  we proved that $(\mathcal{A} 2)$ is fulfilled.

As in Theorem \ref{general convergence thm},  we will define 
$$\xi_k = 
\begin{cases}
F_{1} (\frac{k+1}{2},..., \frac{k+1}{2}) & k \text{ is odd} \\
F_{2} (\frac{k}{2},..., \frac{k}{2}) & k \text{ is even} 
\end{cases}.$$
By Theorem \ref{general convergence thm} it follows that this is a Cauchy sequence and thus converges.  Denote $\widetilde{\xi} = \lim \xi_k$.  We will show that 
$\widetilde{\xi} \in \B^{\rho (H_{1,3})}$ and thus  $\B^{\rho (H_{1,3})} \neq \emptyset$.   

Fix $\underline{t}_0$ to be a monomial.  For $d \geq \deg (\underline{t}_0)$ and $j=1,2$,  we denote $f_j^d \in  \Prob_c (\St_{A_2} (\mathbb{Z} [t_1,...,t_m ])$ to be the probability function such that 
$$ \underline{X}_{\sigma_j (1)}^{d}  \underline{X}_{\sigma_j (2)}^{d} \underline{X}_{\sigma_j (3)}^{d} \underline{X}_{\sigma_j (4)}^{d} \underline{X}_{\sigma_j (5)}^{d} \underline{X}_{\sigma_j (6)}^{d} = \left( \frac{1}{2^d} \sum_{a=0}^{2^d-1} x_{\sigma_j (1)} (a \underline{t}_0) \right) f_j^d .$$
Thus,
\begin{dmath*}
{\left\Vert \rho (x_{\sigma_j (1)} (\underline{t}_0))  F_j (d,....,d) - F_j (d,....,d) \right\Vert = }\\
\left\Vert \rho (x_{\sigma_j (1)} (\underline{t}_0))  \rho \left( \frac{1}{2^d} \sum_{a=0}^{2^d-1} x_{\sigma_j (1)} (a \underline{t}_0) \right)  \rho (f_j^d) \xi  - \rho \left( \frac{1}{2^d} \sum_{a=0}^{2^d-1} x_{\sigma_j (1)} (a \underline{t}_0) \right)  \rho (f_j^d) \xi \right\Vert \lesssim^{\text{Proposition } \ref{ergodic diff prop}} \\
\frac{1}{2^d} \left\Vert \rho (x_{\sigma_j (1)} (2^d \underline{t}_0)) \rho (f_j^d) \xi -  \rho (f_j^d) \xi \right\Vert \lesssim_S^{\text{Proposition } \ref{distance of rho (g) prop},  \Vert \xi \Vert =1} \\
\frac{1}{2^d} \left( \vert x_{\sigma_j (1)} (2^d \underline{t}_0)  \vert_S + \vert f_j^d \vert_S \right) \lesssim^{\text{Proposition } \ref{unipotent bound prop}, \text{Corollary } \ref{vert X_i,j vert_S coro}}
\frac{d^{m+3}}{2^d}.
\end{dmath*}

It follows that 
\begin{dmath*}
\Vert \rho ( x_{\sigma_1 (1)} (\underline{t}_0 ))  \widetilde{\xi}  - \widetilde{\xi} \Vert = 
\lim_{d \rightarrow \infty} \Vert \rho ( x_{\sigma_1 (1)} (\underline{t}_0 ))  \xi_{2d-1}  - \xi_{2d-1} \Vert = \\
\lim_{d \rightarrow \infty} \Vert  \rho (x_{\sigma_1 (1)} (\underline{t}_0))  F_1 (d,....,d) - F_1 (d,....,d) \Vert \lesssim 
\lim_{d \rightarrow \infty}  \frac{d^{m+3}}{2^d} = 0,
\end{dmath*}
and thus $\rho ( x_{\sigma_1 (1)} (\underline{t}_0 ))  \widetilde{\xi} = \widetilde{\xi}$.  This holds for every $\underline{t}_0$ and thus $\widetilde{\xi} \in \B^{\rho (K_{1,2})}$.  Similarly, 
\begin{dmath*}
\Vert \rho ( x_{\sigma_2 (1)} (\underline{t}_0 ))  \widetilde{\xi}  - \widetilde{\xi} \Vert =
\lim_{d \rightarrow \infty} \Vert \rho ( x_{\sigma_2 (1)} (\underline{t}_0 )  \xi_{2d}  - \xi_{2d} \Vert = \\
\lim_{d \rightarrow \infty} \Vert  \rho (x_{\sigma_2 (1)} (\underline{t}_0))  F_2 (d,....,d) - F_2 (d,....,d) \Vert = 0,
\end{dmath*}
and thus $\rho ( x_{\sigma_2 (1)} (\underline{t}_0 ) ) \widetilde{\xi} = \widetilde{\xi}$.  This holds for every $\underline{t}_0$ and thus $\widetilde{\xi} \in \B^{\rho (K_{2,3})}$.

We conclude that 
$$\widetilde{\xi} \in \B^{\rho (K_{1,2})} \cap \B^{\rho (K_{2,3})} = \B^{\rho (H_{1,3})}$$
as needed.
\end{proof}

\subsection{Relative fixed point property for $\St_{\Phi}$}

Below,  for a subgroup $K < \Gamma$ and $g \in \Gamma$,  we denote $K^g = g K g^{-1}$.  Also,  given a root $\gamma \in \Phi$ and a ring $R$,  we will denote 
$$K_{\gamma} (R) = \lbrace x_{\gamma} (p) : p \in R \rbrace.$$

\begin{observation}
\label{automorp obs}
We note that if the pair $(\Gamma, K)$ has relative property $(F_{\mathcal{E}_{uc}})$ and $\phi : \Gamma \rightarrow \Gamma$ is an isomorphism of $\Gamma$,  then the $(\Gamma,  \phi (K))$ has relative property $(F_{\mathcal{E}_{uc}})$.  In particular,  if the pair $(\Gamma, K)$ has relative property $(F_{\mathcal{E}_{uc}})$,  then for every $g \in \Gamma$,  the pair  $(\Gamma,  K^{g})$ has relative property $(F_{\mathcal{E}_{uc}})$. 
\end{observation}

\begin{theorem}
\label{weak relative f.p. for C2 thm}
If for some root $\gamma_0 \in C_2$,  the pair $(\St_{C_2} (\mathbb{Z} [t_1,...,t_m ]),  K_{\gamma_0})$ has relative property $(F_{\mathcal{E}_{uc}})$, then for every $\gamma \in  C_2$,  the pair $(\St_{C_2} (\mathbb{Z} [t_1,...,t_m ]),  K_{\gamma})$ has relative property $(F_{\mathcal{E}_{uc}})$.
\end{theorem}

\begin{proof}

We note that for every $\gamma , \gamma ' \in C_2$ such that $\gamma$ and $\gamma '$ are both short/long roots,  there is an automorphism $\phi : \St_{C_2} (\mathbb{Z} [t_1,...,t_m ]) \rightarrow \St_{C_2} (\mathbb{Z} [t_1,...,t_m ])$ such that $\phi (K_{\gamma}) = K_{\gamma '}$.  This implies that if $\gamma , \gamma '$ are both both short/long roots,  then,  by Observation \ref{automorp obs},  the pair $(\St_{C_2} (\mathbb{Z} [t_1,...,t_m ]),  K_{\gamma})$ has relative property $(F_{\mathcal{E}_{uc}})$ if and only if the pair $(\St_{C_2} (\mathbb{Z} [t_1,...,t_m ]) ,  K_{\gamma '})$ has relative property $(F_{\mathcal{E}_{uc}})$.

Let $\gamma_0 \in C_2$ such that the pair $(\St_{C_2} (\mathbb{Z} [t_1,...,t_m ]),  K_{\gamma_0})$ has relative property $(F_{\mathcal{E}_{uc}})$.  \\

\paragraph*{\textbf{Case 1: $\gamma_0$ is a long root}} By our argument above, it follows that for every long root $\gamma$,  the pair $(\St_{C_2} (\mathbb{Z} [t_1,...,t_m ],  K_{\gamma})$ has relative property $(F_{\mathcal{E}_{uc}})$.  We are left to show that there is a short root $\gamma$ such that the pair $(\St_{C_2} (\mathbb{Z} [t_1,...,t_m ]) ,  K_{\gamma})$ has relative property $(F_{\mathcal{E}_{uc}})$.

Explicitly,  let $\alpha$ denote a long root of $C_2$ and $\beta$ denote a short root of  $C_2$ such that the angle between $\alpha$ and $\beta$ is $\frac{3}{4} \pi$.  We already showed that $(\St_{C_2} (\mathbb{Z}[t_1,...,t_m]),  K_{\alpha})$ and $(\St_{C_2} (\mathbb{Z}[t_1,...,t_m]),  K_{\alpha + 2 \beta})$ have relative property $(F_{\mathcal{E}_{uc}})$.  We will show that the pair \\ $(\St_{C_2} (\mathbb{Z}[t_1,...,t_m]),  K_{\alpha + \beta})$ has relative property $(F_{\mathcal{E}_{uc}})$.

By Proposition \ref{C2,  G2 relations prop}, it holds for every $p \in \mathbb{Z}[t_1,...,t_m]$ that 
$$[x_{\alpha} (p),  x_\beta (1) ] = x_{\alpha + \beta} (p) x_{\alpha + 2 \beta} (p).$$
Thus, 
\begin{dmath*}
x_{\alpha + \beta} (p) = [x_{\alpha} (p),  x_\beta (1) ] x_{\alpha + 2 \beta} ( -p) = x_{\alpha} ( -p)  (x_\beta (-1) x_{\alpha} ( p) x_\beta (1) ) x_{\alpha + 2 \beta} ( -p)  \subseteq K_{\alpha} K_{\alpha}^{x_\beta (-1)}   K_{\alpha + 2 \beta},
\end{dmath*}
and it follows that 
$$K_{\alpha + \beta} \subseteq K_{\alpha} K_{\alpha}^{x_\beta (-1)} K_{\alpha + 2 \beta}$$
and in particular,  $K_{\alpha + \beta}$ is boundedly generated by $K_{\alpha}$,  $K_{\alpha}^{x_\beta (-1)}$ and $K_{\alpha + 2 \beta}$.  

Thus,  $K_{\alpha + \beta}$ is boundedly generated by $3$ subgroups that all have relative property $(F_{\mathcal{E}_{uc}})$ and it follows from Lemma \ref{bounded generation lemma} that $(\St_{C_2} (\mathbb{Z}[t_1,...,t_m]),  K_{\alpha + \beta})$ has relative property $(F_{\mathcal{E}_{uc}})$. \\

\paragraph*{\textbf{Case 2: $\gamma_0$ is a short root}} The proof of this case is similar to the proof of the previous case (but it is more involved).  By our argument above, it follows that for every short root $\gamma$,  the pair $(\St_{C_2} (\mathbb{Z} [t_1,...,t_m ]),  K_{\gamma})$ has relative property $(F_{\mathcal{E}_{uc}})$.  We are left to show that there is a long root $\gamma$ such that the pair $(\St_{C_2} (\mathbb{Z} [t_1,...,t_m ]) ,  K_{\gamma})$ has relative property $(F_{\mathcal{E}_{uc}})$.

Explicitly,  for $\alpha,  \beta$ as above,  we need to show that the pair $(\St_{C_2} (\mathbb{Z}[t_1,...,t_m]),  K_{\alpha + 2 \beta})$ has relative property $(F_{\mathcal{E}_{uc}})$.

A monomial $\underline{t} = t_1^{k_1} ... t_m^{k_m}$ will be called square-free if $k_i \in \lbrace 0,1\rbrace$ for every $1 \leq i \leq m$.  We note that for every monomial $\underline{t} = t_1^{k_1} ... t_m^{k_m}$,  there is a unique square-free monomial,  which we will denote by $\underline{t}_{\text{sf}}$ and a unique monomial $\underline{t}_{\text{rt}}$ such that $\underline{t} = \underline{t}_{\text{sf}} \underline{t}_{\text{rt}}^2$.  We note that there are exactly $2^m$ different square-free monomials.   For a square-free monomial $\underline{t} '$, we will denote $(\mathbb{Z} [t_1,...,t_m])_{\underline{t}'}$ to be polynomials that are of the form $a_1 \underline{t}_1 +... + a_n \underline{t}_n$ such that for every $1 \leq i \leq n$,   $(\underline{t}_i)_{\text{sf}} = \underline{t} '$.  We define 
$$K_{\alpha + 2 \beta} ((\mathbb{Z} [t_1,...,t_m])_{\underline{t}'} ) = \lbrace x_{\alpha + 2 \beta} (p) : p  \in (\mathbb{Z} [t_1,...,t_m])_{\underline{t}'} \rbrace,$$
and note that $K_{\alpha + 2 \beta}$ is boundedly generated by $K_{\alpha + 2 \beta} ((\mathbb{Z} [t_1,...,t_m])_{\underline{t}'} ), \underline{t} ' \text{ square-free}$.  Thus,  it is enough to show that for every square-free $\underline{t} '$,  it holds that the pair \\
$(\St_{C_2} (\mathbb{Z}[t_1,...,t_m]),  K_{\alpha + 2 \beta} ((\mathbb{Z} [t_1,...,t_m])_{\underline{t}'} ))$ has relative property $(F_{\mathcal{E}_{uc}})$.

Fix  $\underline{t} '$ to be a square-free monomial.  By Proposition \ref{C2,  G2 relations prop},  it holds for every $p_1 \in \mathbb{Z}[t_1,...,t_m]$ that 
$$[x_{\alpha + \beta} (p_1),  x_{\beta} (1) ] = x_{\alpha + 2 \beta} (2 p_1)$$
and thus $K_{\alpha + 2 \beta} (2 \mathbb{Z} [t_1,...,t_m] )$ is boundedly generated by $K_{\alpha + \beta}$ and $K_{\beta}$.  

Also,  by Proposition \ref{C2,  G2 relations prop},  it holds for every $p_2 \in  \mathbb{Z}[t_1,...,t_m]$ that 
$$[x_{\alpha} (\underline{t} ' ),  x_\beta (p_2) ] = x_{\alpha + \beta} (\underline{t} ' p_2) x_{\alpha + 2 \beta} (\underline{t} ' p_2^2).$$
Thus,  the set 
$$\lbrace  x_{\alpha + 2 \beta} (\underline{t} ' p_2^2) : p_2 \in \mathbb{Z}[t_1,...,t_m] \rbrace$$ 
is boundedly generated by $K_{\beta},  K_{\alpha + \beta}$ and $K_\beta^{x_{\alpha} (- \underline{t} ') }$.  

It follows that it is enough to show that for every $p \in (\mathbb{Z} [t_1,...,t_m])_{\underline{t}'}$, there are $p_1,  p_2 \in \mathbb{Z}[t_1,...,t_m]$ such that $p = 2 p_1 + \underline{t} ' p_2^2$.  Indeed,  if such $p_1,  p_2$ exist for every $p \in (\mathbb{Z} [t_1,...,t_m])_{\underline{t}'}$, then it follows that $K_{\alpha + 2 \beta} ((\mathbb{Z} [t_1,...,t_m])_{\underline{t}'} )$ is boundedly generated by $K_{\beta},  K_{\alpha + \beta}$ and $K_\beta^{x_{\alpha} (- \underline{t} ') }$ and by Lemma \ref{bounded generation lemma},  the pair $(\St_{C_2} (\mathbb{Z}[t_1,...,t_m]),  K_{\alpha + 2 \beta} ((\mathbb{Z} [t_1,...,t_m])_{\underline{t}'} ))$ has relative property $(F_{\mathcal{E}_{uc}})$.

Fix $p \in  (\mathbb{Z} [t_1,...,t_m])_{\underline{t}'}$.  We note that there are monomials $\underline{t}_1,..., \underline{t}_n \in (\mathbb{Z} [t_1,...,t_m])_{\underline{t}'}$ such that $p + 2 \mathbb{Z} [t_1,...,t_m] = \sum_{i=1}^{n} \underline{t}_i + 2 \mathbb{Z} [t_1,...,t_m]$.  As noted above,  for every $i$, there is a monomial $(\underline{t}_i)_{\text{rt}}$ such that $\underline{t}_i = \underline{t} ' (\underline{t}_i)_{\text{rt}}^2$.  Thus, 
\begin{dmath*}
p + 2 \mathbb{Z} [t_1,...,t_m] =  \sum_{i=1}^{n} \underline{t}_i + 2 \mathbb{Z} [t_1,...,t_m]=
\underline{t}' \sum_{i=1}^{n} (\underline{t}_i)_{\text{rt}}^2 + 2 \mathbb{Z} [t_1,...,t_m] =
\underline{t}' \left( \left( \sum_{i=1}^{n} (\underline{t}_i)_{\text{rt}} \right)^2 - 2 \sum_{1 \leq i < j \leq n} (\underline{t}_i)_{\text{rt}} (\underline{t}_j)_{\text{rt}} \right)+ 2 \mathbb{Z} [t_1,...,t_m] =
\underline{t}' \left(\sum_{i=1}^{n} (\underline{t}_i)_{\text{rt}} \right)^2 + 2 \mathbb{Z} [t_1,...,t_m].
\end{dmath*}
It follows that there are $p_1, p_2 \in \mathbb{Z} [t_1,...,t_m]$ such that $p = 2 p_1 + \underline{t} ' p_2^2$ as needed. 
\end{proof}

The proof for relative fixed point property in the case where $\Phi = G_2$ appeared  in \cite{EJZK} in the context of property (T). The adaptation of the proof of \cite{EJZK} to our setting is straight-forward and we claim no originality here.

\begin{theorem}
\label{relative f.p. for G2 thm}
For every root $\gamma \in G_2$,  the pair $(\St_{G_2} (\mathbb{Z} [t_1,...,t_m ],  K_\gamma)$ has relative property $(F_{\mathcal{E}_{uc}})$. 
\end{theorem}

\begin{proof}

Using the notations of Proposition \ref{C2,  G2 relations prop}, we note that $K_{\lbrace \pm \alpha, \pm (\alpha + 3 \beta ),  \pm ( 2\alpha + 3 \beta ) \rbrace}  \cong \St_{A_2} (\mathbb{Z} [t_1,...,t_m])$.  By Theorem \ref{relative prop f.p thm},  we deduce that for every long root $\gamma \in G_2$,  it holds that pair  $(\St_{G_2} (\mathbb{Z}[t_1,...,t_m]), K_\gamma)$ has relative property $(F_{\mathcal{E}_{uc}})$.  Thus, we are left to prove relative property $(F_{\mathcal{E}_{uc}})$ for subgroups of short roots in $G_2$.  By symmetry,  it is enough to show that the pair $(\St_{G_2} (\mathbb{Z}[t_1,...,t_m]), K_{\alpha + 2 \beta})$ has relative property $(F_{\mathcal{E}_{uc}})$.
 
By Proposition \ref{C2,  G2 relations prop},  for every $p_1,  p_2 \in \mathbb{Z}[t_1,...,t_m]$ it holds that
$$[x_{\alpha} (p_1),  x_{\beta} (p_2) ] = x_{\alpha + \beta} (p_1 p_2) x_{\alpha + 2 \beta} (p_1 p_2^2) x_{\alpha + 3 \beta} (p_1 p_2^3) x_{2 \alpha + 3 \beta} (p_1^2 p_2^3),$$ 
and
$$[x_{\alpha} (p_1 p_2),  x_{\beta} (1) ] = x_{\alpha + \beta} (p_1 p_2) x_{\alpha + 2 \beta} (p_1 p_2) x_{\alpha + 3 \beta} (p_1 p_2) x_{2 \alpha + 3 \beta} (p_1^2 p_2^2).$$ 
Thus,
\begin{dmath*}
x_{\alpha + 2 \beta} (p_1 (p_2^2 - p_2))  \in [x_{\alpha} (p_1 p_2),  x_{\beta} (1) ]^{-1} [x_{\alpha} (p_1),  x_{\beta} (p_2) ] K_{\alpha + 3 \beta} K_{2 \alpha + 3 \beta} \subseteq 
K_\alpha^{x_{\beta} (1)} K_\alpha K_\alpha^{x_{\beta} (- p_2)} K_{\alpha + 3 \beta} K_{2 \alpha + 3 \beta}.  
\end{dmath*}

Taking $p_1 = p \in \mathbb{Z}[t_1,...,t_m]$ to be a general polynomial and $p_2 = 2,  t_1,  t_2,...,t_m$,  we deduce that 
$$K_{\alpha + 2 \beta} (2  \mathbb{Z}[t_1,...,t_m]) \subseteq K_\alpha^{x_{\beta} (1)} K_\alpha K_\alpha^{x_{\beta} (- 2)} K_{\alpha + 3 \beta} K_{2 \alpha + 3 \beta},$$
$$K_{\alpha + 2 \beta} ((t_i^2 - t_i)  \mathbb{Z}[t_1,...,t_m]) \subseteq K_\alpha^{x_{\beta} (1)} K_\alpha K_\alpha^{x_{\beta} (- t_i )} K_{\alpha + 3 \beta} K_{2 \alpha + 3 \beta},  \forall 1 \leq i \leq m.$$

Denote 
$$ K_{\alpha + 2 \beta} ' = K_{\alpha + 2 \beta} \left(2  \mathbb{Z}[t_1,...,t_m] + \sum_{i=1}^m (t_i^2 - t_i)  \mathbb{Z}[t_1,...,t_m] \right).$$ 
By the computations above,  $K_{\alpha + 2 \beta} '$ is boundedly generated by $K_\alpha,   K_{\alpha + 3 \beta}, K_{2 \alpha + 3 \beta},  K_\alpha^{x_{\beta} (1)},  K_\alpha^{x_{\beta} (- 2)}$ and $K_\alpha^{x_{\beta} (- t_i )},  i=1,...,m$.  It follows from Lemma \ref{bounded generation lemma} that $(\St_{G_2} (\mathbb{Z}[t_1,...,t_m]), K_{\alpha + 2 \beta} ')$ has relative property $(F_{\mathcal{E}_{uc}})$. We note that $K_{\alpha + 2 \beta} '$ is a finite index subgroup of $K_{\alpha + 2 \beta}$ (explicitly,  $K_{\alpha + 2 \beta} / K_{\alpha + 2 \beta} ' = \oplus_{i=1}^{2^m} (\mathbb{Z} / 2 \mathbb{Z})$) and thus it follows that $(\St_{G_2} (\mathbb{Z}[t_1,...,t_m]), K_{\alpha + 2 \beta} )$ has relative property $(F_{\mathcal{E}_{uc}})$ as needed.
\end{proof}

Combining these results and Theorem \ref{relative prop f.p thm} give rise to the following result that appeared in the introduction:
\begin{theorem} 
\label{final relative f.p. thm}
Let $\Phi \neq C_2$ be a classical,  reduced,  irreducible root system of rank $\geq 2$ and $R$ a finitely generated commutative (unital) ring.  For every $\alpha \in \Phi$,  the pair $(\St_{\Phi} (R),  K_{\alpha})$ has relative property $(F_{\mathcal{E}_{uc}})$.
\end{theorem}

\begin{proof}
Fix $\Phi$ and $R$ as above.

We note that there is  an integer $m \in \mathbb{N}$ such that there is a ring epimorphism $\phi : \mathbb{Z} [t_1,...,t_m] \rightarrow R$.  This ring epimorphism induces a group epimorphism $\widetilde{\phi} : \St_{\Phi} ( \mathbb{Z} [t_1,...,t_m]) \rightarrow \St_{\Phi} (R)$ such that every $\alpha \in \Phi$,  $\widetilde{\phi} (K_{\alpha} (\mathbb{Z} [t_1,...,t_m] )) = K_\alpha (R)$.   Thus,  it is enough to prove that for every $\alpha \in \Phi$,  the pair $( \St_{\Phi} ( \mathbb{Z} [t_1,...,t_m]),  K_{\alpha})$ has relative property $(F_{\mathcal{E}_{uc}})$. \\

\paragraph*{\textbf{Case 1: $\Phi$ is simply laced,  $\Phi = F_4$ or $\Phi = G_2$}} In this case,  every root $\alpha$ of $\Phi$ is either in an $A_2$ subsystem or in a $G_2$ system (if $\Phi = G_2$) and thus by Theorems \ref{relative prop f.p thm},  \ref{relative f.p. for G2 thm},  for every $\alpha \in \Phi$,  the pair  $(\St_{\Phi} (\mathbb{Z}[t_1,...,t_m]), K_\alpha)$ has relative property $(F_{\mathcal{E}_{uc}})$.  \\

\paragraph*{\textbf{Case 2: $\Phi = B_n,  n \geq 3$ or $\Phi = C_n,  n \geq 3$}}   In this case,  the only roots that are not in an $A_2$ subsystem are the short roots in the $C_2$ subsystem (if $\Phi = B_n$) or  the long roots in the $C_2$ subsystem (if $\Phi = C_n$).  By Theorem \ref{weak relative f.p. for C2 thm}, it follows that for every $\alpha$ the pair  $(\St_{\Phi} (\mathbb{Z}[t_1,...,t_m]), K_\alpha)$ has relative property $(F_{\mathcal{E}_{uc}})$. 
\end{proof}

\section{General synthesis machinery}
\label{General synthesis section}

The aim of this section is to develop a general machinery that allows one to deduce property $(F_{\mathcal{E}_{uc}})$ for a group $\Gamma$ from relative property $(F_{\mathcal{E}_{uc}})$ given a certain subgroup structure of $\Gamma$.  This machinery was inspired by the work of Mimura \cite{Mimura4} (which in turn was inspired by the work of Shalom \cite{Shalom2}), but our machinery is much more general than that of \cite{Mimura4}.

\subsection{Ultraproduct of Banach spaces and group actions}

We recall that an \textit{ultrafilter} $\mathcal{U}$ on $\mathbb{N}$ is a non-empty collections of subsets of $\mathbb{N}$ such that the following holds:
\begin{itemize}
\item $\emptyset \notin \mathcal{U}$. 
\item If $B \in \mathcal{U}$ and $B \subseteq A$, then $A \in \mathcal{U}$.
\item If $A, B \in \mathcal{U}$, then $A \cap B \in \mathcal{U}$.
\item If $A \notin \mathcal{U}$, then $\mathbb{N} \setminus A \in \mathcal{U}$.
\end{itemize}

An ultrafilter $\mathcal{U}$ on $\mathbb{N}$ is called \textit{non-principal} if for every $n \in \mathbb{N}$,  $\lbrace n \rbrace \notin \mathcal{U}$.  The existence of non-principal ultrafilters is proven via the axiom of choice.  In the discussion below, we fix $\mathcal{U}$ to be a non-principal ultrafilter.  

Given a sequence of real numbers $\lbrace a_n \rbrace_{n \in \mathbb{N}}$,  the limit of the sequence along the ultrafilter $\mathcal{U}$ is denoted by $\lim_{n \rightarrow \mathcal{U}} a_n$ and defined to be a number $a \in \mathbb{R}$ such that for every $\varepsilon >0$,  
$$\lbrace n : \vert a - a_n \vert < \varepsilon \rbrace \in \mathcal{U}.$$
One can show that if $\lbrace a_n \rbrace_{n \in \mathbb{N}}$ is a bounded sequence, then this limit always exists and it is unique.  

Given a sequence of pointed Banach spaces $\lbrace \B_n \rbrace_{n \in \mathbb{N}}$,  the $\ell^{\infty}$ product $\ell^{\infty} (\mathbb{N}, \lbrace \B_n \rbrace_{n \in \mathbb{N}})$ is the space of bounded sequence, i.e.,
$$\ell^{\infty} (\mathbb{N}, \lbrace \B_n \rbrace_{n \in \mathbb{N}}) = \lbrace \lbrace \xi_n \in \B_n \rbrace_{n \in \mathbb{N}} : \sup_{n} \Vert \xi_n \Vert < \infty \rbrace.$$
The \textit{ultraproduct} of the sequence  $\lbrace \B_n \rbrace_{n \in \mathbb{N}}$ is defined as follows: we define
$$N_{\mathcal{U}} = \lbrace \lbrace \xi_n \rbrace_{n \in \mathbb{N}} \in \ell^{\infty} (\mathbb{N}, \lbrace \B_n \rbrace_{n \in \mathbb{N}}) : \lim_{n \rightarrow \mathcal{U}} \Vert \xi_n \Vert = 0 \rbrace,$$
and the ultraproduct is defined as 
$$\lbrace \B_n \rbrace_{\mathcal{U}} = \ell^{\infty} (\mathbb{N}, \lbrace \B_n \rbrace_{n \in \mathbb{N}}) / N_{\mathcal{U}},$$
with the norm
$$\Vert \lbrace \xi_n \rbrace_{n \in \mathbb{N}} \Vert = \lim_{n \rightarrow \mathcal{U}} \Vert \xi_n \Vert .$$

We note that $\lbrace \B_n \rbrace_{\mathcal{U}}$ is a Banach space and for a class of Banach spaces $\mathcal{E}$,  we say that $\mathcal{E}$ is \textit{closed under passing to ultraproducts} if for every sequence $\lbrace \B_n \rbrace_{n \in \mathbb{N}} \subseteq \mathcal{E}$ it holds that $\lbrace \B_n \rbrace_{\mathcal{U}} \in \mathcal{E}$.  

Let $\Gamma$ be a finitely generated group and fix a finite generating set $S$.  For an affine isometric action $\rho$ of $\Gamma$ on a Banach space $\B$,  we define the displacement function $\disp_{\rho,  S} : \B \rightarrow [0, \infty)$ as 
$$\disp_{\rho,  S} (\xi) = \max_{s \in S} \Vert \xi - \rho (s) \xi \Vert. $$
We note that the action $\rho$ has a fixed point if and only if there is $\xi_0 \in \B$ such that $\disp_{\rho,  S} (\xi_0) =0$.  

Let $\lbrace \B_n \rbrace_{n \in \mathbb{N}}$ be a sequence of Banach spaces and a sequence of affine isometric actions $\rho_n$ of $\Gamma$ on $\B_n$ such that
$$\sup_{n} \disp_{\rho_n,  S} (0_n) < \infty,$$
where $0_n \in \B_n$ denotes the origin of $\B_n$.  In this setting,  there is an affine isometric action $\rho_{\mathcal{U}}$ of $\Gamma$ on the ultraproduct $\lbrace \B_n \rbrace_{\mathcal{U}}$ induced by the action of $\Gamma$ on $ \ell^{\infty} (\mathbb{N}, \lbrace \B_n \rbrace_{n \in \mathbb{N}})$.  

An affine isometric action $\rho$ will be called \textit{uniform with respect to $S$} (or \textit{$S$-uniform}) if
$$\inf_{\xi \in \B} \disp_{\rho,  S} (\xi) \geq 1 .$$

Given a class of Banach space $\mathcal{E}$,  we define the class of uniform $S$-actions as 
$$\mathcal{C}_{\mathcal{E}}^{S-\text{uniform}} = \lbrace (\B, \rho) : \B \in \mathcal{E},  \rho \text{ is uniform with respect to } S \rbrace.$$

The class $\mathcal{C}_{\mathcal{E}}^{S-\text{uniform}}$ can be thought of as a class of actions that are ``uniformly far'' from having a fixed point.  What will be important to us in the discussion below is that if $\mathcal{E}$ is closed under passing to ultraproducts,  then so is $\mathcal{C}_{\mathcal{E}}^{S-\text{uniform}}$,  i.e.,  if $\mathcal{E}$ is closed under passing to ultraproducts,  then given a sequence $\lbrace (\B_n,  \rho_n) \rbrace_{n \in \mathbb{N}} \subseteq \mathcal{C}_{\mathcal{E}}^{S-\text{uniform}}$ with 
$$\sup_{n} \disp_{\rho_n,  S} (0_n) < \infty,$$
then $(\lbrace \B_n \rbrace_{\mathcal{U}},  \rho_{\mathcal{U}}) \in \mathcal{C}_{\mathcal{E}}^{S-\text{uniform}}$. 

The following result is proved (in much greater generality) in \cite{Stalder} (where it is attributed to Gromov):
\begin{theorem} \cite[Theorem 1.5]{Stalder}
\label{S-uniform thm}
Let $\Gamma$ be a finitely generated group with a finite generating set $S$ and $\mathcal{E}$ be a class of Banach spaces that is closed under passing to ultraproducts.  If are $\B \in \mathcal{E}$ and an affine isometric action $\rho$ of $\Gamma$ on $\B$ such that $\B^{\rho (\Gamma)} = \emptyset$,  then 
$\mathcal{C}_{\mathcal{E}}^{S-\text{uniform}} \neq \emptyset$.  

In other words,  if $\mathcal{E}$ is closed under passing to ultraproducts and there is an action of $\Gamma$ on some $\B \in \mathcal{E}$ that does not admit a fixed point,  then there is an $S$-uniform action of $\Gamma$ on a space in $\mathcal{E}$.
\end{theorem}

\subsection{Our synthesis machinery}

We start with two results which were proven by Shalom \cite{Shalom2} in the Hilbert setting and by Mimura \cite{Mimura1} in the Banach setting (we include the proofs for completeness).

\begin{theorem}
\label{N fixed xi^i thm}
Let $\Gamma$ be a discrete group with finite Abelianization and $N, H_{+},  H_{-} < \Gamma$ be subgroups such that $N$ normalizes both $H_{+}$ and $H_{-}$ and $H_{+},  H_{-}$ generate $\Gamma$.  Also,  let $\B$ be a uniformly convex Banach space such that the pairs $(\Gamma,  H_{+}),  (\Gamma, H_{-})$ have relative property $(F_{\B})$.  Given an affine isometric action $\rho$ of $\Gamma$ denote 
$$D_\rho =  \inf_{\xi^{+} \in \B^{\rho (H_{+})}, \xi^{-}  \in \B^{\rho (H_{+})}} \Vert \xi^{+}  - \xi^{-}  \Vert.$$

Assume that there are $\xi^{+} \in \B^{\rho (H_{+})},  \xi^{-} \in \B^{\rho (H_{-})}$ such that $\Vert \xi^{+} - \xi^{-} \Vert = D_\rho$.  Then $\xi^{+} \in \B^{\rho (\langle H_{+}, N \rangle)}$ and $\xi^{-} \in \B^{\rho (\langle H_{-}, N \rangle)}$.
\end{theorem}

\begin{proof}


Let $\pi$ denote the linear part of $\rho$ and $c$ denote its cocycle.  By \cite[Proposition 2.6]{BFGM}, there are $\pi$-invariant subspaces $\B^{\pi (\Gamma)}, \B'$ such that $\B = \B^{\pi (\Gamma)} \oplus \B'$.   We decompose the cocycle $c = c_0 + c'$ where $c_0 : \Gamma \rightarrow  \B^{\pi (\Gamma)}$ and $c ' : \Gamma \rightarrow \B '$ (note that $c_0, c'$ are cocycles).  It follows that $c_0 : \Gamma \rightarrow \B^{\pi (\Gamma)}$ is a group homomorphism between $\Gamma$ and $\B^{\pi (\Gamma)}$.  Since $\Gamma$ has a finite Abelianization,  it follows that $c_0 \equiv 0$.  We will first show that for every $g \in \Gamma$ and every $\xi \in \B$, it holds that
$$\rho (g) \xi - \xi \in \B'.$$

Indeed,  fix some $\xi \in \B$ and decompose $\xi = \xi_0 + \xi '$,  where $\xi_0 \in \B^{\rho (\Gamma)}$ and $\xi ' \in \B '$.  Then 
\begin{dmath*}
\rho (g) \xi - \xi = \pi (g) \xi + c(g) - \xi = 
\pi (g) (\xi_0 + \xi ') +  c '(g) - (\xi_0 + \xi ')= 
\xi_0 + \pi (g) \xi ' + c ' (g) - (\xi_0 + \xi ') = {\pi (g) \xi ' + c' (g) - \xi ' \in \B' } .
\end{dmath*}

We need to prove that for every $g \in N$ it holds that $\rho (g) \xi^\pm = \xi^\pm$.  By the fact that $\rho (g) \xi^{\pm} - \xi^{\pm} \in \B'$,  this is equivalent to showing that for every $g \in N$,  it holds that $\rho (g) \xi^\pm - \xi^\pm \in \B^{\pi (\Gamma)}$.

Fix $g \in N$.  We assumed that $N$ normalizes $H_{+}, H_{-}$ and thus $\rho (g) \xi^\pm \in \B^{\rho (H_\pm)}$.  By the fact that $\rho$ is an isometry, it follows that 
$$\Vert \rho (g) \xi^{+} - \rho (g) \xi^{-} \Vert = \Vert \xi^{+} - \xi^{-} \Vert = D_\rho.$$

We note that if $\Vert (\rho (g) \xi^{+} - \rho (g) \xi^{-}) - (\xi^{+} - \xi^{-}) \Vert >0$, then by strict convexity it follows that 
$$\Vert \frac{1}{2} (\rho (g) \xi^{+} - \rho (g) \xi^{-}) + \frac{1}{2} (\xi^{+} - \xi^{-}) \Vert < D_\rho$$
and thus for $\eta^\pm = \frac{\xi^\pm + \rho (g) \xi^\pm}{2} \in \B^{\rho (H_\pm)}$ it holds that 
$$\Vert \eta^+ - \eta^- \Vert <  D_\rho,$$
which is a contradiction to the definition of $D_\rho$.  Thus 
$$(\rho (g) \xi^{+} - \rho (g) \xi^{-}) - (\xi^{+} - \xi^{-}) = 0$$
and it follows that 
$$\rho (g) \xi^{+} - \xi^{+} = \rho (g) \xi^{-} - \xi^{-}.$$
We note that $\rho (g) \xi^{+},  \xi^{+} \in \B^{\rho (H_{+})}$ and thus $\rho (g) \xi^{+} - \xi^{+} \in \B^{\pi (H_{+})}$.  Similarly,  $\rho (g) \xi^{-} - \xi^{-} \in \B^{\pi (H_{-})}$.  Thus from the equality $\rho (g) \xi^{+} - \xi^{+} = \rho (g) \xi^{-} - \xi^{-}$,  it follows that
$$\rho (g) \xi^\pm - \xi^\pm \in \B^{\pi (H_{+})} \cap \B^{\pi (H_{-})} = \B^{ \pi (\langle H_{+}, H_{-} \rangle)} = \B^{\pi (\Gamma)}$$
as needed. 
\end{proof}

\begin{theorem}
\label{D is realized thm}
Let $\Gamma$ be a finitely generated group with a finite generating set $S$ and $H_{+},  H_{-} < \Gamma$ subgroups that generate $\Gamma$. Also,  let $\mathcal{E}$ be a class of Banach spaces that is closed under passing to ultraproducts such that $\mathcal{C}_{\mathcal{E}}^{S-\text{uniform}} \neq \emptyset$.  Assume that the pairs $(\Gamma,  H_+)$,  $(\Gamma,  H_-)$ have relative property $(F_{\mathcal{E}})$.  
For every $(\B,  \rho) \in \mathcal{C}_{\mathcal{E}}^{S-\text{uniform}}$ we denote
$$D_\rho =  \inf_{\xi^{+}  \in \B^{\rho (H_{+})}, \xi^{-}  \in \B^{\rho (H_{-})}} \Vert \xi^{+}  - \xi^{-}  \Vert$$
and further denote
$$D = \inf_{(\B,  \rho) \in \mathcal{C}_{\mathcal{E}}^{S-\text{uniform}}} D_\rho.$$

Then there is $(\B,  \rho) \in \mathcal{C}_{\mathcal{E}}^{S-\text{uniform}}$ and $\xi^{\pm}  \in \B^{\rho (H_{\pm})}$ such that 
$\Vert \xi^{+} - \xi^{-} \Vert = D$.
\end{theorem}

\begin{proof}
Let $\lbrace (\B_n, \rho_n) \rbrace_{n \in \mathbb{N}} \subseteq \mathcal{C}_{\mathcal{E}}^{S-\text{uniform}}$ such that for every $n$,  $D_{\rho_n} \leq D+\frac{1}{2n}$.  Thus,  for every $n$,  there are $\xi^{+}_n  \in \B^{\rho_n (H_{+})}, \xi^{-}_n  \in \B^{\rho_n (H_{-})}$ such that 
$$\Vert \xi^{+}_n - \xi^{-}_n \Vert \leq D + \frac{1}{n} \leq D+1.$$
By translating the origin of $\B_n$, we can assume that for every $n$,  $\xi^{+}_n = 0_n$.  

Denote $\B_{\mathcal{U}} = \lbrace \B_n \rbrace_{\mathcal{U}}$ to be the ultraproduct.  We note that the sequences $\lbrace \xi^{+}_n \rbrace_{n \in \mathbb{N}}$ and $\lbrace \xi^{-}_n \rbrace_{n \in \mathbb{N}}$ are in $\ell^{\infty} (\mathbb{N},  \lbrace \B_n \rbrace_{n \in \mathbb{N}})$,  i.e., that these are bounded sequences.  Indeed,  we translated the origin such that for every $n$,  $\Vert \xi^{+}_n  \Vert = 0$ and 
$$\Vert \xi^{-}_n \Vert =  \Vert \xi^{-}_n - \xi^{+}_n \Vert \leq D +1.$$
Denote $\xi^{+}_\infty$ and $\xi^{-}_\infty$ to be the $\mathcal{U}$-limits of these sequences, then $\Vert \xi^{+}_\infty - \xi^{-}_\infty \Vert = D$.  We are left to check that the action $\rho_{\mathcal{U}}$ on $\B_{\mathcal{U}}$ is well-defined (and it will readily follow that $\xi^{+}_\infty  \in \B_{\mathcal{U}}^{\rho_{\mathcal{U}} (H_{+})}$ and  $\xi^{-}_\infty  \in \B_{\mathcal{U}}^{\rho_{\mathcal{U}} (H_{-})}$). Thus, it is left to verify that for every $g \in S$,  
$$\sup_n \Vert  \rho_n (g) 0_n \Vert < \infty.$$

Fix $g \in S$.   We will show that for every $h \in H_{+} \cup H_{-}$ it holds for every $n$ that 
$$\Vert \rho_n (h) 0_n  \Vert \leq 2 (D+1).$$
If $h \in H_{+}$,  then $\rho_n (h) 0_n = 0_n$ and thus $\Vert \rho_n (h) 0_n  \Vert =0$.  If $h \in H_{-}$,  then
\begin{dmath*}
\Vert \rho_n (h) 0_n  \Vert \leq \Vert \rho_n (h) 0_n - \rho_n (h) \xi^{-}_n  \Vert + \Vert \rho_n (h) \xi^{-}_n \Vert = 
2 \Vert \xi^{-}_n \Vert \leq 2 (D+1).
\end{dmath*}

Thus, by Proposition \ref{distance of rho (g) prop},  it follows that 
$$\sup_{n \in \mathbb{N}} \Vert \rho_n (g) 0_n \Vert \leq 2 (D+1) \vert g \vert_{H_{+} \cup H_{-}}$$
as needed.
\end{proof}

After this preparation,  we prove a general criterion for property $(F_{\mathcal{E}})$:
\begin{theorem}
\label{general reduction thm - directed graph}
Let $\Gamma$ be a finitely generated group with a finite Abelianization and $\mathcal{E}$ a class of uniformly convex Banach spaces that is closed under passing to ultraproducts.  Also,  let $\vec{\mathcal{G}}$ be a directed graph with a vertex set $V$ such that for every $u,v \in V,  u \neq v$,  there are directed paths from $u$ to $v$ and from $v$ to $u$.  Assume that for every $v \in V$ there are subgroups $N^{v},  H_{+}^{v},  H_{-}^{v} < \Gamma$ such that the following holds:
\begin{enumerate}
\item For every $v \in V$,  $N^{v}$ normalizes $H_{+}^{v}$ and $H_{-}^{v}$.
\item For every $v \in V$,  $\langle H_{+}^{v},  H_{-}^{v} \rangle = \Gamma$.
\item For $u,v \in V$,  if $u \rightarrow v$,  then
$H_{\pm}^{v} < \langle H_{\pm}^{u}, N^{u} \rangle .$
\item It holds that
$$\langle  H_{+}^{v}, N^{v}  : v \in V \rangle = \Gamma.$$
\item For every $v \in V$,  the pairs $(\Gamma,  H_{+}^{v}),  (\Gamma,  H_{-}^{v})$ have relative property $(F_{\mathcal{E}})$.
\end{enumerate} 
Then $\Gamma$ has property $(F_{\mathcal{E}})$. 
\end{theorem}

\begin{proof}
To ease the reading, we will denote $u \rightarrow\rightarrow v$ if there is a directed path from $u$ to $v$. 

Fix $S$ to be a finite generating set of $\Gamma$.  Assume towards contradiction that $\Gamma$ does not have property $(F_{\mathcal{E}})$.  By Theorem \ref{S-uniform thm},  it follows that $\mathcal{C}_{\mathcal{E}}^{S-\text{uniform}} \neq \emptyset$.  For every $(\B,  \rho) \in \mathcal{C}_{\mathcal{E}}^{S-\text{uniform}}$ and every $v \in V$,  we denote 
$$D^{v}_{\rho} =  \inf_{\xi^{+}  \in \B^{\rho (H^{v}_+)}, \xi^{-}  \in \B^{\rho (H^{v}_-)}} \Vert \xi^{+}  - \xi^{-}  \Vert$$
and further denote
$$D^{v} = \inf_{(\B,  \rho) \in \mathcal{C}_{\mathcal{E}}^{S-\text{uniform}}} D_{\rho}^{v}.$$

Let $u,v \in V$ such that $u \rightarrow v$.  By Theorem \ref{D is realized thm},  there are $(\B,  \rho) \in \mathcal{C}_{\mathcal{E}}^{S-\text{uniform}}$,  $\xi^{+}  \in \B^{\rho (H^{u}_+)}$ and $\xi^{-}  \in \B^{\rho (H^{u}_-)}$ such that $\Vert \xi^{+} - \xi^{-} \Vert = D^{u}$.  By Theorem \ref{N fixed xi^i thm},  it follows that $\xi^{\pm} \in \B^{\rho (\langle H^{u}_\pm,  N^{u} \rangle)}$.  By our assumptions,  $H_{\pm}^{v} < \langle H_{\pm}^{u}, N^{u} \rangle$ it follows that $\xi^{\pm} \in \B^{\rho (H^{v}_\pm )}$ and thus 
$$D^{u} = \Vert \xi^{+} - \xi^{-} \Vert \geq^{\xi^{\pm} \in \B^{\rho (H_{\pm}^{v})}} D^{v}_\rho \geq D^{v}.$$

By induction,  if $u \rightarrow\rightarrow v$,  then $D^{u} \geq D^{v}$.  By our assumptions,  for every $u,v \in V,  u \neq v$,  it holds that $u \rightarrow\rightarrow v$ and $v \rightarrow\rightarrow u$.  It follows that for every $u,v \in V$,  $D^{u} = D^{v}$.

Fix $v_0 \in V$.  By Theorem \ref{D is realized thm},  there are $(\B,  \rho) \in \mathcal{C}_{\mathcal{E}}^{S-\text{uniform}}$,  $\xi^{+}  \in \B^{\rho (H^{v_0}_+)}$ and $\xi^{-}  \in \B^{\rho (H^{v_0}_-)}$ such that $\Vert \xi^{+} - \xi^{-} \Vert = D^{v_0}$.  By Theorem \ref{N fixed xi^i thm},  it follows that $\xi^{\pm} \in \B^{\rho (\langle H^{v_0}_\pm,  N^{v_0} \rangle)}$.  We will show that for every $v \in V$,  $\xi^{\pm} \in \B^{\rho (\langle H^{v}_\pm,  N^{v} \rangle)}$.  

Let $v \in V$.  We will say that $\dist_{\vec{\mathcal{G}}} (v_0,  v) = n$ where $n$ is the smallest integer such that there are $v_0, v_1,...,v_n =v$ and $v_i  \rightarrow v_{i+1}$ for every $0 \leq i \leq n-1$.  We will show that for every $v \in V$,  $\xi^{\pm} \in \B^{\rho (\langle H^{v}_\pm,  N^{v} \rangle)}$ by induction on $\dist_{\vec{\mathcal{G}}} (v_0,  v)$.  For $n =0$,  we already showed that $\xi^{\pm} \in \B^{\rho (\langle H^{v_0}_\pm,  N^{v_0} \rangle)}$.  Assume that for $n-1$,  if $\dist_{\vec{\mathcal{G}}} (v_0,  u) = n-1$,  then  $\xi^{\pm} \in \B^{\rho (\langle H^{u}_\pm,  N^{u} \rangle)}$.  Let $v \in V$ such that $\dist_{\vec{\mathcal{G}}} (v_0,  v) = n$.  Then there is $u \in V$ such that $\dist_{\vec{\mathcal{G}}} (v_0,  u) = n-1$ and $u \rightarrow v$.  By our assumptions,  $H_{\pm}^{v} < \langle H_{\pm}^{u}, N^{u} \rangle$ it follows that $\xi^{\pm} \in \B^{\rho (H^{v}_\pm )}$ and thus 
$$D^{v} = D^{v_0} = \Vert \xi^{+} - \xi^{-} \Vert \geq^{\xi^{\pm} \in \B^{\rho (H_{\pm}^{v})}} D^{v}_\rho \geq D^{v}.$$
Thus $D^{v}_\rho = D^{v}$ and by Theorem \ref{N fixed xi^i thm},  it follows that $\xi^{\pm} \in \B^{\rho (\langle H^{v}_\pm,  N^{v} \rangle)}$ as needed.   

Therefore we showed that 
$$\xi^{+} \in \bigcap_{v \in V} \B^{\rho (\langle H^{v}_+,  N^{v} \rangle)} = \B^{\rho (\langle H^{v}_\pm,  N^{v} : v \in V \rangle)} = \B^{\rho (\Gamma)}.$$ 
This is a contradiction to the fact that $\rho$ is $S$-uniform and thus $\B^{\rho (\Gamma)} = \emptyset$.  
\end{proof}

As a consequence,  we deduce the following machinery for the class of all uniformly convex Banach spaces:
\begin{theorem}
\label{general uc reduction thm - directed graph}
Let $\Gamma$ be a finitely generated group with a finite Abelianization.  Also,  let $\vec{\mathcal{G}}$ be a directed graph with a vertex set $V$ such that for every $u,v \in V,  u \neq v$,    there are directed paths from $u$ to $v$ and from $v$ to $u$.  Assume for every $v \in V$ there are subgroups $N^{v},  H_{+}^{v},  H_{-}^{v} < \Gamma$ such that the following holds:
\begin{enumerate}
\item For every $v \in V$,  $N^{v}$ normalizes $H_{+}^{v}$ and $H_{-}^{v}$.
\item For every $v \in V$,  $\langle H_{+}^{v},  H_{-}^{v} \rangle = \Gamma$.
\item For $u,v \in V$,  if $u \rightarrow v$,  then
$H_{\pm}^{v} < \langle H_{\pm}^{u}, N^{u} \rangle .$
\item It holds that
$$\langle  H_{+}^{v}, N^{v}  : v \in V \rangle = \Gamma.$$
\item For every $v \in V$,  the pairs $(\Gamma,  H_{+}^{v}),  (\Gamma,  H_{-}^{v})$ have relative property $(F_{\mathcal{E}_{uc}})$.
\end{enumerate} 
Then $\Gamma$ has property $(F_{\mathcal{E}_{uc}})$. 
\end{theorem}

\begin{proof}
We note that cannot apply Theorem \ref{general reduction thm - directed graph} directly for $\mathcal{E} = \mathcal{E}_{uc}$,  because $\mathcal{E}_{uc}$ is not closed under passing to ultraproducts.  However,  $\mathcal{E}_{uc}$ can be written as a union of classes of Banach spaces that are closed under passing to ultraproducts.  
Explicitly,  for $\delta_0 : (0,2] \rightarrow (0,1]$,  we denote $\mathcal{E}_{uc} (\delta_0)$ to be the class of all uniformly convex Banach spaces with moduli of convexity $\geq \delta_0$,  i.e., $\B \in  \mathcal{E}_{uc} (\delta_0)$ if it is uniformly convex with modulus of convexity $\delta : (0,2] \rightarrow (0,1]$ such that for every $0 < \varepsilon \leq 2$,  it holds that $\delta (\varepsilon) \geq \delta_0 (\varepsilon)$.  We note that
$$\mathcal{E}_{uc} = \bigcup_{\delta_0 : (0,2] \rightarrow (0,1]}  \mathcal{E}_{uc} (\delta_0)$$ 
and for every $\delta_0 : (0,2] \rightarrow (0,1)$,  the class $\mathcal{E}_{uc} (\delta_0)$ is closed under passing to ultraproducts. 

By our assumptions,  for every $\delta : (0,2] \rightarrow (0,1)$, the conditions of Theorem \ref{general reduction thm - directed graph} holds for $\mathcal{E} = \mathcal{E}_{uc} (\delta_0)$ and thus for every $\delta : (0,2] \rightarrow (0,1)$,  $\Gamma$ has property $(F_{\mathcal{E}_{uc} (\delta_0)})$.  It readily follows that $\Gamma$ has property $(F_{\mathcal{E}_{uc}})$. 
\end{proof}

Theorem \ref{general reduction thm - directed graph} and Theorem \ref{general uc reduction thm - directed graph} have the following variation in which we use a non-directed graph that appeared in the introduction (Theorem \ref{general reduction thm intro - non-directed graph} above):
\begin{theorem}
\label{general reduction thm - non-directed graph}
Let $\Gamma$ be a finitely generated group with a finite Abelianization and $\mathcal{E}$ a class of uniformly convex Banach spaces such that either is closed under passing to ultraproducts or $\mathcal{E} = \mathcal{E}_{uc}$.  Also, let $\mathcal{G}$ be a (non-directed) connected graph with a vertex set $V$.  Assume that for every $v \in V$ there are subgroups $N^{v},  H_{+}^{v},  H_{-}^{v}< \Gamma$ such that the following holds:
\begin{enumerate}
\item For every $v \in V$,  $N^{v}$ normalizes $H_{+}^{v}$ and $H_{-}^{v}$.
\item For every $v \in V$,  $\langle H_{+}^{v},  H_{-}^{v} \rangle = \Gamma$.
\item If $u,v \in V$ such that $u \sim v$,  then
$H_{\pm}^{u} < \langle H_{\pm}^{v}, N^{v} \rangle ,$
and
$H_{\pm}^{v} < \langle H_{\pm}^{u}, N^{u} \rangle .$
\item It holds that
$$\langle  H_{+}^{v}, N^{v}  : v \in V \rangle = \Gamma.$$
\item For every $v \in V$,  the pairs $(\Gamma,  H_{+}^{v}),  (\Gamma,  H_{-}^{v})$ have relative property $(F_{\mathcal{E}})$.
\end{enumerate} 
Then $\Gamma$ has property $(F_{\mathcal{E}})$. 
\end{theorem}

\begin{proof}
We define $\vec{\mathcal{G}}$ to be the directed graph with the vertex set $V$ such that $u \rightarrow v$ and $v \rightarrow u$ in $\vec{\mathcal{G}}$ if and only if $u \sim v$ in $\mathcal{G}$.  For this graph,  we can apply Theorem \ref{general reduction thm - directed graph} when $\mathcal{E}$ is closed under passing to ultraproducts or Theorem \ref{general uc reduction thm - directed graph} when $\mathcal{E} = \mathcal{E}_{uc}$ and deduce that $\Gamma$  has property $(F_{\mathcal{E}})$. 
\end{proof}

\section{Synthesis of the fixed point property for groups graded by root systems}
\label{Synthesis of the fixed point property for groups graded by root systems sec}

Groups graded by root systems were defined Ershov,  Jaikin-Zapirain and Kassabov \cite{EJZK} as a generalization of Steinberg groups over (classical) root systems.  The aim of the section is to prove that from groups that are (strongly) graded by a root system,  relative property $(F_{\mathcal{E}_{uc}})$ of the root groups implies property $(F_{\mathcal{E}_{uc}})$ of the entire group.

\begin{definition}
Let $\Gamma$ be a group and $\Phi$ be a root system in $E$. We say that $\Gamma$ is \textit{graded by $\Phi$} if there are subgroups $\lbrace K_{\alpha} \rbrace_{\alpha \in \Phi}$ of $\Gamma$ such that the following holds:
\begin{itemize}
\item $\langle K_\alpha : \alpha \in \Phi \rangle = \Gamma$.
\item For every $\alpha, \beta \in \Phi$,  $\beta \notin \mathbb{R}_{<0} \alpha$,  
$$[K_\alpha,  K_\beta] \subseteq \langle X_\gamma : \gamma = a \alpha + b \beta \in \Phi,  a,b \geq 1 \rangle.$$
\end{itemize}
If the above holds,  we say that $\langle K_\alpha : \alpha \in \Phi \rangle$ is a grading of $\Gamma$.  Moreover,  we say that $\Gamma$ is \textit{strongly graded by $\Phi$} (and that $\langle K_\alpha : \alpha \in \Phi \rangle$ is a strong grading of $\Gamma$),  if it is graded by $\Phi$ and for every $\Phi_1 \in \Borel (\Phi )$ and every $\beta \in \Core (\Phi_1)$ it holds that 
$$K_\beta \subseteq \langle K_\gamma : \gamma \in \Phi_1 \setminus (\mathbb{R}_{>0} \beta ) \rangle.$$
\end{definition}

The main example for groups graded by root systems are the Steinberg groups:
\begin{proposition} \cite[Proposition 7.7]{EJZK}
\label{Steinberg groups are strongly graded prop}
Let $\Phi$ be a reduced irreducible classical root system of rank $\geq 2$ and $R$ a commutative ring.  Also, let $\St_{\Phi} (R)$ be the Steinberg group and $\lbrace K_\alpha (R) : \alpha \in \Phi \rbrace$ be the root subgroups of $\St_{\Phi} (R)$.  Then $\St_{\Phi} (R)$ is strongly graded by $\lbrace K_\alpha (R) : \alpha \in \Phi \rbrace$. 
\end{proposition}

Let $\Phi$ be a root system in $E$ and $\Gamma$ be a group graded by $\Phi$.  For a non-empty subset $A \subseteq \Phi$, we denote 
$$K_A = \langle K_\alpha : \alpha \in A \rangle.$$

\begin{lemma}
\label{K alpha normalizes lemma}
Let $\Phi$ be a root system in $E$ and $\Gamma$ a group that is graded by $\Phi$.  Also,  let $\Phi_1 \in \Borel (\Phi)$ and $\alpha \in \partial \Phi_1$.  Then $K_{\Phi \cap \mathbb{R}_{>0} \alpha}$ normalizes $K_{\Phi_1 \setminus \mathbb{R} \alpha}$.
\end{lemma}

\begin{proof}
Let $\alpha ' \in \Phi \cap \mathbb{R}_{>0} \alpha$ and $\beta \in \Phi_1 \setminus \mathbb{R} \alpha$. Then
$$[K_{\alpha '},  K_\beta] \subseteq \langle X_\gamma : \gamma = a \alpha ' + b \beta \in \Phi,  a,b \geq 1 \rangle \subseteq^{\text{Lemma } \ref{facts about root sys lemma} (2)} K_{\Core (\Phi_1)} \subseteq K_{\Phi_1 \setminus \mathbb{R} \alpha}.$$
\end{proof}

\begin{corollary}
\label{Phi1 cap Phi2 normalizes coro}
Let $\Phi$ be a root system in $E$ and $\Gamma$ a group that is graded by $\Phi$.  Also, let $\Phi_1, \Phi_2 \in \Borel (\Phi)$ that are co-maximal.  Then $K_{(\Phi_1 \cup \Phi_2) \setminus (\Phi_1 \cap \Phi_2)}$ normalizes $K_{\Phi_1 \cap \Phi_2}$ and $K_{- (\Phi_1 \cap \Phi_2)}$.
\end{corollary}

\begin{proof}
We assumed that $\Phi_1, \Phi_2$ are co-maximal and thus there is $\alpha \in \partial \Phi_1$ such that $- \alpha \in \partial \Phi_2$ and 
$$(\Phi_1 \cup \Phi_2) \setminus (\Phi_1 \cap \Phi_2) = \Phi \cap \mathbb{R} \alpha.$$

By Lemma \ref{K alpha normalizes lemma},  $K_{\Phi \cap \mathbb{R}_{>0} \alpha}$ normalizes $K_{\Phi_1 \setminus \mathbb{R}_{>0} \alpha} = K_{\Phi_1 \cap \Phi_2}$.  Also,  by Lemma \ref{K alpha normalizes lemma},  $K_{\Phi \cap \mathbb{R}_{<0} \alpha}$,  $K_{\Phi \cap \mathbb{R}_{>0} (-\alpha)} = K_{\Phi \cap \mathbb{R}_{<0} \alpha} $ normalizes $K_{\Phi_2 \setminus \mathbb{R}_{>0} (-\alpha)} = K_{\Phi_1 \cap \Phi_2}$.  Thus, we conclude that $K_{(\Phi_1 \cup \Phi_2) \setminus (\Phi_1 \cap \Phi_2)}$ normalizes $K_{\Phi_1 \cap \Phi_2}$.

It follows that $K_{(\Phi_1 \cup \Phi_2) \setminus (\Phi_1 \cap \Phi_2)}$ also normalizes $K_{- (\Phi_1 \cap \Phi_2)} = K_{(- \Phi_1) \cap (-\Phi_2)} $, since 
$$(\Phi_1 \cup \Phi_2) \setminus (\Phi_1 \cap \Phi_2) = ((-\Phi_1) \cup (-\Phi_2)) \setminus ((-\Phi_1) \cap (-\Phi_2)).$$
\end{proof}

\begin{lemma}
\label{K Phi1 cap Phi2 generates lemma}
Let $\Phi$ be a regular root system in $E$ and $\Gamma$ a group that is strongly graded by $\Phi$.  Also,  let $\Phi_1, \Phi_2 \in \Borel (\Phi)$ that are co-maximal.  Then $\langle K_{\Phi_1 \cap \Phi_2},  K_{- (\Phi_1 \cap \Phi_2)} \rangle = \Gamma$.
\end{lemma}

\begin{proof}
It is sufficient to prove that for every $\beta \in (\Phi_1 \cup \Phi_2) \setminus (\Phi_1 \cap \Phi_2)$, it holds that 
$$K_\beta < \langle K_{\Phi_1 \cap \Phi_2},  K_{- (\Phi_1 \cap \Phi_2)} \rangle.$$

Fix $\beta \in (\Phi_1 \cup \Phi_2) \setminus (\Phi_1 \cap \Phi_2)$.  By Lemma \ref{facts about root sys lemma} (3),  there is $\Phi_3 \in \Borel (\Phi)$ such that $\beta \in \Core (\Phi_3)$. 
By the assumption of strong grading, 
$$K_\beta \subseteq \langle K_\gamma : \gamma \in \Phi_3 \setminus (\mathbb{R}_{>0} \beta ) \rangle <\langle K_{\Phi_1 \cap \Phi_2},  K_{- (\Phi_1 \cap \Phi_2)} \rangle $$
as needed.
\end{proof}

\begin{lemma}
\label{K Phi1 is boundedly generated lemma}
Let $\Phi$ be a root system in $E$ and $\Gamma$ be a group graded by $\Phi$.  For every $\Phi_1 \in \Borel (\Phi)$,  the group $K_{\Phi_1}$ is boundedly generated by $K_{\alpha},  \alpha \in \Phi_1$.
\end{lemma}

\begin{proof}
Fix $\Phi_1 \in \Borel (\Phi)$ and let $f \in \mathfrak{F} (\Phi_1)$ such that $\Phi_f = \Phi_1$.  We index the roots in $\Phi_f$ according to their $f$-value: $\Phi_f = \lbrace \alpha_1,...,\alpha_n \rbrace$ such that for every $1 \leq i <j \leq n$,  it holds that $f( \alpha_i) < f (\alpha_j)$.  We observe that for every $1 \leq i < j \leq n$,  it holds that
$$[K_{\alpha_i},  K_{\alpha_j}] \subseteq \langle X_\gamma : \gamma = a \alpha_i + b \alpha_j \in \Phi,  a,b \geq 1 \rangle < \langle X_\gamma : \gamma \in \Phi,  f(\gamma) > f(\alpha_i) \rangle.$$
Thus,  $K_{\alpha_i}$ normalizes $K_{\lbrace \alpha_{i+1},..., \alpha_{n} \rbrace}$ and it follows that 
$$K_{\lbrace \alpha_i, \alpha_{i+1},..., \alpha_{n} \rbrace} \subseteq K_{\alpha_i} K_{\lbrace \alpha_{i+1},..., \alpha_{n} \rbrace}.$$
By induction, we conclude that
$$K_{\Phi_1} = K_{\lbrace \alpha_{1},..., \alpha_{n} \rbrace} \subseteq K_{\alpha_1} ... K_{\alpha_n}$$
and in particular,  $K_{\Phi_1}$ is boundedly generated by $K_{\alpha_1}, ..., K_{\alpha_n}$.
\end{proof}

Using the lemmas above,  we will prove the following theorem that was mentioned in the introduction:
\begin{theorem}
\label{synthesis thm for groups graded by root systems}
Let $\Phi$ be a regular root system in $E$ and $\Gamma$ a group that is strongly graded by $\Phi$.  Also,  let $\mathcal{E}$ be a class of uniformly convex Banach spaces such that either $\mathcal{E}$ is closed under passing to ultraproducts or $\mathcal{E} = \mathcal{E}_{uc}$.  If for every $\alpha \in \Phi$,  the pair $(\Gamma , K_\alpha)$ has relative property $(F_{\mathcal{E}})$, then $\Gamma$ has property $(F_{\mathcal{E}})$.
\end{theorem}

\begin{proof}

We first note that by Lemma \ref{facts about root sys lemma} (3) and by the assumption that $\Gamma$ is strongly graded by $\Phi$, it follows for every $\alpha \in \Phi$ that $K_\alpha < [\Gamma,  \Gamma]$ and since $\Gamma = \langle K_\alpha : \alpha \in \Phi \rangle$ it follows that $\Gamma$ has a trivial Abelinization.  

The idea of the proof is to use the machinery given in Theorem \ref{general reduction thm - non-directed graph}, i.e.,  to define a connected graph $\mathcal{G}$ such that for every vertex $v$ of $\mathcal{G}$ there are groups $(H_+^{v},  H_-^{v},  N^{v} )$ that fulfil conditions (1)-(5) given in Theorem \ref{general reduction thm - non-directed graph}.  

The graph $\mathcal{G}$ we will use will be the line graph of $\mathcal{G}_{\text{co-max}}$ defined above.   We recall that given a graph $\mathcal{G} '$,  the line graph of $\mathcal{G} '$, denoted $\Line (\mathcal{G} ')$,  is the graph  that represents the adjacencies between the edges of $\mathcal{G} '$.  Explicitly,   
$$V (\Line (\mathcal{G} ')) = \lbrace \lbrace u, v \rbrace : u,v \in V (\mathcal{G} ') \text{ and } u \sim^{\mathcal{G} '} v \rbrace$$ 
and for $\lbrace u,v \rbrace,  \lbrace u',v' \rbrace \in \Line (\mathcal{G} '),  \lbrace u,v \rbrace \neq \lbrace u',v' \rbrace$ it holds that $\lbrace u, v \rbrace \sim^{\Line (\mathcal{G}')} \lbrace u',v' \rbrace$ if and only if $\lbrace u,v \rbrace \cap \lbrace u',v' \rbrace \neq \emptyset$.  It is not hard to see that if $\mathcal{G}'$ is connected, then so in $\Line (\mathcal{G} ')$.

Let $\mathcal{G} = \Line (\mathcal{G}_{\text{co-max}})$, i.e., 
$$V (\mathcal{G} ) = \lbrace \lbrace \Phi_1,  \Phi_2 \rbrace : \Phi_1, \Phi_2 \in \Borel (\Phi) \text{ and } \Phi_1, \Phi_2 \text{ are co-maximal} \rbrace .$$
We note that by Lemma \ref{co-max graph is connected lemma},  the graph $\mathcal{G}_{\text{co-max}}$ is connected and thus $\mathcal{G}$ is connected.  


For $\lbrace \Phi_1,  \Phi_2 \rbrace \in V (\mathcal{G} )$, we define
$$H_+^{\lbrace \Phi_1,  \Phi_2 \rbrace } = K_{\Phi_1 \cap \Phi_2  }, H_-^{\lbrace \Phi_1,  \Phi_2 \rbrace} = K_{- (\Phi_1 \cap \Phi_2)  },  N^{\lbrace \Phi_1,  \Phi_2 \rbrace} = K_{(\Phi_1 \cup \Phi_2  ) \setminus (\Phi_1 \cap \Phi_2 )}.$$

We will complete the proof by checking that these subgroups fulfil the conditions of Theorem \ref{general reduction thm - non-directed graph}:
\begin{enumerate}
\item For every $\lbrace \Phi_1,  \Phi_2 \rbrace \in V (\mathcal{G} )$,  it holds by Corollary \ref{Phi1 cap Phi2 normalizes coro} that $N^{\lbrace \Phi_1,  \Phi_2 \rbrace} $ normalizes $H_+^{\lbrace \Phi_1,  \Phi_2 \rbrace}$ and $H_-^{\lbrace \Phi_1,  \Phi_2 \rbrace}$.
\item  For every $\lbrace \Phi_1,  \Phi_2 \rbrace \in V (\mathcal{G} )$,  it holds by Lemma \ref{K Phi1 cap Phi2 generates lemma} that $\langle H_{+}^{\lbrace \Phi_1,  \Phi_2 \rbrace},  H_{-}^{\lbrace \Phi_1,  \Phi_2 \rbrace} \rangle = \Gamma$.
\item Let $\lbrace \Phi_1,  \Phi_2 \rbrace,   \lbrace \Phi_1 ',   \Phi_2 ' \rbrace \in V (\mathcal{G} )$ such that $\lbrace \Phi_1,  \Phi_2 \rbrace \neq \lbrace \Phi_1 ',   \Phi_2' \rbrace$ and $\lbrace \Phi_1,  \Phi_2 \rbrace \sim^{\mathcal{G}} \lbrace \Phi_1 ',   \Phi_2' \rbrace$.  Then without loss of generality,  $\Phi_1 = \Phi_1 '$.  Thus, 
$$H_{\pm}^{\lbrace \Phi_1,  \Phi_2 \rbrace} = K_{ \pm (\Phi_1 \cap  \Phi_2)} < K_{\pm \Phi_1} < \langle K_{\pm (\Phi_1 \cap \Phi_2 ')},  K_{ (\Phi_1 \cup \Phi_2 ') \setminus (\Phi_1 \cap \Phi_2 ')} \rangle = \langle H_{\pm}^{\lbrace \Phi_1,  \Phi_2 ' \rbrace},  N^{\lbrace \Phi_1,  \Phi_2 ' \rbrace} \rangle,$$
and similarly
$$H_{\pm}^{\lbrace \Phi_1,  \Phi_2 ' \rbrace}  <  \langle H_{\pm}^{\lbrace \Phi_1,  \Phi_2 \rbrace},  N^{\lbrace \Phi_1,  \Phi_2  \rbrace} \rangle .$$
\item Fix some $\lbrace \Phi_1 ',  \Phi_2 '  \rbrace \in V (\mathcal{G} )$.  We note that $\lbrace - \Phi_1 ',  - \Phi_2 ' \rbrace \in V (\mathcal{G} )$ and thus 
\begin{dmath*}
\langle H_{+}^{\lbrace \Phi_1,  \Phi_2 \rbrace},  N^{\lbrace \Phi_1,  \Phi_2 \rbrace} : \lbrace \Phi_1,  \Phi_2 \rbrace \in V (\mathcal{G} )\rangle \supseteq
\langle H_{+}^{\lbrace \Phi_1' ,  \Phi_2 ' \rbrace},  H_{+}^{\lbrace - \Phi_1' ,  - \Phi_2 ' \rbrace},  N^{\lbrace \Phi_1 ',  \Phi_2'  \rbrace} \rangle =
\langle K_{\Phi_1' \cup \Phi_2 '},  K_{ - (\Phi_1' \cap \Phi_2 ' )},  K_{(\Phi_1 ' \cup \Phi_2 ' ) \setminus (\Phi_1 ' \cap \Phi_2 ' )} \rangle = \Gamma.
\end{dmath*}
 
\item We note that $H_+^{\lbrace \Phi_1,  \Phi_2 \rbrace} < K_{\Phi_1}$ and $H_-^{\lbrace \Phi_1,  \Phi_2 \rbrace} < K_{- \Phi_1}$. Thus,  in order to prove that pairs 
$(\Gamma,  H_+^{\lbrace \Phi_1,  \Phi_2 \rbrace})$ and $(\Gamma,  H_-^{\lbrace \Phi_1,  \Phi_2 \rbrace})$ have relative property $(F_{\mathcal{E}})$, it is enough to prove that for every $\Phi_3 \in \Borel (\Phi)$,  the pair $(\Gamma,  K_{\Phi_3})$ has relative property $(F_{\mathcal{E}})$.  Fix $\Phi_3 \in \Borel (\Phi)$.  By Lemma \ref{K Phi1 is boundedly generated lemma},  it holds that $K_{\Phi_3}$ is boundedly generated by $K_\alpha, \alpha \in \Phi_3$.  We assumed that for every $\alpha \in \Phi$,  the pair $(\Gamma,  K_{\alpha})$ has relative property $(F_{\mathcal{E}})$ and thus,  by Lemma \ref{bounded generation lemma},  it follows that the pair $(\Gamma,  K_{\Phi_3})$ has relative property $(F_{\mathcal{E}})$ as needed.    
\end{enumerate}

\end{proof}

As an application,  we derive the following result for Steinberg groups that appeared in the introduction as Theorem \ref{uc synthesis thm for Steinberg - intro}:
\begin{theorem}
\label{uc synthesis thm for Steinberg}
Let $\Phi$ be a classical reduced irreducible root system of rank $\geq 2$,  $R$ a commutative ring,  $\St_{\Phi} (R)$ the Steinberg group and $\lbrace K_\alpha : \alpha \in \Phi \rbrace$ the root subgroups of $\St_{\Phi} (R)$.   

If for every $\alpha \in \Phi$,  the pair $(\St_{\Phi} (R) , K_\alpha)$ has relative property $(F_{\mathcal{E}_{uc}})$, then $\St_{\Phi} (R)$ has property $(F_{\mathcal{E}_{uc}})$.
\end{theorem}

\begin{proof}
We note that $\Phi$ is regular and that by Proposition \ref{Steinberg groups are strongly graded prop},  $\St_{\Phi} (R)$ is strongly graded by the root subgroups.  Thus,  property $(F_{\mathcal{E}_{uc}})$ follows from Theorem \ref{synthesis thm for groups graded by root systems}. 
\end{proof}

\section{Banach fixed point properties for Steinberg groups and elementary Chevalley groups}
\label{Banach fixed point properties for Steinberg groups and elementary Chevalley groups sec}

Combining our relative fixed point property results with our synthesis argument implies our main result:
\begin{theorem}
Let $\Phi$ be a classical reduced irreducible root system of rank $\geq 2$ such that $\Phi \neq C_2$.  For every finitely generated commutative (unital) ring $R$,  the groups $\St_{\Phi} (R)$ and $\EL_{\Phi} (R)$ have property $(F_{\mathcal{E}_{uc}})$.  
\end{theorem}

\begin{proof}
From Theorem \ref{final relative f.p. thm} and Theorem \ref{uc synthesis thm for Steinberg} it follows that $\St_{\Phi} (R)$  has property $(F_{\mathcal{E}_{uc}})$.   Since $\EL_{\Phi} (R)$ is a quotient of $\St_{\Phi} (R)$ and property $(F_{\mathcal{E}_{uc}})$ is preserved under quotients,  it follows that $\EL_{\Phi} (R)$ has property $(F_{\mathcal{E}_{uc}})$.  
\end{proof}

\section{Super-expanders}

\label{super-expanders section}

In \cite{MendelNaor},  Mendel and Naor defined the notion of super-expanders.  We recall their definition here.  Let $\B$ be a Banach space and $\lbrace (V_i,E_i) \rbrace_{i \in \mathbb{N}}$ be a sequence of finite graphs with uniformly bounded degree,  such that $\lim_i \vert V_i \vert = \infty$.  We say that $\lbrace (V_i,E_i) \rbrace_{i \in \mathbb{N}}$ has a \textit{Poincar\'{e} inequality with respect to $\B$} if there are constants $p, \gamma \in (0, \infty)$ such that for every $i \in \mathbb{N}$ and every $f : V_i \rightarrow \B$ we have
$$\frac{1}{\vert V_i \vert^2} \sum_{(u,v) \in V_i \times V_i} \Vert f (u) - f (v) \Vert^p \leq \frac{\gamma}{\vert V_i \vert} \sum_{(x,y) \in E_i} \Vert f (x) - f (y) \Vert^p.$$

The sequence $\lbrace (V_i,E_i) \rbrace_{i \in \mathbb{N}}$ is called a \textit{super-expander family} if it has a Poincar\'{e} inequality with respect to every uniformly convex Banach space. 

Given a Banach space $\B$,  property $(T_{\B})$ is a group property defined in \cite{BFGM} as Banach version of property (T).  We will not recall this definition here, but only recall the following result:
\begin{theorem} \cite[Theorem 1.3]{BFGM}
\label{f.p. implies T thm}
Let $\Gamma$ be a locally compact, second countable group.  For a Banach space $\B$,  it $\Gamma$ has property $(F_\B)$,  then $\Gamma$ has property $(T_\B)$. 
\end{theorem}

For Cayley graphs, the following proposition of Lafforgue gives a relation between Poincar\'{e} inequality of Cayley graphs and Banach property $(T)$:
\begin{proposition}\cite[Proposition 5.2]{Laff1}
\label{prop T imply exp prop}
Let $\Gamma$ be a finitely generated discrete group and let $\Gamma_i$ be a sequence of finite quotients of $\Gamma$ with epimorphism $\phi_i : \Gamma \rightarrow \Gamma_i$  such that $\vert \Gamma_i \vert \rightarrow \infty$.  Also,  let $\B$ be a Banach space.  If $\Gamma$ has Banach property $(T_{\B '})$ for  $\B ' = \ell^2 (\bigcup_i \Gamma_i; \B)$,  then for every fixed finite symmetric generating set $S$ of , the family of Cayley graphs of $\lbrace (\Gamma_i,  \phi_i (S)) \rbrace_{i \in \mathbb{N}}$ has a Poincar\'{e} inequality with respect to $\B$.
\end{proposition}

Combining this result with our Theorem \ref{main thm intro} implies the following result that appeared in the introduction (Theorem \ref{super-exp thm intro}): 
\begin{theorem}
Let $n \geq 3,  m \in \mathbb{N}$ and let $\lbrace R_i \rbrace_{i \in \mathbb{N}}$ be a sequence of finite commutative (unital) rings such that for each $i$,  $R_i$ is generated by $p_0^{(i)} =1,  p_1^{(i)},...,p_m^{(i)} \in R_i$ and $\vert R_i \vert \rightarrow \infty$.  Also, let $\Phi \neq C_2$ be a classical reduced irreducible root system of rank $\geq 2$.  Denote  $\phi_i : \St_{\Phi} (\mathbb{Z} [t_1,...,t_m]) \rightarrow \EL_\Phi (R_i)$ be the epimorphisms induced by the ring epimorphism $\mathbb{Z} [t_1,...,t_m] \rightarrow R_i,  1 \mapsto r_0^{(i)},  t_j \mapsto r_j^{(i)}, \forall 1 \leq j \leq m$. 

For a finite generating set $S$ of  $\St_\Phi (\mathbb{Z} [t_1,...,t_m])$ it holds that the family of Cayley graphs of $\lbrace (\EL_{\Phi} (R_i),  \phi_i (S)) \rbrace_{i \in \mathbb{N}}$ is a super-expander family.
\end{theorem}

\begin{proof}
Let $\B$ be some uniformly convex Banach space.  By \cite[Theorem 3]{Day},  $\ell^2 (\EL_{\Phi} (R_i); \B)$ is a uniformly convex Banach space.  It follows from Theorems \ref{main thm intro}, \ref{f.p. implies T thm} and Proposition \ref{prop T imply exp prop},  that the family of Cayley graphs of $\lbrace (\EL_{\Phi} (R_i),  \phi_i (S)) \rbrace_{i \in \mathbb{N}}$ is a $\B$-expander family.
\end{proof}

\bibliographystyle{alpha}
\bibliography{bib1}

\def\cprime{$'$}
\begin{thebibliography}{{Opp}17a}

\bibitem[BFGM07]{BFGM}
Uri {Bader}, Alex {Furman}, Tsachik {Gelander}, and Nicolas {Monod}.
\newblock {Property \((T)\) and rigidity for actions on Banach spaces}.
\newblock {\em {Acta Math.}}, 198(1):57--105, 2007.

\bibitem[Car89]{Carter}
Roger~W. Carter.
\newblock {\em Simple groups of {L}ie type}.
\newblock Wiley Classics Library. John Wiley \& Sons, Inc., New York, 1989.
\newblock Reprint of the 1972 original, A Wiley-Interscience Publication.

\bibitem[Day41]{Day}
Mahlon~M. Day.
\newblock Some more uniformly convex spaces.
\newblock {\em Bull. Amer. Math. Soc.}, 47:504--507, 1941.

\bibitem[dLdlS23]{LaatSalle4}
Tim de~Laat and Mikael de~la Salle.
\newblock Actions of higher rank groups on uniformly convex {B}anach spaces.
\newblock \url{https://arxiv.org/abs/2303.01405}, 2023.

\bibitem[EJZ10]{ErshovJZ}
Mikhail Ershov and Andrei Jaikin-Zapirain.
\newblock Property ({T}) for noncommutative universal lattices.
\newblock {\em Invent. Math.}, 179(2):303--347, 2010.

\bibitem[EJZK17]{EJZK}
Mikhail Ershov, Andrei Jaikin-Zapirain, and Martin Kassabov.
\newblock Property {$(T)$} for groups graded by root systems.
\newblock {\em Mem. Amer. Math. Soc.}, 249(1186):v+135, 2017.

\bibitem[KN06]{KasNik}
Martin Kassabov and Nikolay Nikolov.
\newblock Universal lattices and property tau.
\newblock {\em Invent. Math.}, 165(1):209--224, 2006.

\bibitem[Laf08]{Laff1}
Vincent Lafforgue.
\newblock Un renforcement de la propri\'{e}t\'{e} ({T}).
\newblock {\em Duke Math. J.}, 143(3):559--602, 2008.

\bibitem[Laf09]{Laff2}
Vincent Lafforgue.
\newblock Propri\'{e}t\'{e} ({T}) renforc\'{e}e banachique et transformation de
  {F}ourier rapide.
\newblock {\em J. Topol. Anal.}, 1(3):191--206, 2009.

\bibitem[Lia14]{Liao}
Benben Liao.
\newblock Strong {B}anach property ({T}) for simple algebraic groups of higher
  rank.
\newblock {\em J. Topol. Anal.}, 6(1):75--105, 2014.

\bibitem[Mim11]{Mimura1}
Masato Mimura.
\newblock Fixed point properties and second bounded cohomology of universal
  lattices on {B}anach spaces.
\newblock {\em J. Reine Angew. Math.}, 653:115--134, 2011.

\bibitem[Mim16]{Mimura4}
Masato Mimura.
\newblock Strong algebraization of fixed point properties.
\newblock \url{https://arxiv.org/abs/1505.06728}, 2016.

\bibitem[MN14]{MendelNaor}
Manor Mendel and Assaf Naor.
\newblock Nonlinear spectral calculus and super-expanders.
\newblock {\em Publ. Math. Inst. Hautes \'{E}tudes Sci.}, 119:1--95, 2014.

\bibitem[{Opp}17a]{OppRobust}
Izhar {Oppenheim}.
\newblock {Averaged projections, angles between groups and strengthening of
  Banach property (T)}.
\newblock {\em {Math. Ann.}}, 367(1-2):623--666, 2017.

\bibitem[Opp17b]{OppVanBan}
Izhar Oppenheim.
\newblock Vanishing of cohomology with coefficients in representations on
  {B}anach spaces of groups acting on buildings.
\newblock {\em Comment. Math. Helv.}, 92(2):389--428, 2017.

\bibitem[Opp18]{OppAngle}
Izhar Oppenheim.
\newblock Angle criteria for uniform convergence of averaged projections and
  cyclic or random products of projections.
\newblock {\em Israel J. Math.}, 223(1):343--362, 2018.

\bibitem[Opp22]{Opp-SLZ}
Izhar Oppenheim.
\newblock Banach property ({T}) for $\rm {SL}_3 (\mathbb{Z})$ and its
  applications.
\newblock \url{https://arxiv.org/abs/2207.04407}, 2022.

\bibitem[Sha99]{Shalom1}
Yehuda Shalom.
\newblock Bounded generation and {K}azhdan's property ({T}).
\newblock {\em Inst. Hautes \'{E}tudes Sci. Publ. Math.}, (90):145--168 (2001),
  1999.

\bibitem[Sha06]{Shalom2}
Yehuda Shalom.
\newblock The algebraization of {K}azhdan's property ({T}).
\newblock In {\em International {C}ongress of {M}athematicians. {V}ol. {II}},
  pages 1283--1310. Eur. Math. Soc., Z\"{u}rich, 2006.

\bibitem[Sta09]{Stalder}
Yves Stalder.
\newblock Fixed point properties in the space of marked groups.
\newblock In {\em Limits of graphs in group theory and computer science}, pages
  171--182. EPFL Press, Lausanne, 2009.

\bibitem[Vas06]{Vaser}
Leonid Vaserstein.
\newblock Bounded reduction of invertible matrices over polynomial rings by
  addition operations.
\newblock \url{http://www.personal.psu.edu/lxv1/pm2.pdf}, 2006.

\end{thebibliography}
\end{document}